\documentclass[a4paper, 11pt]{article}

\usepackage{amsmath,amssymb,amsthm}
\usepackage{dsfont}
\usepackage{graphicx}
\usepackage{subcaption}
\usepackage[british]{babel}
\usepackage[utf8]{inputenc}
\usepackage{xcolor}
\usepackage[colorlinks=true,linkcolor=blue,citecolor=red]{hyperref}
\usepackage{authblk}
\usepackage{blindtext}
\usepackage{epstopdf}
\usepackage{cite}
\usepackage{bm}
\usepackage{comment}
\usepackage{booktabs} % for tables
\usepackage{algorithm}
\usepackage{tikz}
\usetikzlibrary{calc,trees,positioning,arrows,chains,shapes.geometric,%
  decorations.pathreplacing,decorations.pathmorphing,shapes,%
  matrix,shapes.symbols}
\tikzset{
>=stealth',
 punktchain/.style={
  rectangle, 
   fill=cyan!40,
  draw=black, very thick,
  text width=12em, 
  minimum height=2em, 
  text centered, 
  on chain},
 line/.style={draw, thick, <-},
 element/.style={
  tape,
  top color=white,
  bottom color=blue!50!black!60!,
  minimum width=8em,
  draw=blue!40!black!90, very thick,
  text width=10em, 
  minimum height=2.5em, 
  text centered, 
  on chain},
 every join/.style={->, thick,shorten >=1pt},
 decoration={brace},
 tuborg/.style={decorate},
 tubnode/.style={midway, right=2pt},
}

\newtheorem{remark}{Remark}
\newtheorem{proposition}{Proposition}
\newtheorem{theorem}{Theorem}
\newtheorem{lemma}{Lemma}
\usepackage{mathtools}

\newcommand{\R}{\mathbb{R}}
\newcommand{\E}{\mathbb{E}}
\newcommand{\be}{\begin{equation}}
\newcommand{\ee}{\end{equation}}
\newcommand{\z}{\mathbf{z}}
\newcommand{\nt}{n_{\textrm{TOT}}}

\renewcommand{\epsilon}{\varepsilon}

\allowdisplaybreaks

\begin{document}
\title{Particle simulation methods for the Landau-Fokker-Planck equation with uncertain data}

\author[1]{Andrea Medaglia\thanks{\tt andrea.medaglia02@universitadipavia.it }}
\author[2]{Lorenzo Pareschi\thanks{\tt lorenzo.pareschi@unife.it}}
\author[1]{Mattia Zanella\thanks{\tt mattia.zanella@unipv.it}}
\affil[1]{Department of Mathematics ``F. Casorati", University of Pavia, Italy}
\affil[2]{Department of Mathematics and Computer Science, University of Ferrara, Italy}

\date{}

\maketitle

\abstract{The design of particle simulation methods for collisional plasma physics has always represented a challenge due to the unbounded total collisional cross section, which prevents a natural extension of the classical Direct Simulation Monte Carlo (DSMC) method devised for the Boltzmann equation. One way to overcome this problem is to consider the design of Monte Carlo algorithms that are robust in the so-called grazing collision limit. In the first part of this manuscript, we will focus on the construction of collision algorithms for the Landau-Fokker-Planck equation based on the grazing collision asymptotics and which avoids the use of iterative solvers. Subsequently, we discuss problems involving uncertainties and show how to develop a stochastic Galerkin projection of the particle dynamics which permits to recover spectral accuracy for smooth solutions in the random space. Several classical numerical tests are reported to validate the present approach. 
}
\\[+.2cm]

\tableofcontents

\section{Introduction}
The Landau-Fokker-Planck kinetic equation, also referred to as the Landau equation, holds significant importance as a fundamental tool for understanding the complex dynamics of charged particles within a collisional plasma. The equation plays a crucial role in elucidating the behaviour of particles that interact over long distances, primarily influenced by Coulomb forces. These long-range interactions give rise to small-angle collisions between the particles, leading to a breakdown in the finiteness of the classical Boltzmann collision operator^^>\cite{Landau1936,rosenbluth57,degond1992,bobylev_CMP2013}.

Due to its profound implications in the development of fusion reactors, the formulation of this model has had a far-reaching impact on both the theoretical and applied scientific communities. Researchers and engineers in the field have extensively utilized this formalized model to gain insights into the behaviour of charged particles in plasmas, paving the way for advancements in fusion reactor design and other related applications.

The construction of numerical methods for the Landau-Fokker-Planck equation has to deal with several challenges, among which are the high dimensionality of the problem, the structural properties, and the presence of multiple scales. As in most kinetic equations, the construction of numerical methods for plasma physics can be subdivided into two main categories, the first based on direct discretizations of the collisional PDEs, like spectral methods, see e.g.^^>\cite{filbet2002,pareschi2000,pareschi2003,crouseilles2010,DP14,DLPY15}, and the second based on the approximation of the underlying particles' dynamics. In this latter direction we mention particle-in-cell methods^^>\cite{birdsall85}, direct simulation Monte Carlo methods^^>\cite{bobylev2000,caflisch2008,caflisch2010,bobylev2013,RRCD,Trazzi05} and deterministic particle methods ^^>\cite{carrillo2020,CiCP-31-997}. It is worth to remark that Monte Carlo approaches based on the groundbreaking contribution of Bobylev and Nanbu^^>\cite{Nanbu98,bobylev2000} are consistent with the Landau equation thanks to their link with the Boltzmann equation when collisions become grazing. 

To enhance prediction accuracy, it is crucial to understand how sensitive physical models are to potential uncertainties in the constitutive parameters that characterize the system behaviour and influences of the different scales. Consequently, addressing the issue of quantifying these uncertainties becomes a major concern, both analytically and computationally. The development of numerical techniques must consider the added complexity arising from the increased dimensionality of the system due to the presence of random inputs in the model. In this direction, several research efforts focused on the design of efficient numerical techniques for uncertainty quantification (UQ), see^^>\cite{hu2015,xiu2002,despres2016,dimarco2019,pettersson15,poette2009,Pareschi2021,dimarco2020} and the references therein. We mention also recent results on UQ in plasma physics obtained in^^>\cite{chung2020,din2019,hu2018,xiao2021}.

Among the most popular UQ techniques, stochastic Galerkin (sG) methods based on generalized polynomial chaos expansions have shown the ability to achieve spectral accuracy in the random space for smooth solutions^^>\cite{xiu2010}. However, in contrast to stochastic sampling methods, these methods are highly intrusive and require to design new algorithms. Furthermore, a direct application of sG methods may lead to the loss of important structural properties, like conservation of physical quantities, nonnegativity of the solution, and hyperbolicity close to fluid regimes, see e.g.^^>\cite{dai2022,despres2013}. To overcome this issue, a different approach based on projecting the particle dynamics through generalized polynomial chaos was proposed in^^>\cite{carrillo19,carrillo19vietnam,medaglia2022,Pareschi2020}. Recently, the methodology has been extended to plasma simulations with BGK collisions in^^>\cite{medaglia2022JCP}. 

In this article, our focus lies on the more realistic scenario where plasma collisions are described by the Landau equation, taking into account uncertain quantities. Specifically, we concentrate on the space homogeneous case, where the collision algorithm follows the approach outlined in^^>\cite{bobylev2000,caflisch2010} whereas the uncertain velocity changes are approximated using an sG particle projection. \textcolor{black}{In this direction, we will construct a sG particle approximation which is consistent with the Landau equation in the grazing regime.}
To ensure spectral accuracy in the random space, our particle method is designed to leverage the general structure of collision kernels discussed in^^>\cite{bobylev2000}. To this aim, we introduce a \textcolor{black}{new regularised kernel, which guarantees that the scaled Boltzmann equation for Coulomb collisions is capable to efficiently approximate the Landau equation and that} offers several advantages, including the avoidance of iterative techniques commonly employed in other particle simulation methods^^>\cite{caflisch2010}. Being based on particle reconstruction, the method preserves the nonengativity of the solution and the main physical properties. \textcolor{black}{Furthermore, the resulting sG particle solver leads to a consistent approximation of the uncertain quantities linked to the solution of the Landau equation. }
To validate our method, we conduct several tests, these include classical benchmarks such as BKW, Trubnikov's solution for Coulombian particles, and the bump-on-tail problem. 

The rest of the manuscript is organized as follows. In Section \ref{sec:2} we recall the basic ideas in the design of collision algorithm for the Landau-Fokker-Planck equation inspired by the corresponding grazing limit of the Boltzmann equation. Section \ref{sec:3} is devoted to present the details of the DSMC approach with a regularized kernel and the corresponding sG projection. A suite of numerical examples is then reported in Section \ref{sec:4} which confirms the efficiency and the accuracy of the present approach. The last section collects some concluding remarks and future developments.

\section{Foundations of collision algorithms for the Landau equation}
\label{sec:2}
In this section, following^^>\cite{bobylev2000,caflisch2010}, we recall some results concerning the derivation of direct simulation collision algorithm for the Landau-Fokker-Planck equation based on the grazing collision limit of the Boltzmann equation. To this aim, we will make use of a first order approximation of the space homogeneous Boltzmann equation and devote particular attention to the Maxwellian and the Coulombian cases. 
%This derivation makes it possible to define a DSMC particle scheme, in this direction, we mention also the work^^>\cite{caflisch2010}.

%\subsection{} \label{sec:model}
We consider the space homogeneous Boltzmann equation^^>\cite{CIP94}
\begin{equation} \label{eq:boltzmann}
\dfrac{\partial f(v,t)}{\partial t} = Q(f,f)(v,t),
\end{equation}
describing the time evolution of the one-particle distribution function $f(v,t)$ with velocity $v\in \R^3$, at the time $t>0$. 
%The random vector $\z=(z_1,\dots,z_d)\in\R^{d_\z}$ collects all the uncertainties of the system and we will suppose to know its distribution $p(\z)$, such that
%\[
%\mathrm{Prob}(\z\in \Omega) = \int_\Omega p(\z) d\z,
%\]
%for any $\Omega\in\R^{d_\z}$.
%
The collision term $Q(f,f)$ on the right-hand side of \eqref{eq:boltzmann} is a bilinear operator describing the binary interactions between particles and it is given by
\begin{equation}\label{eq:Q}
Q(f,f)(v,t) = \int_{\R^3} \int_{S^2} B\left( |q|, \frac{q\cdot n}{|q|} \right) \left( f(v',t)f(v_*',t)-f(v,t)f(v_*,t) \right) \, dn \, dv_*,
\end{equation}
where $q=v-v_*$ is the relative velocity and $n\in S^2$ is the unit vector normal to the unit sphere $S^2$ in $\R^3$. The binary interactions rules characterizing collisions between particles are
\[
\begin{split}
	v' &=\frac{1}{2}(v+v_*+|q|n) \\
	v'_*&=\frac{1}{2}(v+v_*-|q|n)
\end{split}
\]
and the collisional kernel $B(|q|,\cos\theta)$ reads
\[
B\left( |q|, \cos\theta \right) = |q| \, \sigma(|q|, \theta), \qquad \textrm{with} \qquad (0\leq \theta \leq \pi),
\]
where $\cos\theta=q\cdot n/|q|$ and $\sigma(|q|, \theta)$ is the so-called differential collision cross section at the scattering angle $\theta$, corresponding to the number of particles scattered per unit of incident flux, per unit of solid angle, in the unit time. We introduce also the total scattering cross section 
\[
\sigma_{tot}(|q|)=2\pi \, \int_{0}^{\pi} \sigma(|q|, \theta) \, \sin\theta \, d\theta,
\]
and the momentum-transfer (or -transport) scattering cross section defined as
\[
\sigma_{tr}(|q|)=2\pi \, \int_{0}^{\pi} \sigma(|q|, \theta)\,(1-\cos\theta) \, \sin\theta \, d\theta,
\]
that describes the average momentum transferred in the collisions. The differential cross section $\sigma(|q|, \theta)$ takes different forms depending on the interactions considered. Amongst the most relevant cases we mention the variable hard sphere (VHS) model and the Coulomb interactions. The VHS model is such that
\be\label{eq:scattering}
\sigma(q,\theta)=C_\gamma |q|^{\gamma-1}, \qquad \textrm{so that} \qquad B(|q|,\theta)=C_\gamma |q|^{\gamma},
\ee
where $C_\gamma>0$ is a constant, with $\gamma=0$ for Maxwell molecules and $\gamma=1$ for hard sphere model, which gives
\[
\sigma_{tot}(|q|)=\sigma_{tr}(|q|)=4\pi C_\gamma |q|^{\gamma-1}. 
\] 
On the other hand, particles subject to Coulomb forces are characterized, according to the Rutherford formula, by the scattering
\be \label{eq:coulomb}
\sigma(q,\theta)= \frac{b_0^2}{4\sin^4 (\theta/2)}, \qquad \textrm{with} \qquad b_0=\frac{e^2}{4\pi\epsilon_0 m_r |q|^2},
\ee
where $e$ is the charge of the particles, $\epsilon_0$ the vacuum permittivity and $m_r$ is the reduce mass, corresponding to $m/2$ for particles of the same species. In this case, it is necessary to introduce a cut-off, since for $\theta\to0$ the cross section is singular. Following the usual approximations justified by the shielding effect, we have
\[
\sigma_{tot}(|q|) = \pi \lambda^2_{d}, \qquad \textrm{where} \qquad \lambda^2_d = \frac{\epsilon_0 k T}{n e^2}
\]
is the Debye length and
\[
\sigma_{tr}(|q|) = 4\pi b^2_0 \log\Lambda, \qquad \textrm{where} \qquad \Lambda = \frac{1}{\sin(\theta^{min}/2)}.
\]

It is well-known that in the grazing collision limit, from equation \eqref{eq:boltzmann} with collisional operator \eqref{eq:Q}, we can recover the Landau-Fokker-Planck equation^^>\cite{Landau1936}
\begin{equation}\label{eq:LFP}
\begin{split}
	\dfrac{\partial f(v,t)}{\partial t} & = \, Q^L(f,f)(v,t) \\
	& = \frac{1}{8}\nabla_v \cdot \left[ \int_{\R^3} \Phi(v-v_*) \left( f(v_*)\nabla_vf(v) - f(v)\nabla_{v_*}f(v_*)\right) dv_* \right]
\end{split}
\end{equation}
where $\Phi$ is a $3 \times 3$ nonnegative symmetric matrix and its usual form is given by 
\[
\Phi(q) = |q|^{\gamma+2}S(q), \qquad S(q) = I - \dfrac{q \otimes q}{|q|^2}, 
\]
with $I$ identity matrix, and where $\gamma=0$ is the Maxwellian case and $\gamma = -3$ is the Coulombian case. In the following, we will discuss in more detail this aspect and derive first order approximation formulas which are suitable for the design of DSMC type algorithms. 
%Note that for Maxwellian particles the cross sections do not depend on the relative velocity, i.e., they are constant.

%The large time behaviour of the collision operator is the Maxwellian distibution
%\[
%\mathcal{M}(v,\z) = \dfrac{1}{\left(2\pi T(x,\z)\right)^{\frac{3}{2}}} \exp\left( - \dfrac{ (v - U(\z))^2}{2 T(\z)}\right),
%\]
%where we introduced the momentum and kinetic temperature as
%\[
%U(\z) = \int_{\R^3} v f(v,t,\z) dv, 
%\]
%\[
%T(\z)=\dfrac{1}{3}\int_{\R^3} |v-U(\z)|^2 f(v,t,\z) dv.
%\]

\subsection{First order approximation of the Boltzmann equation} \label{sec:2.1}
From a direct inspection of the collisional kernel in the presence of Coulomb forces, it is evident that when collisions become grazing, the interaction rate becomes infinite. To overcome this problem, which prevents the direct application of the DSMC algorithm, following^^>\cite{bobylev2000} we introduce a suitable first order approximation of the Boltzmann collision operator that permits to avoid such restrictions. In the sequel, we omit the explicit dependence on the time variable for compactness of notation.

First, we rewrite equation \eqref{eq:boltzmann} as
\[
\frac{\partial f}{\partial t} = \int_{\R^3} J F(U,q) dv_*,
\]
where $U=(v+v_*)/2$ is the velocity of the centre-of-mass, and
\[
F(U,q) = f(U+q/2) f(U-q/2) = f(v)f(v_*).
\]
The operator $J$ is defined such that its action on the angular variable $\omega=q/|q|$ is
\[
J F(U,|q|\omega) = \int_{S^2} B(|q|, \omega\cdot n) \left( F(U,|q|n) - F(U,|q|\omega) \right).
\]
Now we expand the operator $J$ at first order in \textcolor{black}{$\epsilon>0$} as
\[
J\approx\frac{1}{\epsilon} \left( \exp(\epsilon J) - \hat{I} \right),
\]
where $\hat{I}$ is the identity operator and $\epsilon$ is a small parameter. \textcolor{black}{We first observe that in the limit $\epsilon\to 0$ we recover the definition of $J$ and thus the original Boltzmann equation. Within this approximation, the collision term reads}
\begin{equation}\label{eq:boltzmann_epsi}
Q(f,f) \approx  \frac{1}{\epsilon} \int_{\R^3} \left( \exp(\epsilon J) - \hat{I} \right) F(U,q) dv_* = \frac{1}{\epsilon} \left( Q^{+}_{\epsilon}(f,f) - \rho f \right),
\end{equation}
where 
\[
Q^{+}_{\epsilon}(f,f) = \int_{\R^3} \exp(\epsilon J) f(v)f(v_*) dv_*
\]
\textcolor{black}{is the gain operator. To explicitly construct the operator $\exp(\epsilon J)$, it is sufficient to study the initial value problem for any test function $\psi(\omega,\epsilon)$
\[
\frac{\partial \psi}{\partial \epsilon}=J\psi,
\]	
for arbitrary initial conditions $\psi(\omega,0)=\psi_0(\omega),\,\omega\in S^2$ with $\psi_0(\omega) \in L^2(S^2)$. Following \cite{bobylev2000}, we can expand $\psi_0(\omega)$ into its spherical harmonics
\[
\psi_0(\omega) = \sum_{l=0}^{\infty} \sum_{m=-l}^{l} \psi^0_{lm} Y_{lm}(\omega)
\]
with
\[
\psi^0_{lm} = \int_{S^2} \psi_0(\omega) Y^*_{lm}(\omega) d\omega
\]
and $Y^*_{lm}(\omega)$ the complex conjugate of the spherical function $Y_{lm}(\omega)$. The action of $J$ on $Y_{lm}$ reads
\[
JY_{lm} = - \lambda_{l} Y_{lm}, \qquad \textrm{with} \qquad \lambda_{l} = 2 \pi \int_{-1}^1 B(\mu) (1-P_l(\mu)) d\mu
\]
with $P_l(\mu)$ Legendre polynomials of order $l$ and $\mu=\omega\cdot n=\cos\theta$. Therefore, the solution of the equation furnishes the action of the operator in integral form, since we have 
\[
\exp(\epsilon J)\psi_0(\omega) = \psi(\omega,\epsilon)= \int_{S^2} B_{\epsilon}(\mu, |q|) \psi(n) dn,
\]
}
with 
\[
B_{\epsilon}(\mu, |q|) = \sum_{l=0}^{+\infty} \frac{2l+1}{4\pi} \exp(-\lambda_l(|q|)\epsilon) P_l(\mu)
\]
the Green function, and 
\[
\lambda_l(|q|) = 2 \pi \int_{-1}^{1} \textcolor{black}{B_\epsilon(\mu, |q|)}(1-P_l(\mu)) d\mu.
\]
\textcolor{black}{With the above representation of the operator $\exp(\epsilon J)$}, the approximated gain operator explicitly reads
\[
Q^{+}_{\epsilon}(f,f)=\int_{\R^3}\int_{S^2} B_{\epsilon}(\omega\cdot n, |q|) f(v') f(v'_*) dv_*.
\]

\textcolor{black}{Since we are interested in the behaviour of the equation in the grazing collision limit, i.e. when the differential scattering cross section $\sigma(|q|,\theta)$ is concentrated at a small angle $\theta\approx 0$, we assume that $\mu$ in $B_{\epsilon}(\mu, |q|)$ is concentrated near the value $\mu=1$. In this regime, it is well known that the Boltzmann equation converges to the Landau equation.}
%Let us now assume that the differential scattering cross section $\sigma(|q|,\theta)$ is concentrated at a small angle $\theta\approx 0$, that is the grazing collision regime, thus $\mu$ in $B_{\epsilon}(\mu, |q|)$ is concentrated near the value $\mu=1$. 
We have
\[
\lambda_l(|q|) \simeq 2\pi\int_{-1}^{+1} B_{\epsilon}(\mu, |q|) \left(1-P_l(\mu)+(1-\mu)P'_l(1)\right)d\mu=\pi l(l+1) \int_{-1}^{+1} B_{\epsilon}(\mu, |q|) (1-\mu) d\mu,
\]
where $P'_l(1)=l(l+1)/2$. \textcolor{black}{With the above formula, the Green function $B_{\epsilon}(\mu, |q|)$ can be further approximate as $D(\mu, |q|)$}
\[
B_{\epsilon}(\mu, |q|) \simeq D(\mu, |q|) = \sum_{l=0}^{+\infty} \frac{2l+1}{4\pi} P_l(\mu) \exp\left(-\frac{l(l+1)}{2}|q|\sigma_{m}(|q|)\epsilon\right).
\]
Finally, one can show that equation \eqref{eq:boltzmann_epsi} with the above collisional kernel approximates the Landau-Fokker-Planck equation for small values of $\varepsilon$. \textcolor{black}{We define $\tau$ such that}
\be \label{eq:oneovertau}
\textcolor{black}{\frac{1}{\tau} = \rho |q| \sigma_{tr}(|q|)}
\ee
and the approximated Green function with the cut-off of the scattering angle as
\begin{equation}\label{eq:D}
	D\left(\mu,\tau^{(\gamma)}_0\right)=\sum_{l=0}^{+\infty} \frac{2l+1}{4\pi} P_l(\mu) \exp\left(-l(l+1)\tau^{(\gamma)}_0\right), \qquad \tau^{(\gamma)}_0=\frac{\epsilon}{2\rho\tau}=\frac{\epsilon}{2}|q| \sigma_{tr}(|q|),
\end{equation}
\textcolor{black}{where the superscript $(\gamma)$ into the definition of $\tau^{(\gamma)}_0$ refers to the type of particle interaction, which changes the form of $\sigma_{tr}(|q|)$.}
We have the first order approximation 
\begin{equation}\label{eq:boltzmann_approx1}
\frac{\partial f}{\partial t} = \frac{1}{\epsilon} \left[\int_{\R^3}\int_{S^2}D\left(\mu,\tau^{(\gamma)}_0\right) f(v')f(v'_*)dndv_* \, - \, \rho f(v)\right], 
\end{equation}
where the first term on the  right-hand side of the previous equation plays the role of the gain operator with the approximated Green function, \textcolor{black}{and the second one is the loss term.  We remark, however, that at this point no assumptions were made on the type of collision. Indeed, the derivation of the approximated Boltzmann equation \eqref{eq:boltzmann_approx1} does not depend on the explicit form of $\tau^{(\gamma)}_0$. Thus, it holds for both the VHS and Coulomb models introduced in Section \ref{sec:2}.}

\textcolor{black}{In the numerical part, Section \ref{sec:4}, we will consider the Maxwell ($\gamma=0$) and Coulomb ($\gamma=-3$) cases. In these scenarios, the explicit form of $\tau^{(\gamma)}_0$ in a single species plasma model reads as follows: if $\gamma=0$, we have
\[
\tau^{(0)}_0 = 4\pi C_0 \frac{\epsilon}{2},
\]
that is independent from the relative velocity, if $\gamma=-3$, we recover
\be \label{eq:tau0_c}
\tau^{(-3)}_0 = 4\pi \left( \frac{e^2}{4\pi\epsilon_0 m_r} \right)^2 \frac{ \log \Lambda}{|q|^3} \frac{\epsilon}{2}.
\ee
With a slight abuse of notation, we will refer to $\tau^{(\gamma)}_0$ simply as $\tau_0$, whenever the differences in the type of interaction do not change the general construction.}

\subsection{Approximated scattering kernels} \label{sec:2.2}
In order to solve numerically equation \eqref{eq:boltzmann_approx1}, we need to sample the probability density $D(\mu,\tau_0)$ defined in \eqref{eq:D}. To this end, a random sample of the quantity $\mu(\epsilon)$ in $[-1,1]$, at each time step, is required. It is easily observed that this can be very heavy from a numerical perspective. 

However, as pointed out in^^>\cite{bobylev2000}, we may consider simpler collisional kernels $D_*(\mu,\tau_0)$ with the following properties
\begin{enumerate}
	\item \label{en:D1} $D_*(\mu,\tau_0)\geq0$, and $2\pi\displaystyle\int_{-1}^{+1}D_*(\mu,\tau_0)d\mu=1$
	\item \label{en:D2} $\lim_{\tau_0\to0}D_*(\mu,\tau_0)=\dfrac{1}{2\pi}\delta(1-\mu)$ 
	\item \label{en:D3} $\lim_{\tau_0\to0}\frac{2\pi}{\tau_0}\displaystyle\int_{-1}^{+1}(D_*(\mu,\tau_0)-D(\mu,\tau_0))P_l(\mu) \textcolor{black}{d\mu}=0$ \textcolor{black}{for any $l=1,2,\dots$}. 
\end{enumerate}
A general form of such a kernel is furnished by the following result.
\begin{lemma}\label{lemma:1}
Conditions \ref{en:D1}-\ref{en:D3} are satisfied for any function of the type
\[
D_*(\mu,\tau_0) = \psi\left(\frac{1-\mu}{2\tau_0}\right) \left[ 4\pi\tau_0 \int_{0}^{1/\tau_0}\psi(x)dx \right]^{-1},
\]
where $\psi(x)\geq0$ for any $x>0$ and
\begin{itemize}
	\item $\displaystyle\int_{0}^{+\infty}\psi(x)dx=\int_{0}^{+\infty}x\psi(x)dx<\infty$;
	\item $\lim_{\tau_0\to0}\tau_0\displaystyle\int_{1/\tau_0}^{+\infty}x^n \psi(x) dx = 0$ for any $n=2,3,\dots$.
\end{itemize}
\end{lemma}
\begin{proof}
For the proof, we refer to^^>\cite{bobylev2000}.
\end{proof}

In the following, we consider three different kernels. The first one has been proposed in^^>\cite{caflisch2010,nanbu1997} and is commonly used in plasma physics simulation, and it reads
\begin{equation}
\label{eq:D1}
	D^{(1)}_*(\mu,\tau_0) = \frac{A}{4\pi \sinh A} \exp(\mu A)
\end{equation} 
with
\begin{equation} \label{eq:nonlineq}
\coth A - \dfrac{1}{A} = e^{-2\tau_0}.
\end{equation}
As a consequence, to sample $D^{(1)}_*(\mu,\tau_0)$ in the spherical coordinate system $(\theta, \phi)$, we need to solve at every time step the nonlinear equation \eqref{eq:nonlineq} and then compute the scattering angle $\cos\theta$ as
\begin{equation} \label{eq:costheta1}
\cos\theta = \dfrac{1}{A} \ln\left(\exp^{-A}+2r_1\sinh A\right),
\end{equation}
and $\phi$ as 
\begin{equation} \label{eq:phi1}
\phi = 2\pi r_2,
\end{equation}
where $r_1, r_2$ are two uniform random numbers in $(0,1)$.

The second kernel has been proposed in^^>\cite{bobylev2000} as a simplification of $D^{(1)}_*(\mu,\tau_0)$ satisfying Lemma \ref{lemma:1}
\begin{equation}
\label{eq:D2}
	D^{(2)}_*(\mu,\tau_0) = \frac{1}{2\pi} \delta\left(\mu - \nu(\tau_0)\right),
\end{equation}
where 
\be \label{eq:nu}
	\nu(\tau_0) =
	\begin{cases}
		1-2\tau_0 & \text{if } 0 \leq \tau_0 \leq 1 \\
		-1     & \text{if } \tau_0 > 1.
	\end{cases}
\ee
Sampling $D^{(2)}_*(\mu,\tau_0)$ corresponds, in the spherical coordinate, to compute $\phi$ as in \eqref{eq:phi1} and to fix
\begin{equation} \label{eq:costheta2}
\cos \theta = \nu(\tau_0).
\end{equation}
Note that in this case we need only one random number, i.e. $r_2$, and no iterations are required.

A third approximation consistent with Lemma \ref{lemma:1} will be considered in the following and it consists in a regularization of the previous kernel $D^{(2)}_*(\mu,\tau_0)$ defined as
\begin{equation}
\label{eq:D3}
	D^{(3)}_*(\mu,\tau_0) = \frac{1}{2\pi} \delta\left(\mu - \tilde{\nu}(\tau_0)\right),
\end{equation}
with 
\be \label{eq:nutilde}
\tilde{\nu}(\tau_0) = 1 - 2 \tanh \tau_0.
\ee
We highlight that \textcolor{black}{$\cos \theta$} is obtained by fixing 
\begin{equation} \label{eq:costheta3}
	\cos \theta = \tilde{\nu}(\tau_0)
\end{equation}
whereas $\phi$ is computed as in \eqref{eq:phi1}. 
\textcolor{black}{
\begin{proposition}
The kernel $D^{(3)}_*(\mu,\tau_0)$ defined in \eqref{eq:D3} satisfies conditions \ref{en:D1}-\ref{en:D2}-\ref{en:D3}.
\end{proposition}
\begin{proof}
Condition \ref{en:D1} is satisfied because $\tilde{\nu}(\tau_0) = 1 - 2 \tanh \tau_0 \geq 0$ for every $\tau_0\geq 0$, and consequently 
\[
D^{(3)}_*(\mu,\tau_0) = \frac{1}{2\pi} \delta\left(\mu - \tilde{\nu}(\tau_0)\right) \geq 0, \qquad \textrm{for every} \qquad \mu\in[-1,1], \, \tau_0\geq 0.
\]
Moreover, by the definition of Dirac delta we have the normalization
\[
2\pi\displaystyle\int_{-1}^{+1}D^{(3)}_*(\mu,\tau_0)d\mu=2\pi\displaystyle\int_{-1}^{+1}\frac{1}{2\pi} \delta\left(\mu - \tilde{\nu}(\tau_0)\right)d\mu=1.
\]
Condition \ref{en:D2} follows from the fact that $\tanh\tau_0\to0$ for $\tau_0\to0$ and thus
\[
\lim_{\tau_0\to0}\tilde{\nu}(\tau_0) = 1
\]
from which we have
\[
\lim_{\tau_0\to0} D^{(3)}_*(\mu,\tau_0) = \frac{1}{2\pi} \delta\left(\mu - 1 \right).
\]
To prove the last Condition, we recall that \ref{en:D3} can be rewritten
\[
\lim_{\tau_0\to0} \frac{2\pi}{\tau_0} \int_{-1}^{1} D_*(\mu,\tau_0)\left(1-P_l(\mu)\right) d\mu = l(l+1), \qquad \textrm{for every} \qquad l=1,2,\dots
\]
as observed by Bobylev and Nanbu in \cite{bobylev2000}. Then, substituting $D^{(3)}_*(\mu,\tau_0)$ in the previous relation and solving the integral, we end up with
\[
\lim_{\tau_0\to0} \frac{1-P_l(\tilde{\nu}(\tau_0))}{\tau_0} = l(l+1).
\]
\end{proof}
\begin{remark}
As long as a kernel $D^{(i)}_*$ meets the conditions \ref{en:D1}-\ref{en:D2}-\ref{en:D3} or the equivalent Lemma \ref{lemma:1}, the approximated Boltzmann equation \eqref{eq:boltzmann_approx1} with such $D^{(i)}_*$ is a first order surrogate model for the Landau equation. The differences between the three proposed kernel with $i=1,2,3$ will be investigated in the following sections.
\end{remark}
}

\section{Particle methods and their sG projection} \label{sec:3}
We introduce a random vector $\z=(z_1,\dots,z_d)\in\R^{d_{\z}}$ collecting all the uncertainties of the system and for which it is known its distribution $p(\z)$
\[
\mathrm{Prob}(\z\in \Omega) = \int_{I_{\z}} p(\z) d\z,
\]
for any $I_{\z}\in\R^{d_{\z}}$. The vector $\z$ characterizes the uncertainties of the system due, e.g., to missing information on the initial state, measurement of the model parameters and, boundary conditions^^>\cite{medaglia2022JCP,zhu2017}. 

We are interested in the evolution of the density $f(v,t,\z)$, $v \in \mathbb R^3$, $\z \in \mathbb R^{d_{\z}}$, $t\ge0$, solution to the Landau-Fokker-Planck equation \eqref{eq:LFP} with uncertainties
\begin{equation}
\label{eq:Landau}
\partial_t f(v,t,\z) =\dfrac{1}{8} \nabla_v \cdot \left[ \int_{\mathbb R^3} \Phi(v-v_*,\z) (f(v_*,\z)\nabla_v f(v,\z)-f(v,\z)\nabla_{v_*}f(v_{*},\z))dv_* \right]
\end{equation}
complemented with the initial condition $f(v,0,\z) = f_0(v,\z) \in \mathbb R^{3} \times \mathbb R^{d_{\z}}$. In particular, in \eqref{eq:Landau} the interaction matrix $\Phi(\cdot,\z)$ may encode different types of interactions. 

In this section, we first recall the basic ingredients of DSMC particle method in the absence of uncertainties, as proposed in^^>\cite{caflisch2010}. 
%Deterministic particle methods have been introduced in^^>\cite{carrillo2020}, see also^^>\cite{CiCP-31-997} for related approaches. Furthermore, we highlight that extensive efforts have been devoted to spectral methods, see^^>\cite{filbet2002,pareschi2000,pareschi2003} and the references therein. 

Once established the DSMC solver of interest we seek to derive the related sG particle approach. These methods have been previously studied for mean-field model for emerging phenomena^^>\cite{carrillo19,carrillo19vietnam,medaglia2022}, and subsequently extended to the space-homogeneous Boltzmann equation in^^>\cite{Pareschi2020} and to a non-homogeneous plasma dynamics with BGK collision operator in^^>\cite{medaglia2022JCP}. 

\subsection{DSMC methods}
Let us now give the details of the DSMC algorithms based on the previous first order approximation of the Boltzmann equation in the grazing collision limit and in absence of uncertainties (see^^>\cite{caflisch2010,pareschirusso,nanbu1997,bobylev2000,yang2023adjoint} and the references therein).

\subsubsection{Nanbu-Babovsky scheme}
We first present the Nanbu-Babovsky scheme for the Boltzmann equation \eqref{eq:boltzmann_approx1}. We introduce a time discretization $t^n = n\Delta t$, $\Delta t>0$ and we indicate with $f^n(v)$ the approximation of $f(v,n\Delta t)$. Applying a forward discretization to \eqref{eq:boltzmann_approx1} we get
\begin{equation*}
	f^{n+1} = \left( 1 - \frac{\rho\Delta t}{\epsilon} \right) f^n + \frac{\rho\Delta t}{\epsilon} P_{*,\epsilon}^{(i)}(f,f)
	%\left( \int_{\R^3}\int_{S^2}D\left(\mu,\tau_0\right) f(v')f(v'_*)dndv_* \right),
\end{equation*}
where
\[
P_{*,\epsilon}^{(i)}(f,f) = \int_{\R^3}\int_{S^2}D_*^{(i)}\left(\mu,\tau_0\right) f(v')f(v'_*)dndv_*
\]
where we substituted $D\left(\mu,\tau_0\right)$ by any of the previous approximations $D^{(i)}_*\left(\mu,\tau_0\right)$, $i=1,2,3$ defined in \eqref{eq:D1}-\eqref{eq:D2}-\eqref{eq:D3}. 
%If we now discretize the distribution $f$ at the time step $n$ with a sample of $N$ particles $\{v^n_i\}_{i=1}^{N}$ such that
%\[
%f(v,t) \simeq f^{n}_{N}(v) = \frac{1}{N}\sum_{i=1}^{N} \delta(v-v^n_i),
%\] 
%we can apply a probabilistic interpretation. A particle with velocity $v^n_i$ will collide with probability $\rho\Delta t/\epsilon$ according to the law described by the approximated gain operator, and it will not collide with probability $1-\rho\Delta t/\epsilon$.

Hence, indicating with $q^n=v^n_i-v^n_j$ the relative velocity at the time step $n$ of any two particles $v^n_i,v^n_j$, the collision law reads
\begin{equation} \label{eq:collisions}
\begin{split}
	& v'_i = v^n_i - \dfrac{1}{2} \left( q^n ( 1 - \cos\theta) + h^n \sin\theta\right) \\
	& v'_j = v^n_j + \dfrac{1}{2} \left( q^n ( 1 - \cos\theta) + h^n \sin\theta\right)
\end{split}
\end{equation}
with $h^n$ given by
\[
\begin{split}
	& h^n_x = q^n_\perp \cos \phi\\
	& h^n_y = -\left( q^n_y q^n_x \cos\phi + q^n q^n_z \sin\phi \right) / q^n_\perp\\
	& h^n_z = -\left( q^n_z q^n_x \cos\phi - q^n q^n_y \sin\phi \right) / q^n_\perp,
\end{split}
\]
where $q^n_\perp=\left( (q^n_y)^2 + (q^n_z)^2 \right)^{1/2}$. In the above formulas, the angles $\theta$ and $\phi$ are given by \eqref{eq:costheta1}-\eqref{eq:phi1}, \eqref{eq:costheta2}-\eqref{eq:phi1}, or \eqref{eq:costheta3}-\eqref{eq:phi1} according to the approximated kernel we consider.
\begin{figure}[htb]
 \centering
 \begin{minipage}{.9\linewidth} 
\begin{algorithm}[H] 
\footnotesize
\caption{\small{Nanbu-Babovsky for the Landau equation} } \label{NB_det} 
	\begin{itemize}
	\item Compute the initial velocity of the particles $\{v^0_i\}_{i=1}^N$ by sampling from the initial distribution $f^0(v)=f(v,t=0)$;
	\item for $n=1$ to $\nt$, given $\{v^n_i\}_{i=1}^N$:
	\begin{itemize}
		\item set $N_c=\textrm{Sround}(\rho N\Delta t / 2 \epsilon)$;
		\item select the interaction pairs $(i,j)$ uniformly among all the possible ones, and for every pair $(v^n_i,v^n_j)$:
		\begin{itemize}
			\item compute the cumulative scattering angle $\cos \theta$ and the angle $\phi$ according to \eqref{eq:costheta1}-\eqref{eq:phi1} (for $D^{(1)}_*(\mu,\tau_0)$), \eqref{eq:phi1}-\eqref{eq:costheta2} (for $D^{(2)}_*(\mu,\tau_0)$), or \eqref{eq:phi1}-\eqref{eq:costheta3} (for $D^{(3)}_*(\mu,\tau_0)$);
			\item perform the collision according to \eqref{eq:collisions};
			\item set $v^{n+1}_i=v'_i$ and $v^{n+1}_j=v'_j$;
		\end{itemize}
		\item set $v^{n+1}_i=v^n_i$ for all the particles that have not been collided; 						
	\end{itemize}
	\item end for.
	\end{itemize}
\end{algorithm}
\end{minipage}
\end{figure}

We have denoted by $\textrm{Sround}(x)$ the stochastic rounding of a positive real number $x$
\[
\textrm{Sround} = 
\begin{cases}
	\lfloor x \rfloor + 1 & \textrm{with\, probability}\;x - \lfloor x \rfloor\\
	\lfloor x \rfloor    & \textrm{with\, probability}\;1-x + \lfloor x \rfloor, 
\end{cases}
\]
where $\lfloor x \rfloor $ denotes the integer part of $x$. 
\textcolor{black}{
\begin{remark}\label{remark_DSMC}
In standard DSMC algorithms, the presence of a collisional kernel, representing the forces between molecules or the frequency of interaction, is reflected in the number of colliding particles. In particular, for non-Maxwellian models an acceptance-rejection procedure is performed according to the kernel \cite{pareschirusso}. With the reformulation presented in this work, the physics of the interactions is encoded into the sampling of the angles, which depends on $\tau_0$ defined in \eqref{eq:D}. Thus, Algorithm \ref{NB_det} reads as a standard Nanbu-Babovsky scheme for Maxwell particles with collisions defined by \eqref{eq:collisions} and angles $(\theta,\phi)$ given by \eqref{eq:costheta1}-\eqref{eq:phi1}, \eqref{eq:costheta2}-\eqref{eq:phi1}, or \eqref{eq:costheta3}-\eqref{eq:phi1}, according to $D^{(i)}_*$, but it holds for both the VHS and Coulomb scenario.
\end{remark}
\begin{remark}\label{remark_D*}
The choice of the kernel $D^{(i)}_*$, with $i=1,2,3$, is crucial from a numerical perspective. In particular, to sample $\cos\theta$ and $\phi$ from $D^{(1)}_*$, we highlight that the nonlinear equation \eqref{eq:nonlineq} must be solved for every selected interaction pairs of particles, at every time step. This introduces additional numerical errors and may be computationally demanding. On the contrary, $D^{(2)}_*$ and $D^{(3)}_*$ do not require the use of a solver for the zeros of the nonlinear equation, and thus the code may be easily vectorizable with a significant speed-up in the computational time.
\end{remark}
}
\subsubsection{Bird's scheme}
We can also construct a Monte Carlo scheme based on the classical Bird's no time counter method. In this case, the average number of collisions at each time step is given by
\[
N_c = \frac{N\rho\Delta t}{2\epsilon},
\]
so that the average time between two collisions is 
\[
\Delta t_c = \frac{\Delta t}{N_c} = \frac{2\epsilon}{\rho N}.
\]
Note that, in our setting, Bird's method has the same structure of the Nanbu-Babovsky method where at each $\Delta t_c$ only one pair collides, since
\[
f^{n+1} = \left( 1 - \frac{\rho\Delta t_c}{\epsilon} \right) f^n + \frac{\rho\Delta t_c}{\epsilon} P^{(i)}_{*,\epsilon}(f,f)=\left( 1 - \frac{2}{N} \right) f^n + \frac{2}{N} P^{(i)}_{*,\epsilon}(f,f).
\]
As a consequence, recollision between particles are admissed in a time step $\Delta t=N\Delta t_c$ in contrast to Nanbu-Babovsky. 
In Algorithm \ref{B_det} we recall the standard Bird's scheme for Maxwell particles with collisions defined by \eqref{eq:collisions} in the absence of uncertainties. \textcolor{black}{The considerations of Remark \ref{remark_DSMC}-\ref{remark_D*} apply also to the Bird scheme.}

\begin{figure}[htb] 
\centering
\begin{minipage}{.9\linewidth}
\begin{algorithm}[H] 
\footnotesize
\caption{\small{Bird for the Landau equation} } \label{B_det} 
\begin{itemize}
	\item Compute the initial velocity of the particles $\{v^0_i\}_{i=1}^N$ by sampling from the initial distribution $f^0(v)=f(v,t=0)$;
	\item set the time counter $t_c=0$;
	\item set $\Delta t_c=2\epsilon/\rho N$;
	\item for $n=1$ to $\nt$, given $\{v^n_i\}_{i=1}^N$:
	\begin{itemize}
		\item since $t_c < n\Delta t$:
		\begin{itemize}
			\item select a random pair $(i,j)$ uniformly among all the possible ones;
			\item compute the cumulative scattering angle $\cos \theta$ and the angle $\phi$ according to \eqref{eq:costheta1}-\eqref{eq:phi1} (for $D^{(1)}_*(\mu,\tau_0)$), \eqref{eq:costheta2}-\eqref{eq:phi1} (for $D^{(2)}_*(\mu,\tau_0)$), or \eqref{eq:costheta3}-\eqref{eq:phi1} (for $D^{(3)}_*(\mu,\tau_0)$);
			\item perform the collision according to \eqref{eq:collisions};
			\item set $w_i=v'_i$ and $w_j=v'_j$;
			\item update the time counter $t_c=t_c+\Delta t_c$;		
		\end{itemize}
		\item set $v^{n+1}_i=w_i$ for $i=1,\dots,N$;			
	\end{itemize}
	\item end for.
\end{itemize}
\end{algorithm}
\end{minipage}
\end{figure}

%%%%%%%%%%%%%%%%%%%%%%%%%%%%%%%%%%%%%%%%%%%%%%%%%%%%%%%%%%%%%%%%%%%%%%%%%%%
\subsection{Stochastic Galerkin reformulation of the DSMC method}
We consider a sample of $N$ particles $\textcolor{black}{v_i^n(\z)=v_i(t^n,\z)}$, $i=1,\dots,N$ at time $t^n=n\Delta t$, and we expand them by their generalized polynomial chaos (gPC) expansion 
\be\label{eq:gPCE}
\textcolor{black}{v^n_i(\z) \approx v^{n,M}_i(\z) = \sum_{k=0}^M \hat{v}^n_{i,k} \Psi_k(\z)},
\ee
being $\{\Psi_k(\z)\}_{k=0}^M$ a set of polynomials of degree less or equal to $M\in\mathbb{N}$, orthonormal with respect to the measure $p(\z)d\z$
\[
\int_{I_\z} \Psi_k(\z) \Psi_l(\z) p(\z) d\z = \E_{\z} [ \Psi_k(\cdot) \Psi_l(\cdot) ] = \delta_{kl},
\]
where \textcolor{black}{$I_\z\in\R^{d_\z}$} and $\delta_{kl}$ is the Kronecker delta function. The polynomials $\{\Psi_h(\z)\}_{k=0}^M$ are chosen following the so-called Wiener-Askey scheme, depending on the distribution of the parameters. In \eqref{eq:gPCE}, for fixed $k=0,\dots,M$, the term $\hat{v}^n_{i,h}$ is the projection in the space generated by the polynomial of degree $k\ge 0$ 
\be \label{eq:proj}
	\hat{v}^n_{i,k} = \int_{I_{\z}} v^n_i(\z) \Psi_k(\z) p(\z) d\z = \E_{\z} [ v^n_i(\cdot) \Psi_k(\cdot) ].
\ee
\textcolor{black}{We also have
\[
q^{n,M}_{ij}(\z) = v^{n,M}_i(\z) - v^{n,M}_j(\z) = \sum_{k=0}^M \left(\hat{v}^n_{i,k}-\hat{v}^n_{j,k}\right)\Psi_k(\z)=\sum_{k=0}^M \hat{q}^n_{ij,k}\Psi_k(\z),
\]
with $\hat{q}^n_{ij,k}=\hat{v}^n_{i,k}-\hat{v}^n_{j,k}$, and
\[
\begin{split}
	& h^{n,M}_x(\z) = q^{n,M}_\perp(\z) \cos \phi\\
	& h^{n,M}_y(\z) = -\left( q^{n,M}_y(\z) q^{n,M}_x(\z) \cos\phi + q^{n,M}(\z) q^{n,M}_z(\z) \sin\phi \right) / q^{n,M}_\perp(\z)\\
	& h^{n,M}_z(\z) = -\left( q^{n,M}_z(\z) q^{n,M}_x(\z) \cos\phi - q^{n,M}(\z) q^{n,M}_y(\z) \sin\phi \right) / q^{n,M}_\perp(\z).
\end{split}
\]
We recall that in general $\tau$ defined in \eqref{eq:oneovertau} is a function of the random parameter $\tau=\tau(\z)$, since it may depend on the relative velocity, and then we also have $\tau_0=\tau_0(\z)$. As a consequence, the three kernels proposed in Section \ref{sec:2.2} depend on the random inputs $D^{(i)}_*=D^{(i)}_*(\mu,\tau_0(\z))$ for every $i=1,2,3$. To sample the angles $\theta=\theta(\z)$ and $\phi$ according to $D^{(1)}_*(\mu,\tau_0(\z))$, we need to solve
\be \label{eq:nonlineq_z}
\coth A(\z) - \frac{1}{A(\z)} = e^{-2\tau_0(\z)}
\ee
and then compute
\be \label{eq:theta_D1_z}
\cos \theta(\z) = \frac{1}{A(\z)} \ln (\exp^{-A(\z)} + 2 r_1 \sinh A(\z) )
\ee
and
\be \label{eq:phi_D1_z}
\phi = 2\pi r_2
\ee
with $r_1,r_2 \sim \mathcal{U}([0,1])$. 
Sampling $D^{(2)}_*(\mu,\tau_0(\z))$ correspond to fix
\be \label{eq:theta_D2_z}
\cos \theta (\z) = \nu(\tau_0(\z)),
\ee
with $\nu(\cdot)$ as in \eqref{eq:nu}, and compute $\phi$ as in \eqref{eq:phi_D1_z}. Similarly, if $D^{(3)}_*(\mu,\tau_0(\z))$ is considered, we have
\be \label{eq:theta_D3_z}
\cos \theta (\z) = \tilde{\nu}(\tau_0(\z)),
\ee
with $\tilde{\nu}(\cdot)$ as in \eqref{eq:nutilde}, and again $\phi$ as in \eqref{eq:phi_D1_z}.
} 

We substitute the approximation of the velocities into \eqref{eq:collisions} and we project against $\Psi_l(\z)p(\z)d\z$ for every $l=0,\dots,M$ to obtain
\be \label{eq:sG_collision}
\begin{split}
	& \hat{v}'_{i,k} = \hat{v}^{n}_{i,k} - \dfrac{1}{2}\left( \hat{q}^n_{ij,k} - \sum_{l=0}^{M} \hat{q}^n_{ij,l} \hat{V}^n_{lk} + \hat{W}^n_{k} \right) \\
	& \hat{v}'_{j,k} = \hat{v}^{n}_{j,k} + \dfrac{1}{2}\left( \hat{q}^n_{ij,k} - \sum_{l=0}^{M} \hat{q}^n_{ij,l} \hat{V}^n_{lk} + \hat{W}^n_{k} \right)
\end{split}
\ee
where the collision matrices are defined as
\be \label{eq:matrices}
\begin{split}
	& \hat{V}^n_{lk} = \int_{I_{\z}} \cos \theta_{ij}(\z) \Psi_l(\z) \Psi_k(\z) p(\z) d\z \\
	&\hat{W}^n_{ij,k} = \int_{I_{\z}} h_{ij}^{n,M}(\z) \sin\theta_{ij}(\z) \Psi_k(\z) p(\z) d\z,
\end{split} 
\ee
\textcolor{black}{where the subscript in $\theta_{ij}$ refers to angle sampled for the colliding couple $ij$}.
Therefore, we may rephrase the Nanbu-Babovsky scheme in the sG framework as presented in Algorithm \ref{NB_unc}. Whereas the sG version of the Bird's scheme is given in Algorithm \ref{B_unc}. 

\begin{figure}[htb] 
	\centering
	\begin{minipage}{.9\linewidth}
		\begin{algorithm}[H] 
			\footnotesize
			\caption{\small{sG Nanbu-Babovky for the Landau equation} } \label{NB_unc}
			\begin{itemize}
				\item Compute the initial gPC expansions $\{v^{M,0}_i\}_{i=1}^N$ from the initial distribution $f^0(v)$;
				\item for $n=1$ to $\nt$, given the projections $\{\hat{v}^n_{i,k},\,i=1,\dots,N,\,k=0,\dots,M\}$:
				\begin{itemize}
					\item set $N_c=\textrm{Sround}(\rho N\Delta t / 2 \tau)$;
					\item select the interaction pairs $(i,j)$ uniformly among all the possible ones, and for every pair $(\hat{v}^n_{i,k},\hat{v}^n_{j,k})$:
					\begin{itemize}
						\item \textcolor{black}{compute the cumulative scattering angle $\cos \theta(\z)$ and the angle $\phi$ according to \eqref{eq:theta_D1_z}-\eqref{eq:phi_D1_z} (for $D^{(1)}_*(\mu,\tau_0(\z))$), \eqref{eq:theta_D2_z}-\eqref{eq:phi_D1_z} (for $D^{(2)}_*(\mu,\tau_0(\z))$), or \eqref{eq:theta_D3_z}-\eqref{eq:phi_D1_z} (for $D^{(3)}_*(\mu,\tau_0(\z))$);}
						\item compute the matrices $\hat{V}_{kl}$ and $\hat{W}_{k}$ according to \eqref{eq:matrices};
						\item perform the collision according to \eqref{eq:sG_collision};
						\item set $\hat{v}^{n+1}_{i,k}=\hat{v}'_{i,k}$ and $\hat{v}^{n+1}_{j,k}=\hat{v}'_{j,k}$;
					\end{itemize}
					\item set $\hat{v}^{n+1}_{i,k}=\hat{v}^{n}_{i,k}$ for all the particles that have not been collided; 						
				\end{itemize}
				\item end for.
			\end{itemize}
		\end{algorithm}
	\end{minipage}
\end{figure}
\begin{figure}[htb]
	\centering
	\begin{minipage}{.9\linewidth}
		\begin{algorithm}[H] 
			\footnotesize
			\caption{\small{sG Bird for the Landau equation} } \label{B_unc}
			\begin{itemize}
				\item Compute the initial gPC expansions of the particles $\{v^{M,0}_i\}_{i=1}^N$ by sampling from the initial distribution $f^0(v)=f(v,t=0)$;
				\item set the time counter $t_c=0$;
				\item set $\Delta t_c=2\epsilon/\rho N$;
				\item for $n=1$ to $\nt$, given the projections $\{\hat{v}^n_{i,k},i=1,\dots,N,\,k=0,\dots,M\}$:
				\begin{itemize}
					\item since $t_c < n\Delta t$:
					\begin{itemize}
						\item select a random pair $(i,j)$ uniformly among all the possible ones;
						\item \textcolor{black}{compute the cumulative scattering angle $\cos \theta(\z)$ and the angle $\phi$ according to \eqref{eq:theta_D1_z}-\eqref{eq:phi_D1_z} (for $D^{(1)}_*(\mu,\tau_0(\z))$), \eqref{eq:theta_D2_z}-\eqref{eq:phi_D1_z} (for $D^{(2)}_*(\mu,\tau_0(\z))$), or \eqref{eq:theta_D3_z}-\eqref{eq:phi_D1_z} (for $D^{(3)}_*(\mu,\tau_0(\z))$);}
						\item compute the collision matrices $\hat{V}_{kl}$ and $\hat{W}_{k}$ according to \eqref{eq:matrices};
						\item perform the collision according to \eqref{eq:sG_collision};
						\item set $\hat{w}_{i,k}=\hat{v}'_{i,k}$ and $\hat{w}_{j,k}=\hat{v}'_{j,k}$;
						\item update the time counter $t_c=t_c+\Delta t_c$;		
					\end{itemize}
					\item set $\hat{v}^{n+1}_{i,k}=\hat{w}_{i,k}$ for $i=1,\dots,N$;			
				\end{itemize}
				\item end for.
			\end{itemize}
		\end{algorithm}
	\end{minipage}
\end{figure}
\textcolor{black}{We stress that Remarks \ref{remark_DSMC}-\ref{remark_D*} hold also in the presence of uncertainties. Algorithms \ref{NB_unc}-\ref{B_unc} read as standard schemes for Maxwell molecules, since the physics of the interaction is embedded into the sampled angle $\theta(\z)$, which depends on $\tau_0(\z)$. Moreover, since $A=A(\z)$, the computational cost needed to sample $D^{(1)}_*(\mu,\tau_0(\z))$ scales along with the dimensionality of the random parameter $d_\z$, making the use of such kernel even more prohibitive. }

\subsection{Consistency estimate}
\textcolor{black}{
In this section, we study the error produced by the particle sG method in the evaluation of any observable $\varphi=\varphi(v)$. These results are inspired by Appendix B.2 of \cite{Pareschi2020} and Section 4.2 of \cite{medaglia2022}. 
Let us denote by $f(v,t,\z)$ the solution to the Boltzmann equation with random inputs 
\be \label{eq:boltzmann_uq}
\frac{\partial f}{\partial t} (v,t,\z) = \int_{\R^3}\int_{S^2} B\left(q,\frac{q\cdot n}{|q|},\z\right) \left( f(v',\z)f(v'_*,\z) - f(v,\z)f(v_*,\z) \right) dn dv_*,
\ee
and by $f_\epsilon(v,t,\z)$ the solution to the first order approximation of Boltzmann equation in the presence of uncertainties
\begin{equation}\label{eq:boltzmann_approx1_uq}
	\frac{\partial f_\epsilon}{\partial t} (v,t,\z) = \frac{1}{\epsilon} \left[\int_{\R^3}\int_{S^2}D\left(\mu,\tau_0(\z)\right) f_\epsilon(v',\z)f_\epsilon(v'_*,\z)dndv_* \, - \, \rho f_\epsilon(v,\z)\right].
\end{equation}
We further indicate by
\[
f_{\epsilon,N}(v,t,\z) = \dfrac{1}{N} \sum_{i=1}^N \delta(v-v_i(t,\z)) \qquad f^M_{\epsilon,N}(v,t,\z) = \dfrac{1}{N} \sum_{i=1}^N \delta(v-v^M_i(t,\z))
\]
the empirical density functions obtained with the velocities $v_i(t,\z)$ and their gPC approximation $v^M_i(t,\z)$, respectively. The expectations of $\varphi$ reads 
\[
\left\langle \varphi, f_\epsilon \right\rangle (t,\z) = \int_{\R^3} \varphi(v) f_\epsilon(v,t,\z) dv 
\]
so that 
\[
\left\langle \varphi, f_{\epsilon,N} \right\rangle (t,\z) = \dfrac{1}{N} \sum_{i=1}^{N} \varphi ( v_i (t,\z) ), \qquad  \left\langle \varphi, f^{M}_{\epsilon,N} \right\rangle (t, \z) = \dfrac{1}{N} \sum_{i=1}^{N} \varphi ( v^M_i (t,\z) ).
\]
From the central limit theorem we have the following result.
\begin{lemma} \label{lemma_cons}
	If we denote by $\mathbb{E}_{\R^3}[\cdot]$ the expectation with respect to $f_\epsilon$ in the velocity space, the root mean square error satisfies
	\[
	\mathbb{E}_{\R^3} \bigg[ \Big( \left\langle \varphi, f_\epsilon \right\rangle (t, \z) - \left\langle \varphi, f_{\epsilon,N} \right\rangle (t, \z)    \Big)^2 \bigg]^{1/2} = \dfrac{\sigma_\varphi (t,\z)}{N^{1/2}}
	\]
	with
	\[
	\sigma^2_\varphi (t,\z)  =  \int_{\R^3}  ( \varphi(v) - \left\langle \varphi, f_\epsilon \right\rangle (t,\z) )^2  f_\epsilon(v,t,\z) dv.
	\]
\end{lemma}
\noindent Denoting by $H^r_p(I_\z)$ the weighted Sobolev space
\[
H^r_p(I_\z) = \bigg\{ v : I_\z \rightarrow \R^3 \; : \; \dfrac{\partial^k v}{\partial \z^k} \in L^2_p(I_\z), \; 0 \leq k \leq r \bigg\},
\]
from the polynomial approximation theory \cite{xiu2010}, we have the following spectral estimate.
\begin{lemma} \label{lemma_cons2}
	For any $v(\z) \in H^r_p(I_\z), \, r\geq 0$, there exists a constant $C$ independent of $M>0$ such that 
	\[
	\| v - v^M \|_{L^2_p(I_\z)} \leq \dfrac{C}{M^r} \| v \|_{H^r_p(I_\z)}.
	\]
\end{lemma}
\noindent Next, for any random variable $V(t,\z)$ taking values in $L^2_p(I_\z)$, we define
\[
\| V \|_{L^2(I_\z,L^2(\R^3))} = \| \mathbb{E}_{\R^3} [V^2]^{1/2} \|_{L^2_p(I_\z)},
\]
and equivalently
\[
\| W \|_{L^2(\R^3,L^2(I_\z))} =  \mathbb{E}_{\R^3} \left[ \| V \|^2_{L^2_p(I_\z)}  \right]^{1/2}. 
\]
Then, we have the following result.
\begin{theorem} \label{theorem:error_estimate}
	Let $f(v,t,\z)$ be a probability density function in $v$ at the time $t$, solution to the Boltzmann equation with random inputs \eqref{eq:boltzmann_uq}. 
	Let $f^M_{\epsilon,N}(v,t,\z)$ be the empirical measure obtained from the $N$-particles sG approximation $\{v^M_i(\z,t)\}_i$, numerical resolution to the approximated Boltzmann equation with uncertainties \eqref{eq:boltzmann_approx1_uq}. If $v_i(t,\z)\in H^r_p(I_\z)$ for every $i=1,\dots,N$, we have the following estimate 
	\[
	\| \left\langle \varphi, f \right\rangle - \left\langle \varphi, f^M_{\epsilon,N} \right\rangle \|_{L^2(\R^3,L^2_p(I_\z))} \leq O(\epsilon) +  \dfrac{\|\sigma_{\varphi}\|_{L^2_p(I_\z)}}{N^{1/2}} + \dfrac{C}{M^r} \left( \dfrac{1}{N} \sum_{i=1}^{N} \| \nabla \varphi(\xi_i) \|_{L^2_p(I_\z)} \right),
	\]
	where $\varphi$ is a test function, $C>0$ is a constant independent on $M$ and $\xi_i=(1-\vartheta)v_i+\vartheta v^M_i,\,\vartheta\in(0,1)$.
\begin{proof}
	Thanks to the triangular inequality we have 
	\begin{equation}
	\begin{split}
		\| \left\langle \varphi, f \right\rangle - \left\langle \varphi, f^M_{\epsilon,N} \right\rangle \|_{L^2(\R^{3},L^2_p(I_\z))}  \leq & \underbrace{  \| \left\langle \varphi, f \right\rangle   - \left\langle \varphi, f_\epsilon \right\rangle \|_{L^2(\R^{3},L^2_p(I_\z))}   }_{I}  \\ 
		& +\quad \underbrace{  \| \left\langle \varphi, f_{\epsilon} \right\rangle - \left\langle \varphi, f_{\epsilon,N} \right\rangle \|_{L^2(\R^{3},L^2_p(I_\z))}  }_{II} \\
		& +\quad \underbrace{  \| \left\langle \varphi, f_{\epsilon,N} \right\rangle - \left\langle \varphi, f^M_{\epsilon,N} \right\rangle \|_{L^2(\R^{3},L^2_p(I_\z))}   }_{III}.
	\end{split}
	\end{equation} 
	As presented in Section \ref{sec:2.1}, $f_\epsilon(v,t,\z)$ is a first order approximation of the Boltzmann equation in $\epsilon$. As observed by Nanbu and Bobylev in \cite{bobylev2000}, the accuracy is not formally worse than any other first order approximation of the Boltzmann equation. Therefore, the term $I$ is of first order in $\epsilon$: 
	\[
	\| \left\langle \varphi, f \right\rangle   - \left\langle \varphi, f_{\epsilon} \right\rangle \|_{L^2(\R^{3},L^2(I_\z))} = O(\epsilon).
	\]	
	The second term $II$ can be evaluated exploiting the result of Lemma \ref{lemma_cons}. Therefore, we have
	\[
	II = \dfrac{\| \sigma_{\varphi}(\z) \|_{L^2(I_\z)}}{N^{1/2}}.
	\]
	Finally, for $III$ we have
	\[
	\left\| \dfrac{1}{N} \sum_{i=1}^{N} \big(  \varphi(v_i) - \varphi(v^M_i)  \big)\right\|_{L^2(\R^{3},L^2_p(I_\z))} \leq \dfrac{1}{N} \sum_{i=1}^{N} \| \varphi(v_i) - \varphi(v^M_i)  \|_{L^2(\R^{3},L^2_p(I_\z))},
	\]
	and from the mean value theorem $\varphi(v_i) - \varphi(v^M_i) = \nabla \varphi(\xi_i) \cdot (v_i - v^M_i)$, for $\xi_i=(1-\vartheta)v_i+\vartheta v^M_i,\,\vartheta\in(0,1)$. Thanks to Lemma \ref{lemma_cons2} with $C=\max_i C_i \| v_i \|_{H^r_p(I_\z)}$ we have
	\[
	III \leq  \dfrac{1}{N} \sum_{i=1}^{N} \| \nabla\varphi(\xi_i)  \|_{L^2_p(I_\z)} \|v_i - v^M_i  \|_{L^2_p(I_\z)} \leq \dfrac{C}{M^r} \bigg( \dfrac{1}{N} \sum_{i=1}^{N} \| \nabla\varphi(\xi_i)  \|_{L^2_p(I_\z)} \bigg).
	\]
\end{proof}
\end{theorem}
\begin{remark}
The regularity of the Landau approximation in the space of the random parameter is crucial for the gPC approximation, as shown in Lemma \ref{lemma_cons2} and Theorem \ref{theorem:error_estimate}. In this direction, we observe that the angle sampled from the kernel $D^{(3)}_*$ is such that $\cos \theta(\z) = \tilde{\nu}(\tau_0(\z))$ is of class $C^\infty$ in the space of the random parameters. On the contrary, considering the kernel $D^{(2)}_*$, we have that the function $\nu(\tau_0(\z))$ is not differentiable in $\tau_0(\z)=1$. Similarly, taking $D^{(1)}_*$ into account, the numerical resolution of the nonlinear equation \eqref{eq:nonlineq_z} forces the introduction of cut-offs at the boundary of the domain of $A(\z)$ (see \cite{nanbu1997}, Section C).
\end{remark}
}

%%%%%%%%%%%%%%%%%%%%%%%%%%%%%%%%%%%%%%%%%%%%%%%%%%%%%%%%%%%%%%%%%%%%%
\section{Numerical examples and applications}\label{sec:4}
In this section, we present several numerical test and examples to validate our algorithms both without and with uncertain parameters. First, we investigate the Maxwellian case without uncertainties. Then, we show several tests for the DSMC-sG method, for both the Nanbu-Babovsky and the Bird's schemes. We check for the spectral convergence and the accordance with the exact BKW solution of the Maxwellian case with uncertainties. Then, we consider the Coulombian case, focusing the attention on the regularity of the kernels $D^{(1)}_*$, $D^{(2)}_*$, and $D^{(3)}_*$. We check again for the spectral convergence, and then we test the schemes on the standard case studies for the homogeneous Landau equation. In particular, we concentrate on the capability to reach the equilibrium starting from different uncertain initial conditions, i.e., anisotropic initial temperature, sum of two Gaussians, and bump-on-tail distributions. In all the subsequent tests, we consider the physical constants fixed as follows: $e=\epsilon_0=m=\rho=1$, and $\log \Lambda = 0.5$. 

\textcolor{black}{The number of particles is varied between the values $N=10^6,\,5\times10^6,\,5\times10^7$. In more details, $N=10^6$ is used to compute the stochastic Galerkin error and to compare the DSMC-sG and the DSMC-MC schemes; $N=5\times10^6$ is adopted in all the tests showing the moments of the distribution; $N=5\times10^7$ is used to display the distribution $f(\cdot)$ or its marginals. The reason for this choice is that the distribution $f(\cdot)$ is reconstructed using histograms in 3D and therefore it is necessary to choose a high number of particles to smooth out statistical fluctuations due to the Monte Carlo nature of the algorithm. On the other hand, the moments of the distribution are averaged quantities and therefore fewer particles can be used. The convergence and comparison tests are independent from $N$. We will return to this in the following.}

\subsection{Test 1: Exact solution in the Maxwellian case}
\begin{figure}
	\centering
	\includegraphics[width = 0.46\linewidth]{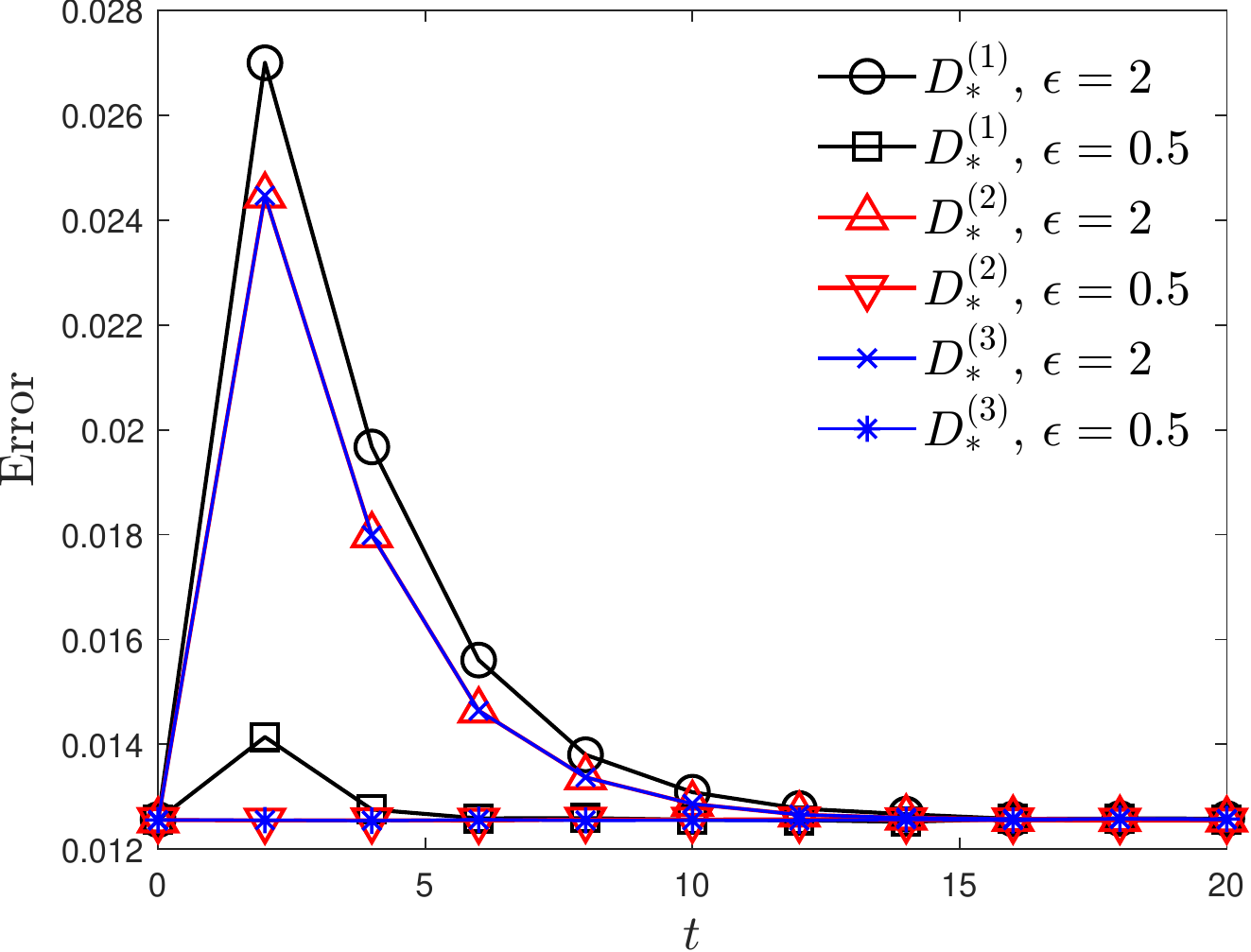}
	\includegraphics[width = 0.45\linewidth]{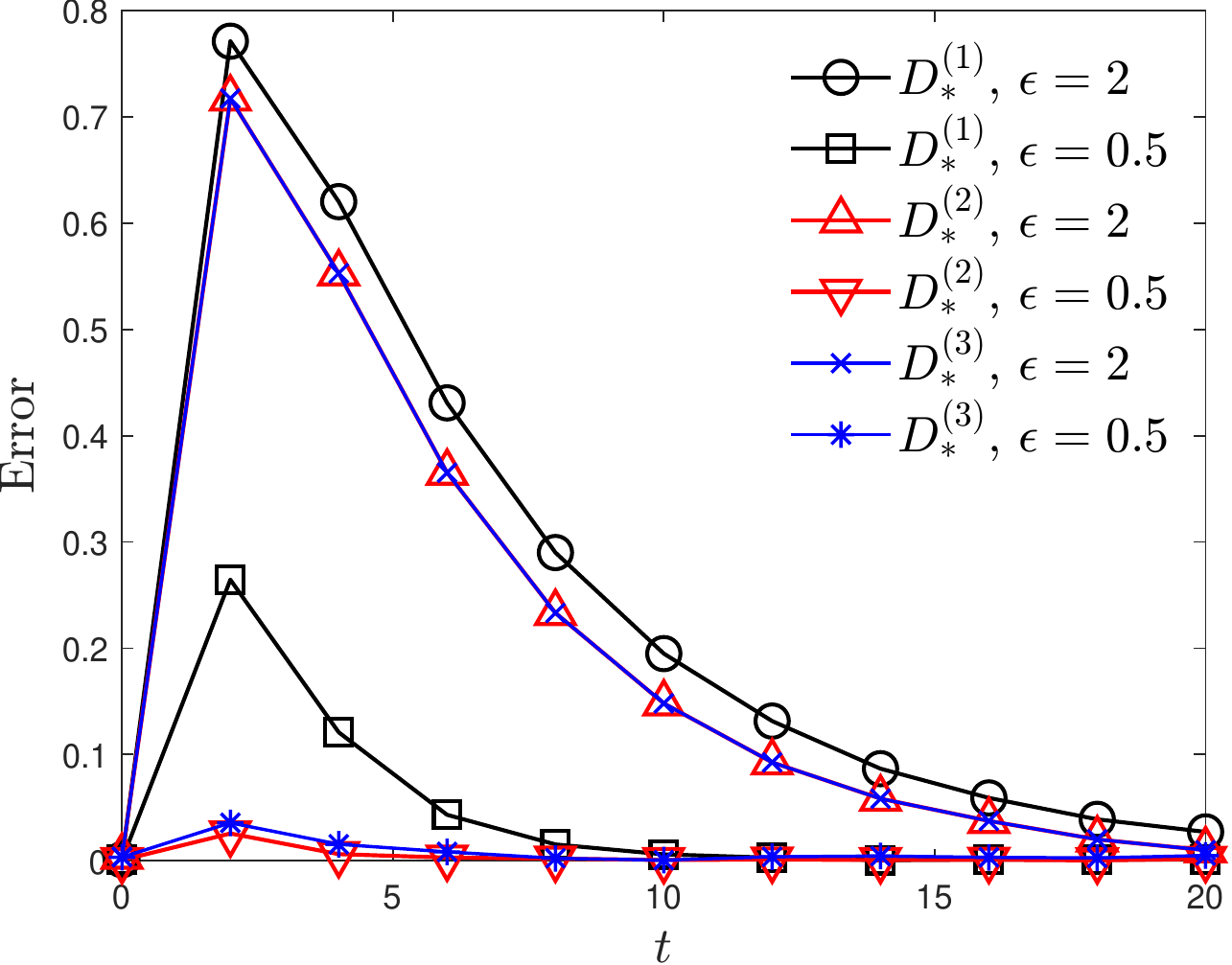}
	\caption{\textcolor{black}{\small{\textbf{Test 1}. Time evolution of the relative $L^2$ errors with respect to the BKW solution of the distribution $f(v,t)$ (left) and of the fourth order moment $\textrm{M}4(t)$ (right), for the kernels $D^{(1)}_*$, $D^{(2)}_*$, and $D^{(3)}_*$, and different values of $\epsilon=\rho\Delta t$. In all the tests, we use $5\times10^7$ particles and the Nanbu-Babovsky scheme. Initial conditions given by \eqref{eq:BKW_det} with $t=0$ and $T=1$.}}}
	\label{fig:test_1_L2_err}
\end{figure}
\begin{figure}
	\centering
	\includegraphics[width = 0.45\linewidth]{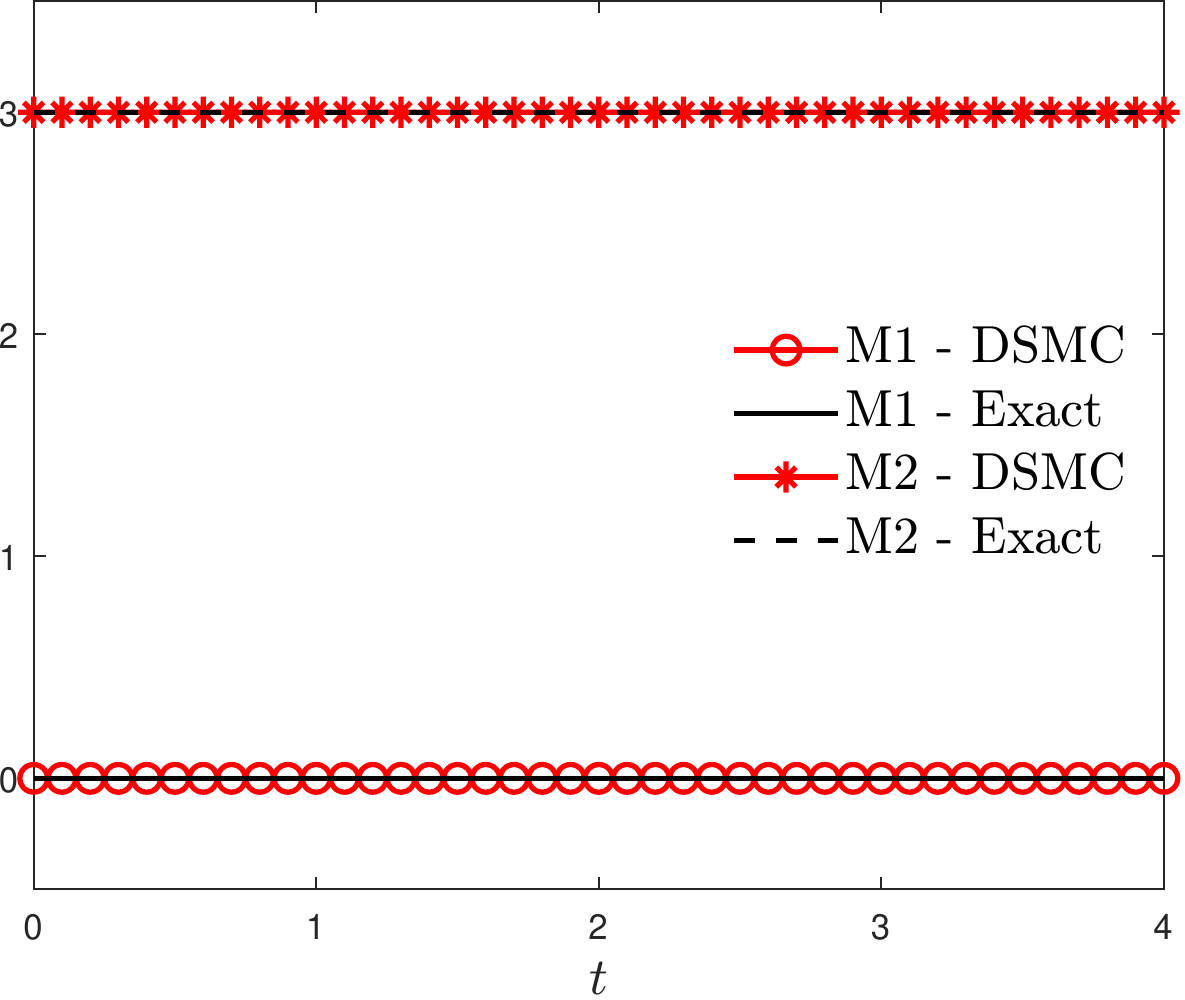}
	\includegraphics[width = 0.46\linewidth]{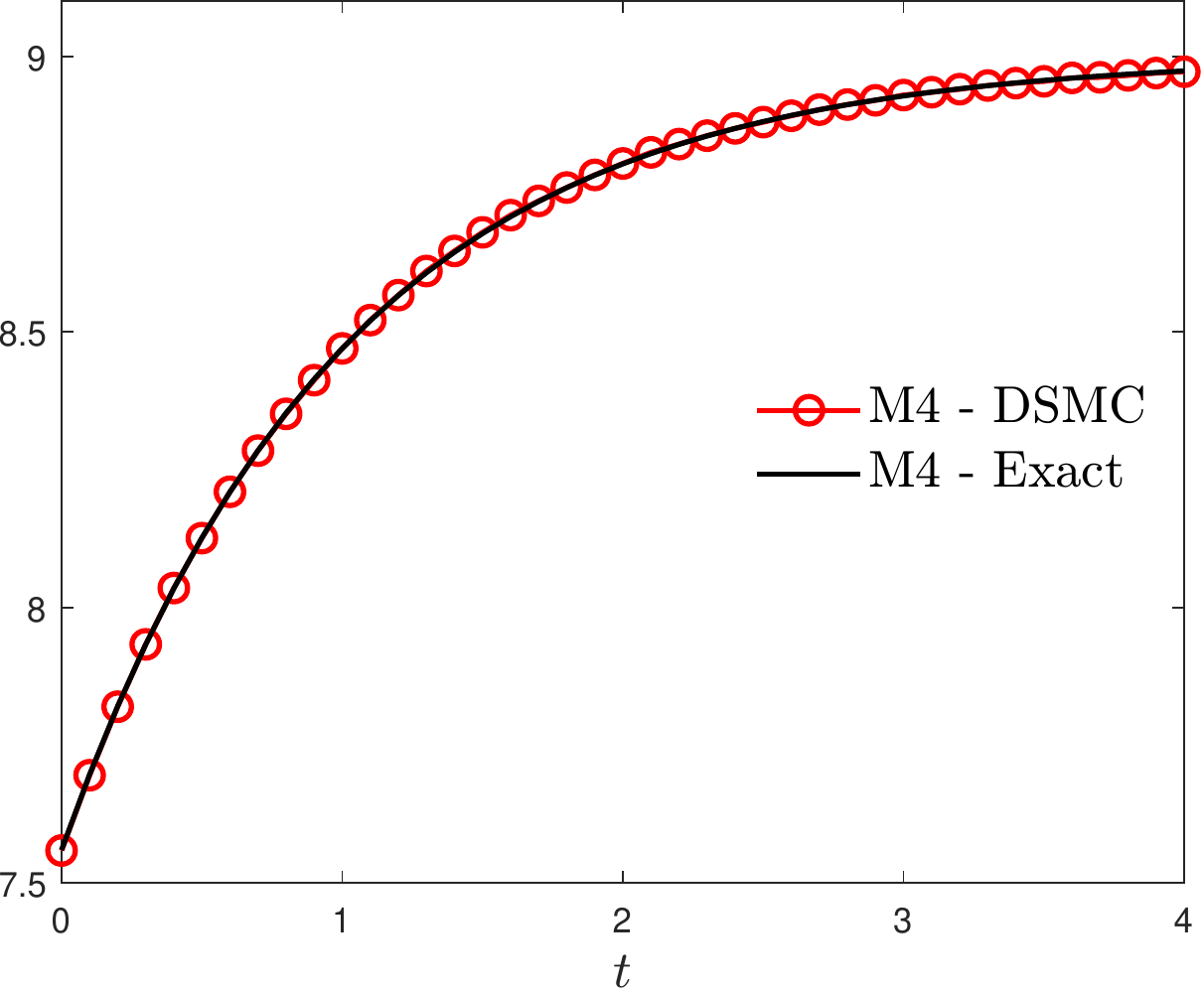}
	\caption{\small{\textbf{Test 1}. Comparison between the time evolution of the first and second order moments (left), and of the fourth order moment (right), obtained with the DSMC method and from the exact BKW solution. In all the tests, we use $5\times10^6$ particles, $\Delta t=\epsilon/\rho=0.1$, the kernel $D^{(3)}_*$, and the Nanbu-Babovsky scheme. Initial conditions given by \eqref{eq:BKW_det} with $t=0$ and $T=1$.}}
	\label{fig:test_1_moments}
\end{figure}

We consider the model with Maxwell molecules, i.e., $\gamma=0$ in \eqref{eq:scattering}. In 3D, an exact solution is given by (see^^>\cite{carrillo2020}, Appendix \ref{sec:appendixA})
\be\label{eq:BKW_det}
f(v,t) = \frac{1}{(2\pi K(t))^{3/2}} e^{-\frac{|v|^2}{2K(t)}} \left(\frac{5K(t) - 3 T}{2K(t)} + \frac{T-K(t)}{2K(t)^2}|v|^2\right),
\ee
with $T$ temperature and
\[
\textcolor{black}{K(t) = T \left( 1 - \frac{2}{5}e^{-t/2} \right).}
\]
We recall that the moments of order zero, one, and two are conserved, while the exact time evolution of the fourth order moment reads
\be \label{eq:m4_det}
\textrm{M}4(t) = 9 K(t) (2T - K(t)).
\ee
In Figure \ref{fig:test_1_L2_err}, we show the relative $L^2$ errors of the distribution $f(v,t)$ and the fourth order moment $\textrm{M}4(t)$ of the DSMC approximation with respect to the exact BKW solution \eqref{eq:BKW_det} and \eqref{eq:m4_det}. In particular, we compare the different kernels $D^{(1)}_*$, $D^{(2)}_*$, and $D^{(3)}_*$, together with different values of the scaling parameter $\epsilon=2,0.5$. We use \textcolor{black}{$N=5\times10^7$ particles}, and initial conditions given by \eqref{eq:BKW_det} with $t=0$ and $T=1$. We choose the Nanbu-Babovsky scheme given by Algorithm \ref{NB_det}. As we can notice, the kernels $D^{(2)}_*$ and $D^{(3)}_*$ perform better than $D^{(1)}_*$ for small times. 

In Figure \ref{fig:test_1_moments} we display the time evolution of the first and second order moments (left panel), and the fourth order moment (right panel) computed with the DSMC scheme, together with the exact solutions. We choose the Nanbu-Babovsky algorithm, kernel $D^{(3)}_*$, $N=5\times10^6$ particles, $\Delta t=\epsilon/\rho=0.1$, and initial conditions given by \eqref{eq:BKW_det} with $t=0$ and $T=1$.

\subsection{Test 2: Trubnikov test} \label{sec:test2}
\begin{figure}
	\centering
	\includegraphics[width = 0.45\linewidth]{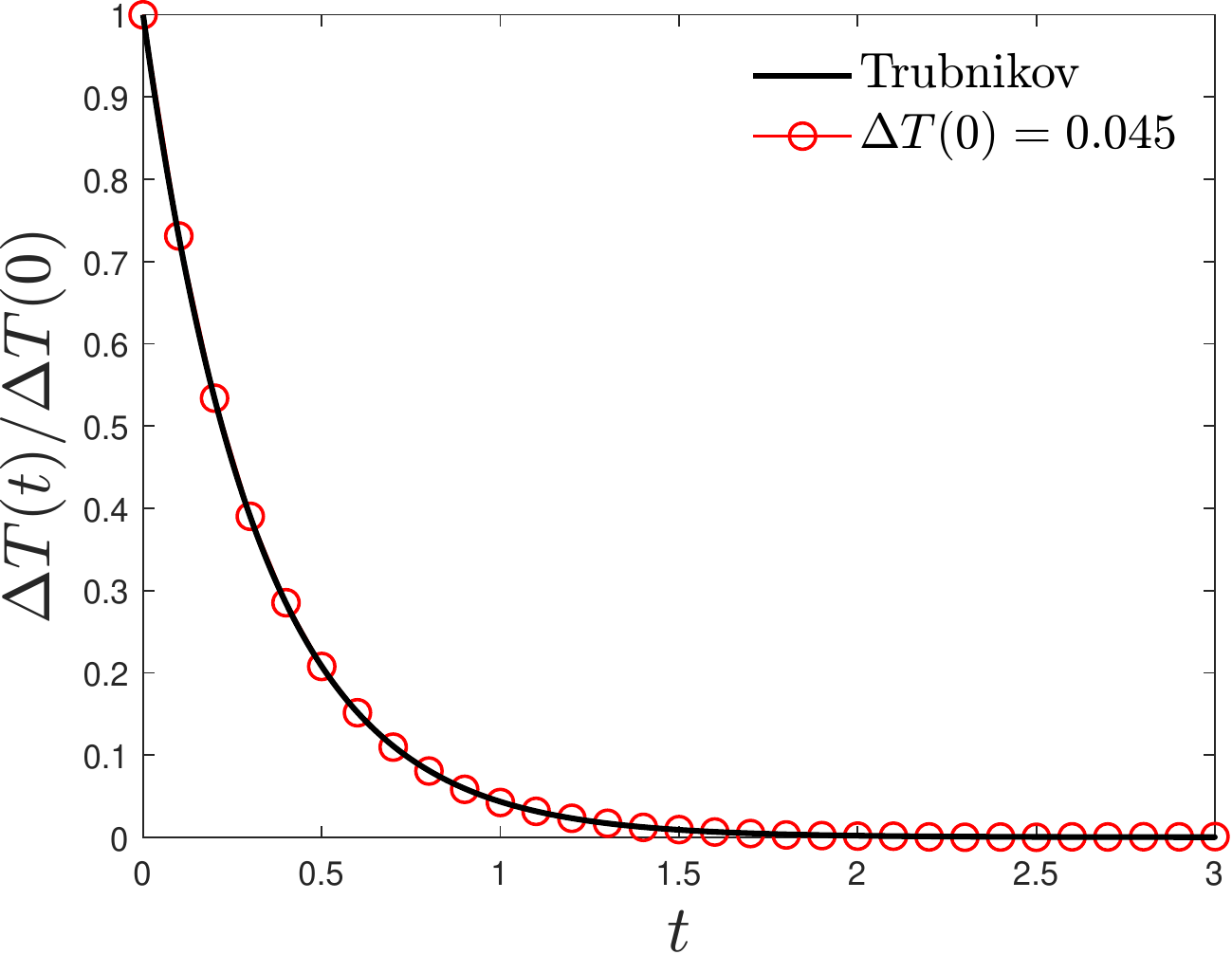}
	\includegraphics[width = 0.45\linewidth]{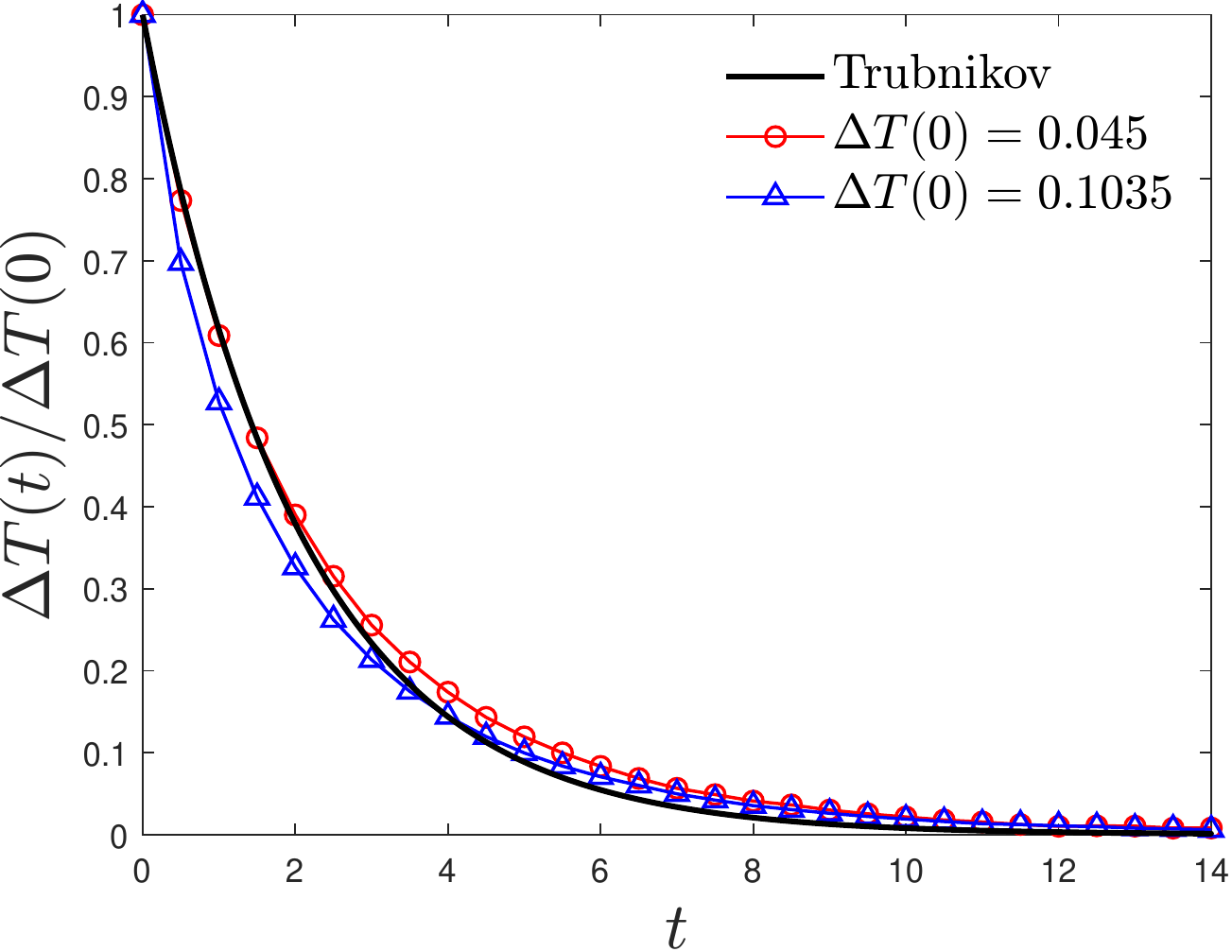}	
	\caption{\small{\textbf{Test 2}. Time evolution of the relative temperature difference $\Delta T(t)/\Delta t(0)$ for the Maxwellian (left) and the Coulomb (right) cases. The solid black lines are the Trubnikov solutions with rate given, respectively, by \eqref{eq:trub_max} and \eqref{eq:trub_coul}. The particles are $N=5\times10^6$, the time step $\Delta t=0.1$, and the temperature $T=0.07$. We choose the Nanbu-Babovsky scheme with the kernel $D^{(3)}_*$. Initial conditions given by \eqref{eq:init_trub_det}, with $z$-temperature $T^0_z=0.04$ for the Maxwellian case, and $T^0_z=0.04,\,0.001$ for the Coulomb case.}}
	\label{fig:test_2_trubnikovdet}
\end{figure}
We initialize the distribution as an ellipsoid, i.e., as
\be \label{eq:init_trub_det}
f_0(v) = \frac{1}{\left(2\pi\right)^{3/2}} \frac{1}{\sqrt{T^0_x T^0_y T^0_z}} \left(e^{-\frac{v^2_x}{2 T^0_x}}e^{-\frac{v^2_y}{2 T^0_y}}e^{-\frac{v^2_z}{2 T^0_z}}\right),
\ee
with anisotropic initial temperature
\[
T^0_x = T^0_y > T^0_z. 
\]
The temperature difference $\Delta T(t)=T_x(t)-T_z(t)$ goes to zero exponentially with a specific rate $\tau_T$
\[
\Delta T(t) = \Delta T(0) e^{-t/\tau_T}.
\]
In the Maxwellian case (see Appendix \ref{sec:appendixA} for further details), we have 
\be \label{eq:trub_max}
\tau_{T} = \frac{2}{3\rho},
\ee
\textcolor{black}{while} in the Coulombian case Trubnikov^^>\cite{Trubnikov1965} obtained an approximated solution in the limit of small temperature difference $|T^0_x - T^0_z| \ll 1$, that is 
\be \label{eq:trub_coul}
\tau_T = \frac{5}{8}\sqrt{2\pi} \left( \frac{8\sqrt{m}}{\pi \sqrt{2}} \frac{T^{3/2}}{e^4 \rho \log \Lambda}\right).
\ee
We choose $N=5\times10^6$ particles, $\Delta t=0.1$, and the Nanbu-Babovsky scheme with the kernel $D^{(3)}_*$. We fix the total temperature $T = 0.07$ and we vary the initial $z$-temperature $T^0_z$ to change $\Delta T(0)$. In Figure \ref{fig:test_2_trubnikovdet} we compare the benchmark solutions with the particle approximations, for both the Maxwellian and the Coulombian cases. In the first scenario (left panel), the decreasing rate does not depend on the magnitude of the temperature difference, therefore we choose $T^0_z=0.04$ so that $\Delta T(0)=0.045$ but any other choice gives the same result. In the Coulombian case (right panel), we choose $T^0_z=0.04,\,0.001$ to have $\Delta T(0)=0.045,\,0.1035$. \textcolor{black}{We note that the higher $\Delta T(0)$, the higher the discrepancy of the numerical results with the Trubnikov solution. In fact, as pointed out in Appendix \ref{sec:appendixA}, while the Maxwellian rate \eqref{eq:trub_max} is exact, the Coulombian rate \eqref{eq:trub_coul} is approximated and requires $\Delta T(0)\ll 1$. Besides, as the analytical Trubnikov solution is approximated, even for small differences in the initial temperature there is always a disagreement with the numerical results.}

\subsection{Test 3: Maxwellian case with uncertainties} \label{sec:test3}
\begin{figure}
	\centering
	\includegraphics[width = 0.45\linewidth]{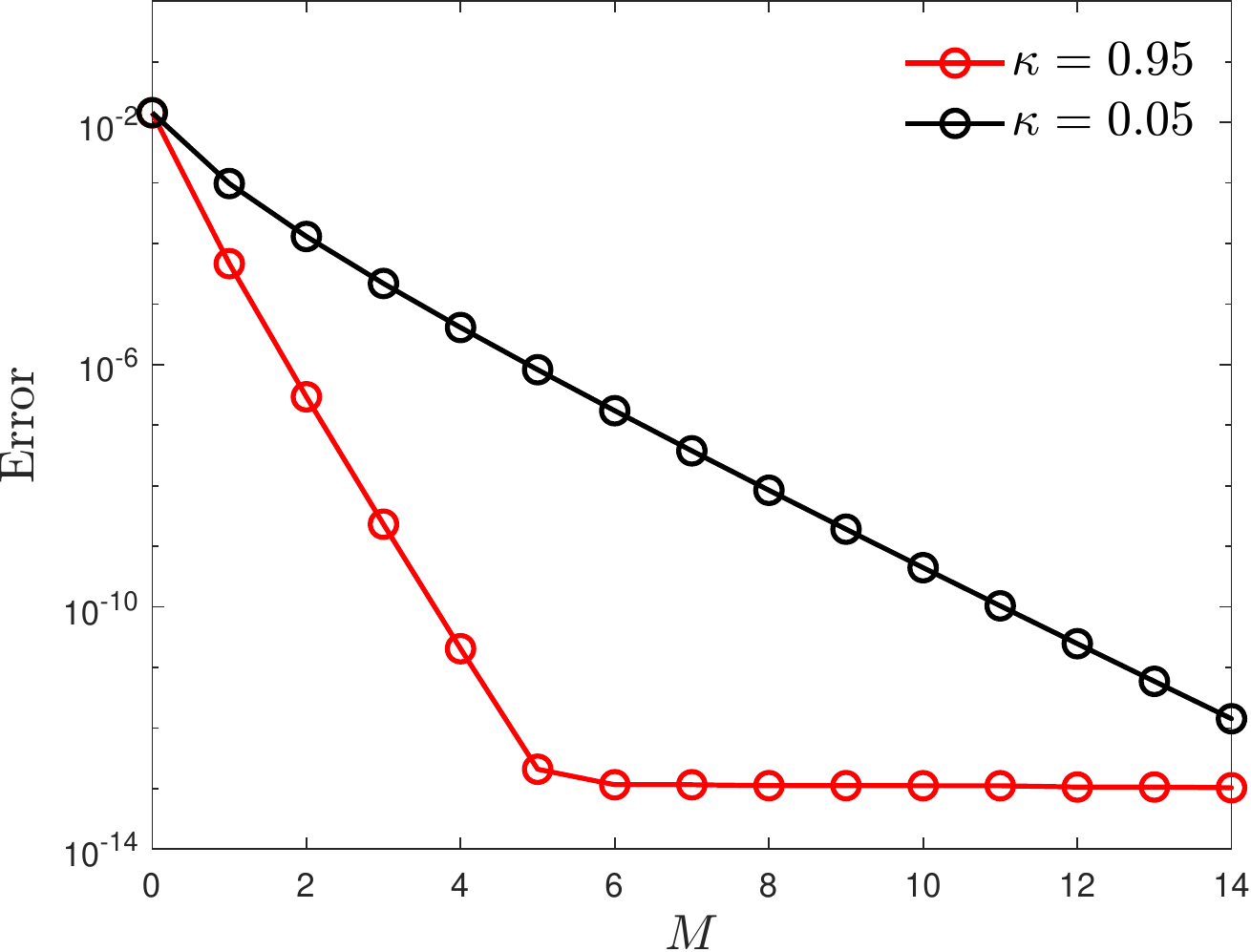}
	\includegraphics[width = 0.45\linewidth]{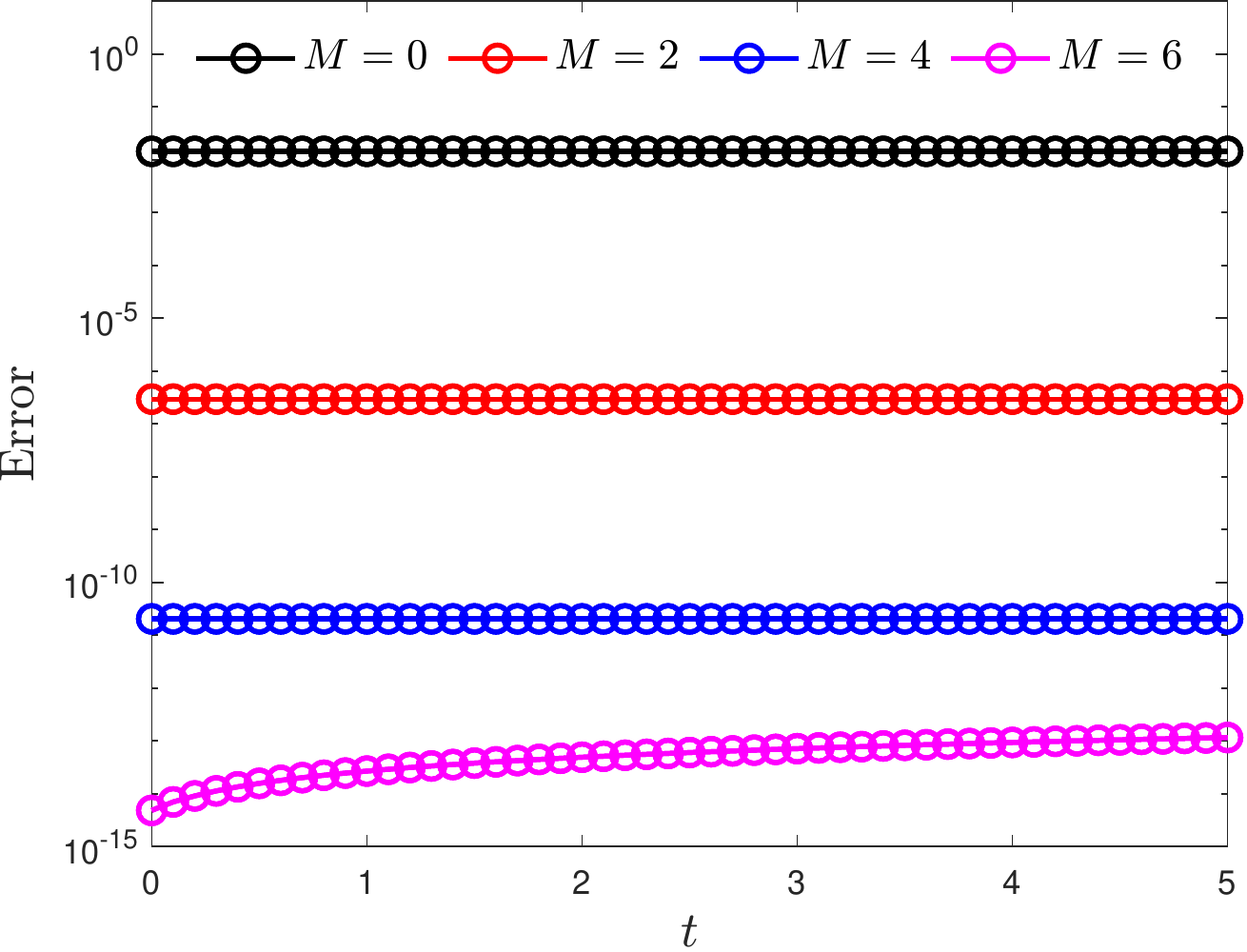}
	\caption{\small{\textbf{Test 3}. Left: $L^2$-Error in the evaluation of the temperature $T(\z)$ at fixed time $t=1$ for increasing $M$, with respect to a reference solution, for different values of $\kappa$. Right: time evolution of the same error in the time span $[0,5]$ for the case $\kappa=0.95$. We consider in both cases $N=10^6$ particles, $\Delta t=\epsilon/\rho=0.1$, and initial conditions given by \eqref{eq:initBKW}. The kernel is $D^{(3)}_*$, the scheme is Nanbu-Babovsky. Reference solution computed with $M=30$.}}
	\label{fig:test_3_sG_error_Maxwell}
\end{figure}
\begin{figure}
	\centering
	\includegraphics[width = 0.3\linewidth]{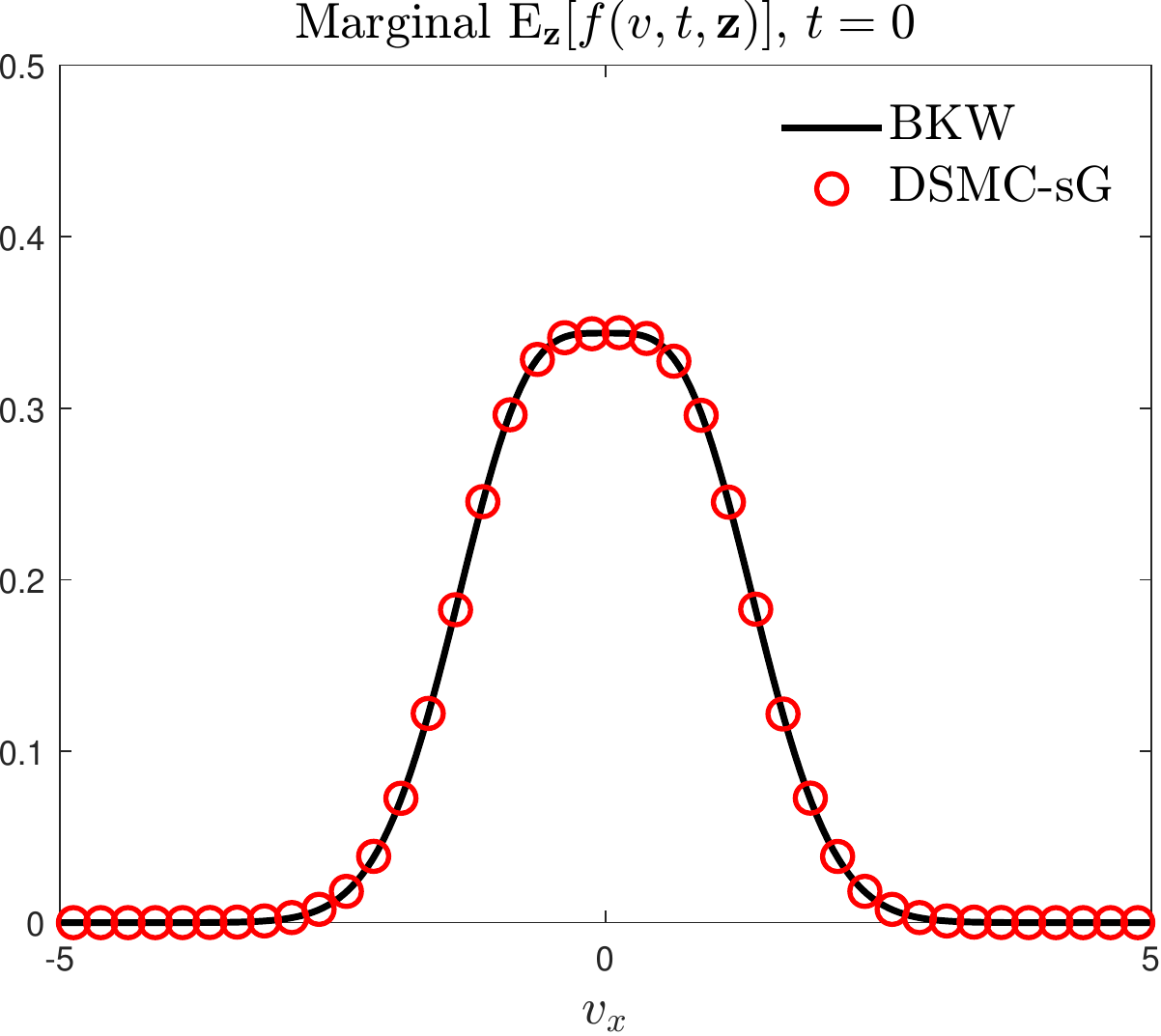}
	\includegraphics[width = 0.3\linewidth]{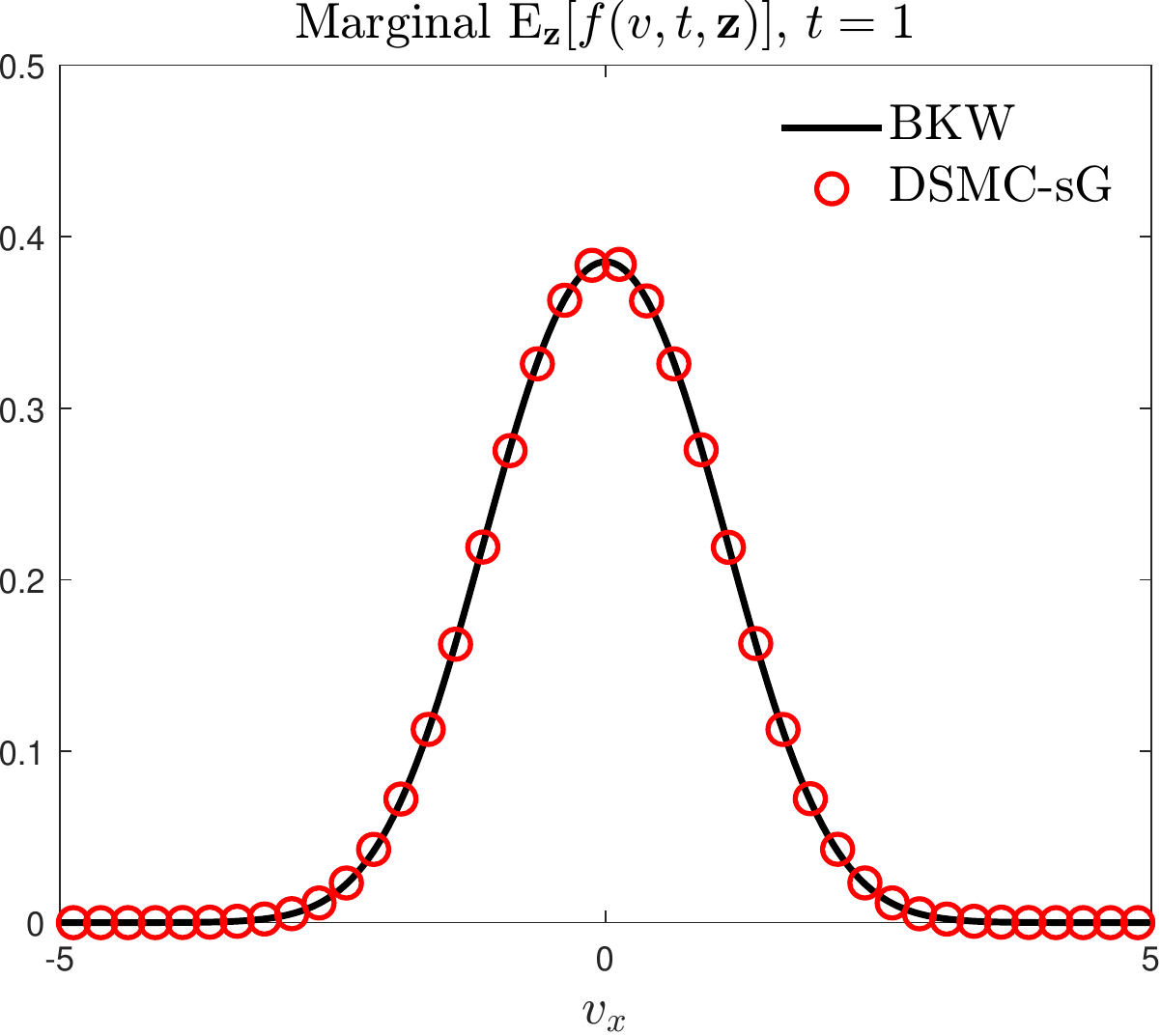}
	\includegraphics[width = 0.3\linewidth]{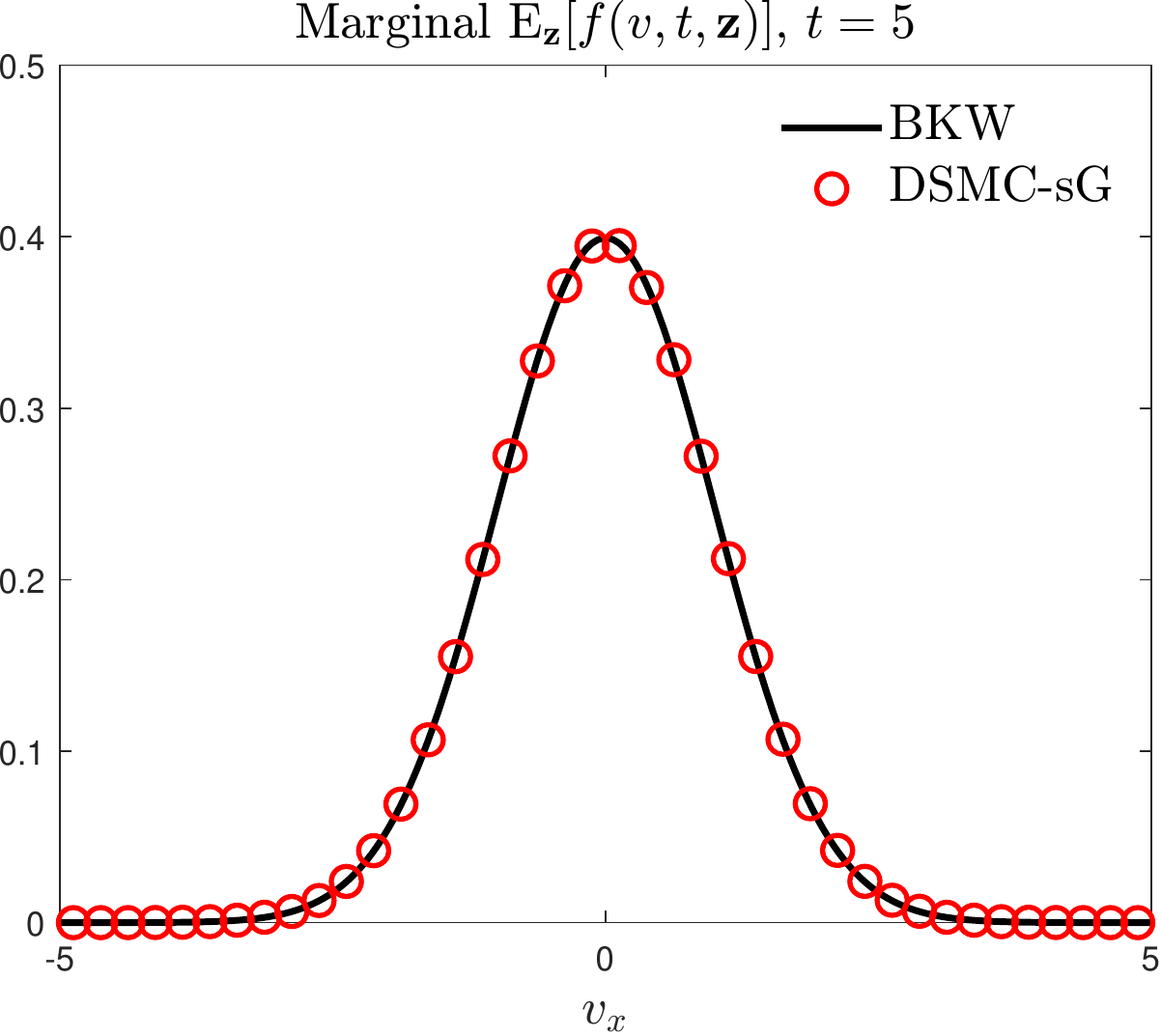}
	\includegraphics[width = 0.3\linewidth]{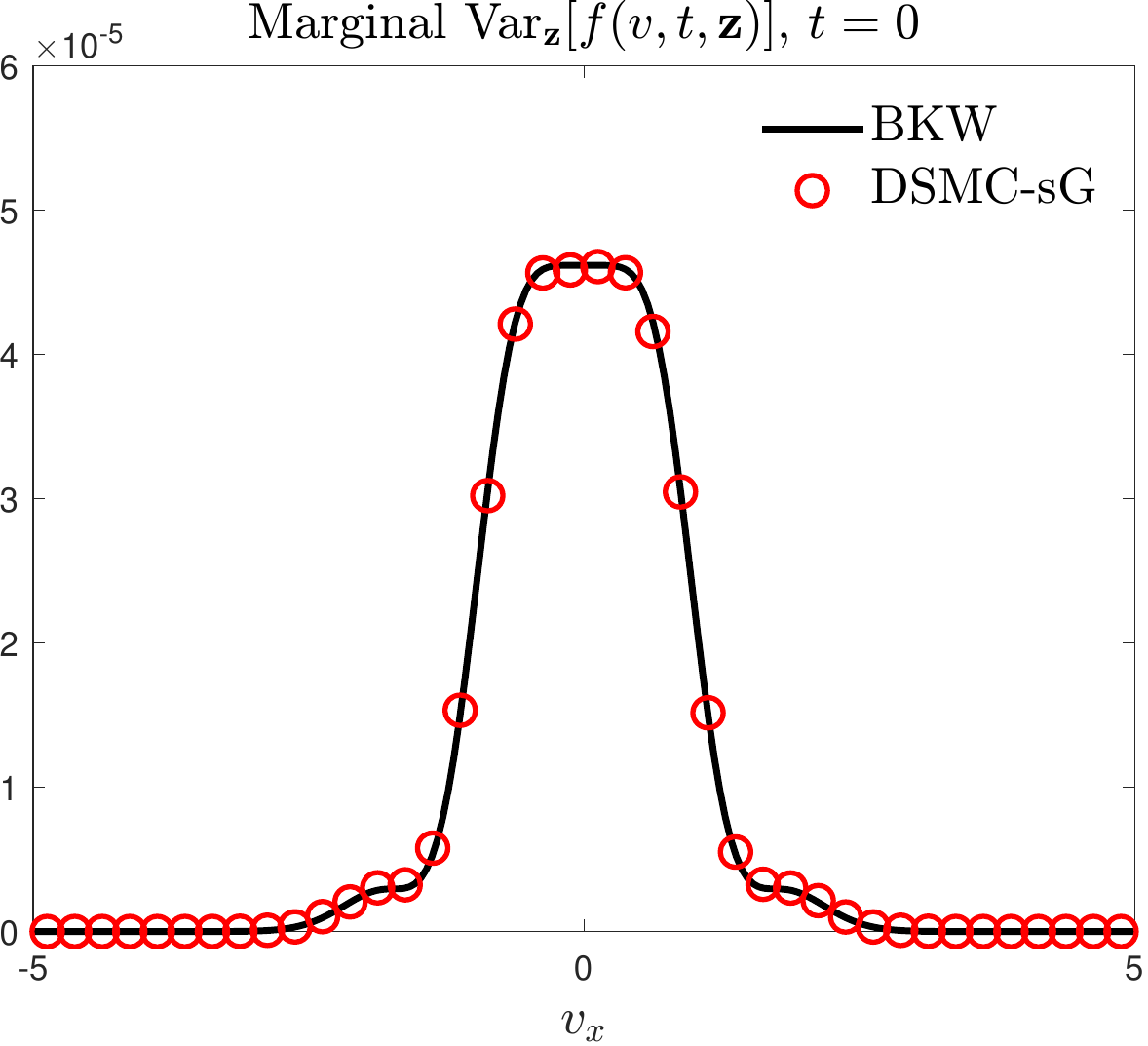}
	\includegraphics[width = 0.3\linewidth]{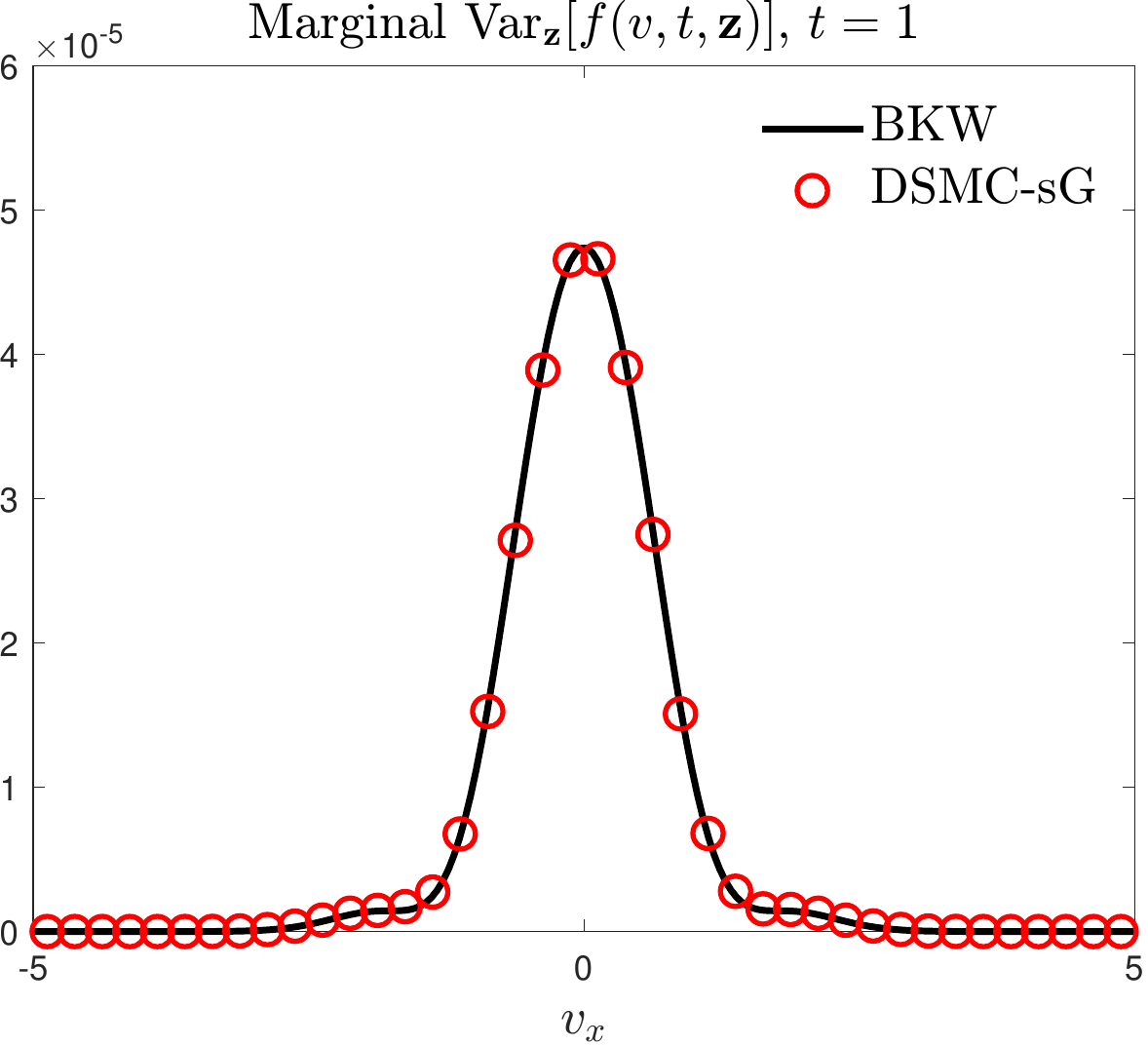}
	\includegraphics[width = 0.3\linewidth]{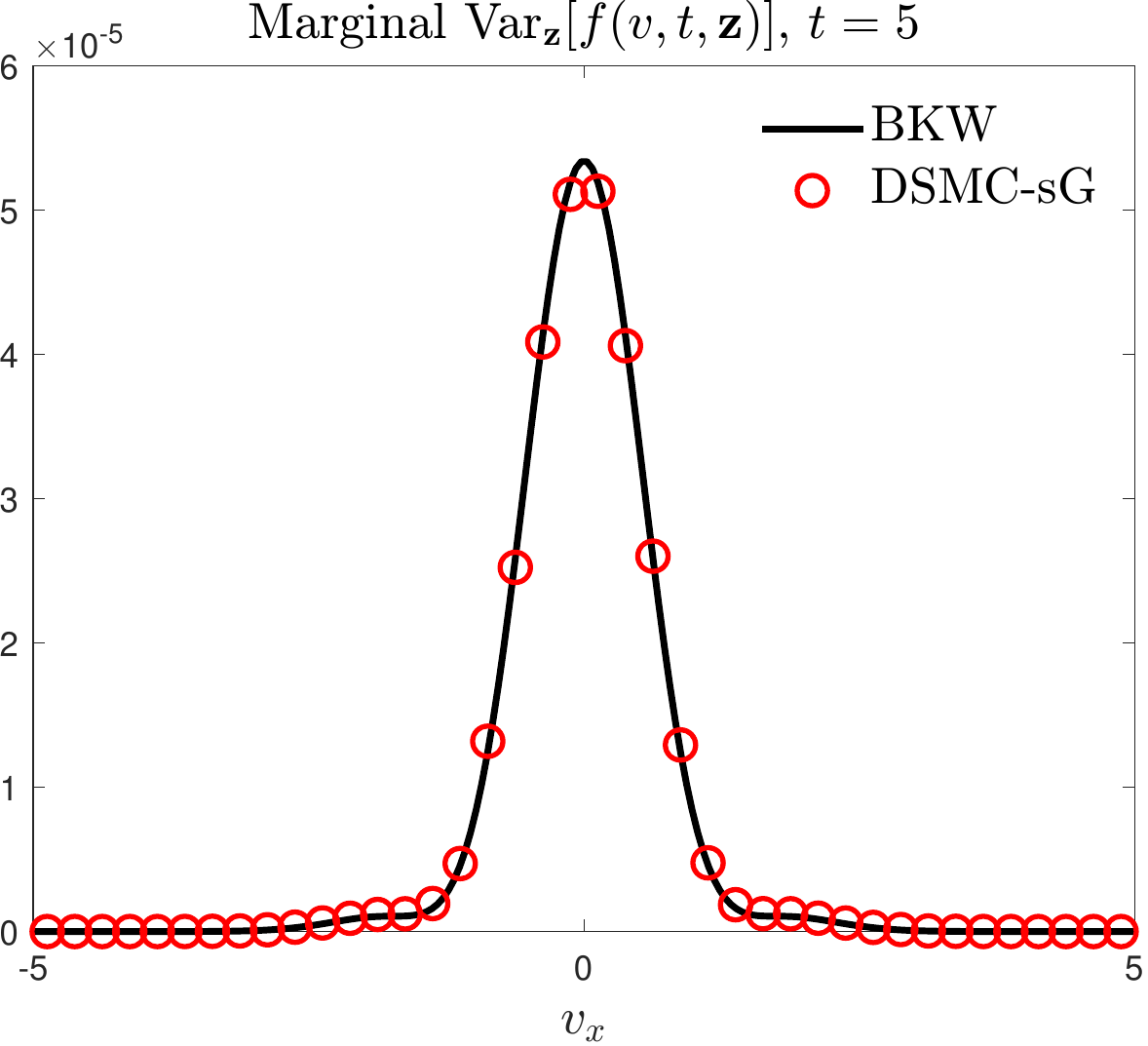}
	\caption{\small{\textbf{Test 3}. Evolution at fixed times $t=0,1,5$ of the marginal $\mathbb{E}_{\z}[f(v,t,\z)]$ and $\mathrm{Var}_{\z}[f(v,t,\z)]$ of the BKW exact solution \eqref{eq:BKW} and of the DSMC-sG approximation of the model for Maxwell molecules, with uncertainty in the initial temperature with $\kappa=0.95$ in \eqref{eq:initBKW}. We consider $N=5\times10^7$ particles, $\Delta t=\epsilon/\rho=0.1$ and $M=5$.}}
	\label{fig:test_3_BKW}
\end{figure}
\begin{figure}
	\centering
	\includegraphics[width = 0.3\linewidth]{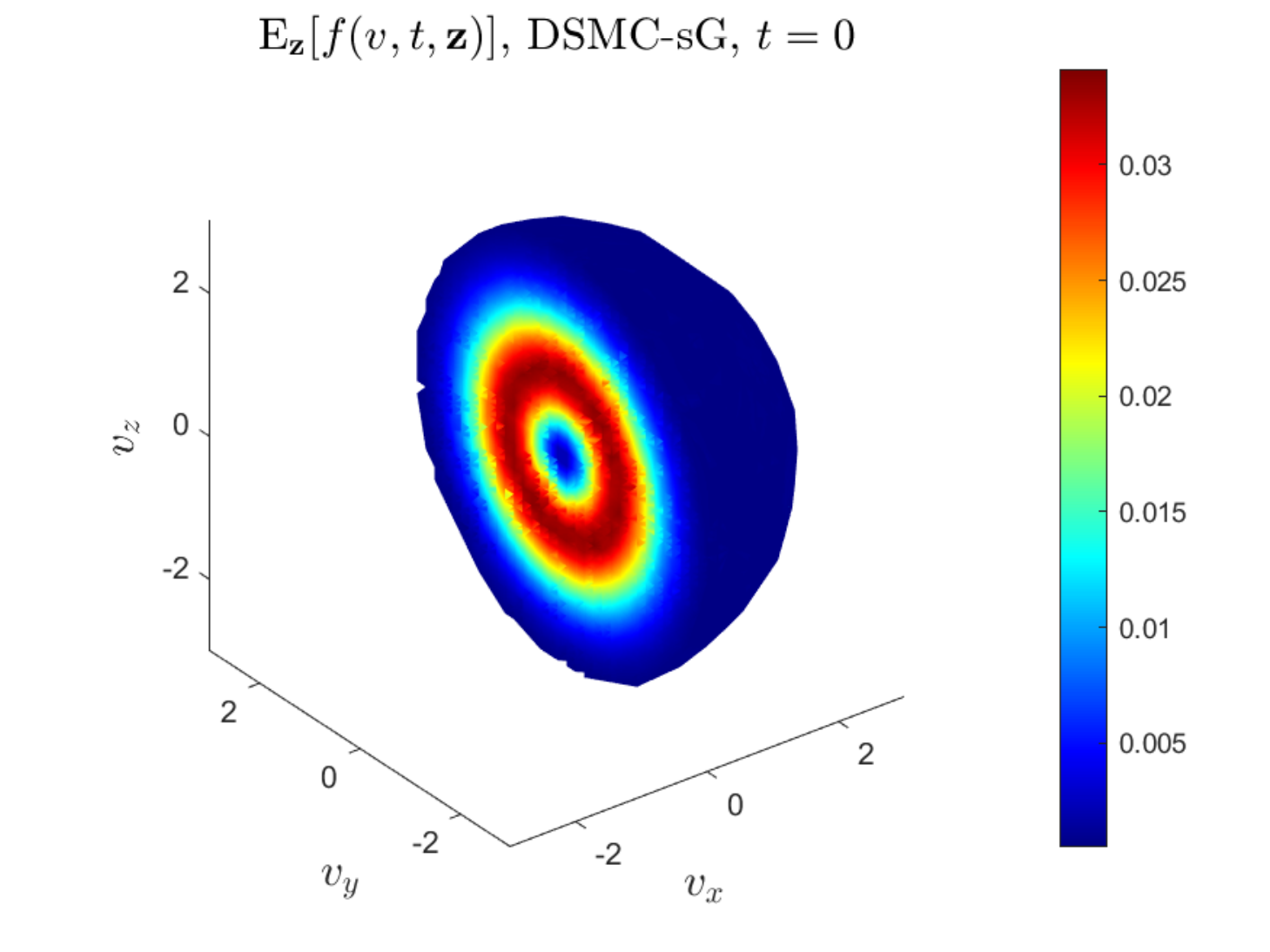}
	\includegraphics[width = 0.3\linewidth]{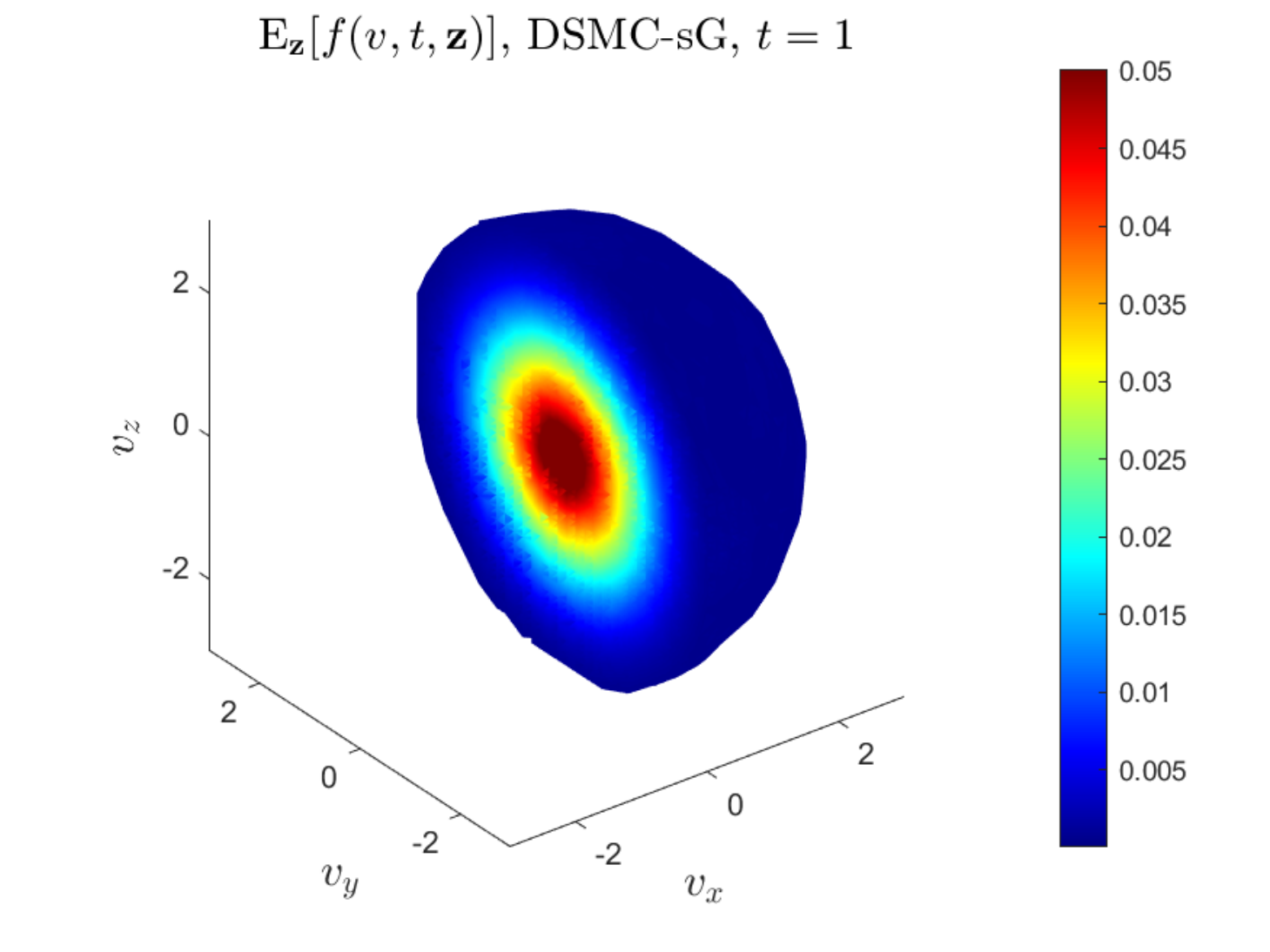}
	\includegraphics[width = 0.3\linewidth]{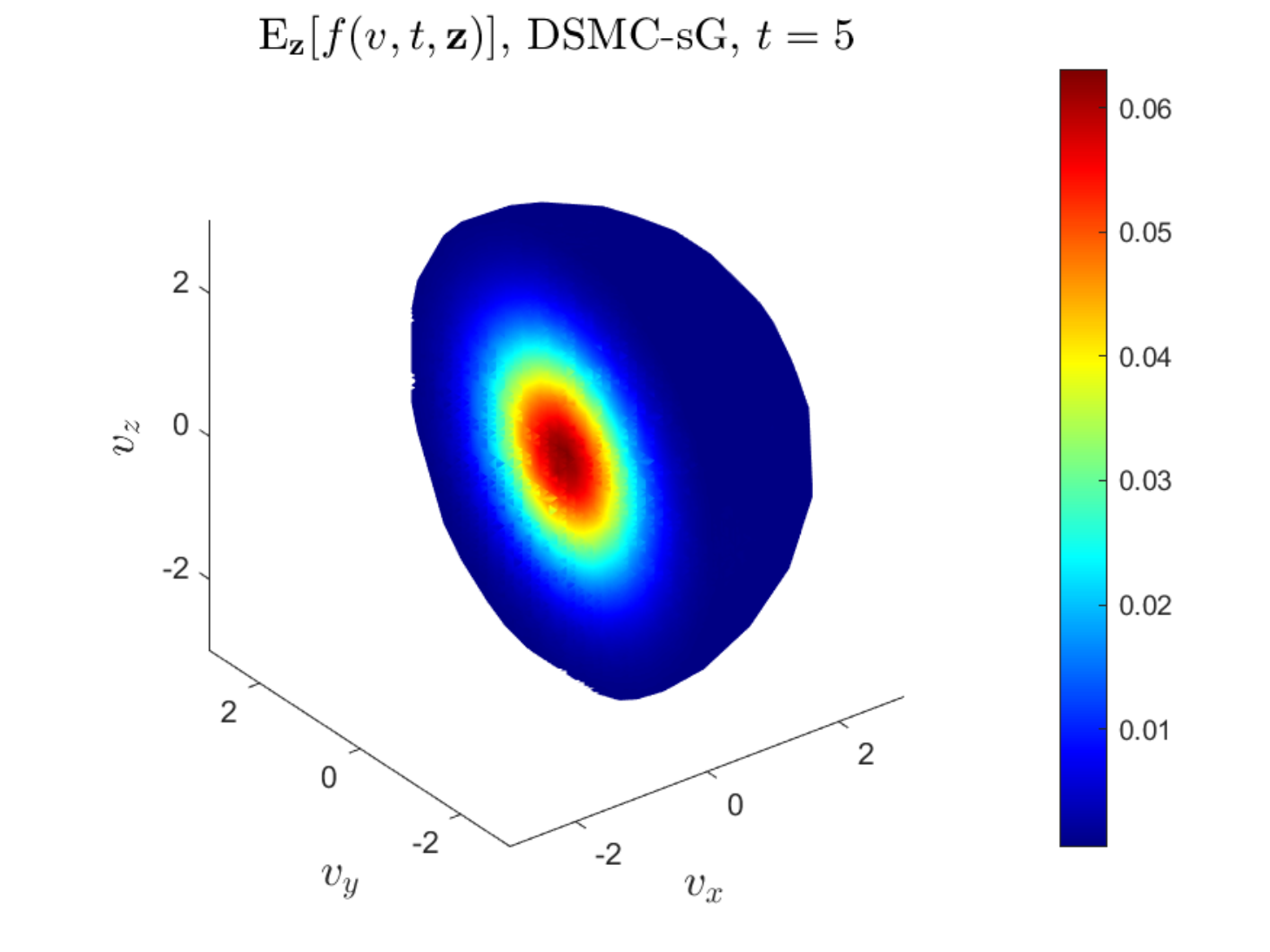}
	\includegraphics[width = 0.3\linewidth]{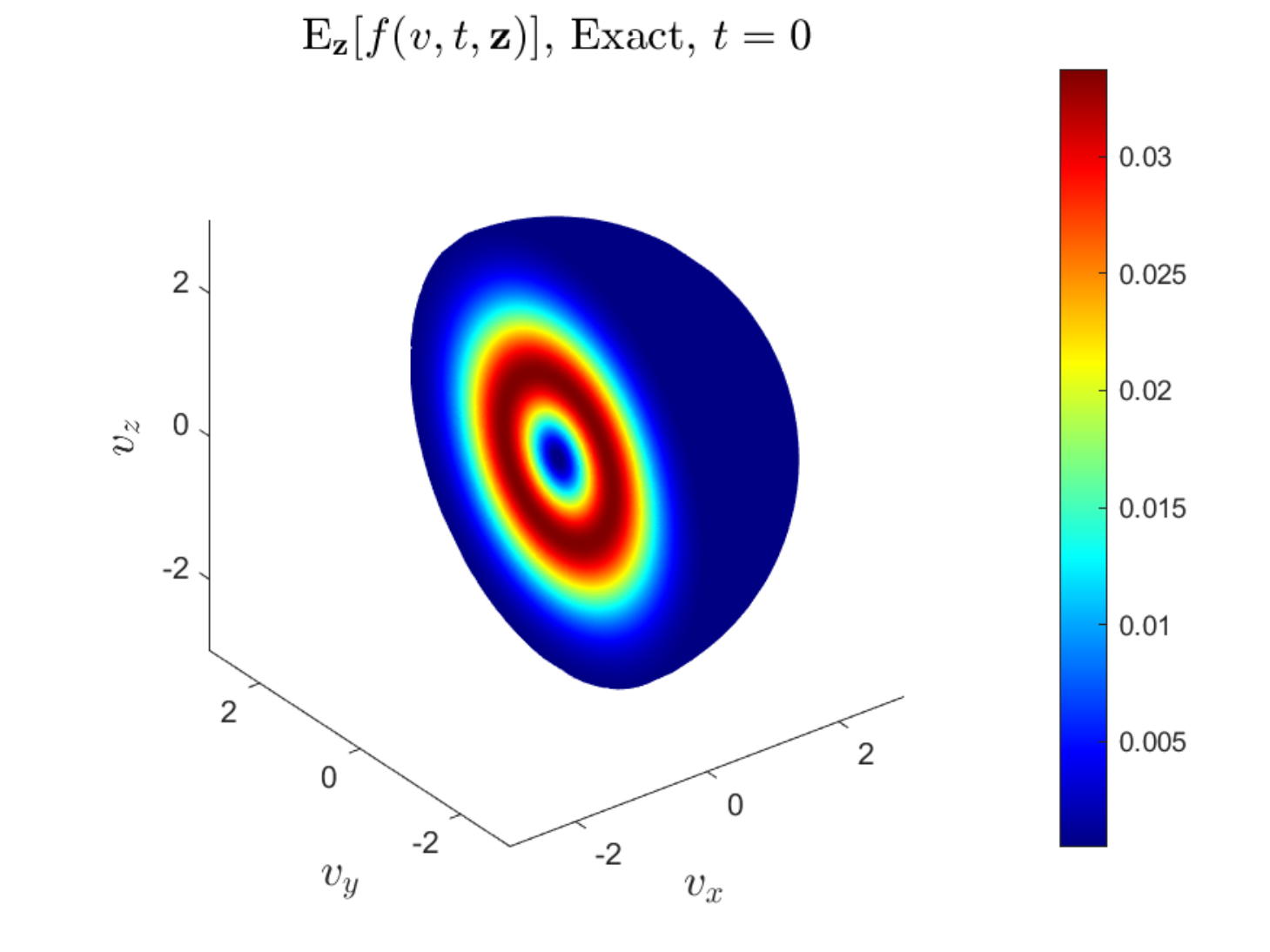}
	\includegraphics[width = 0.3\linewidth]{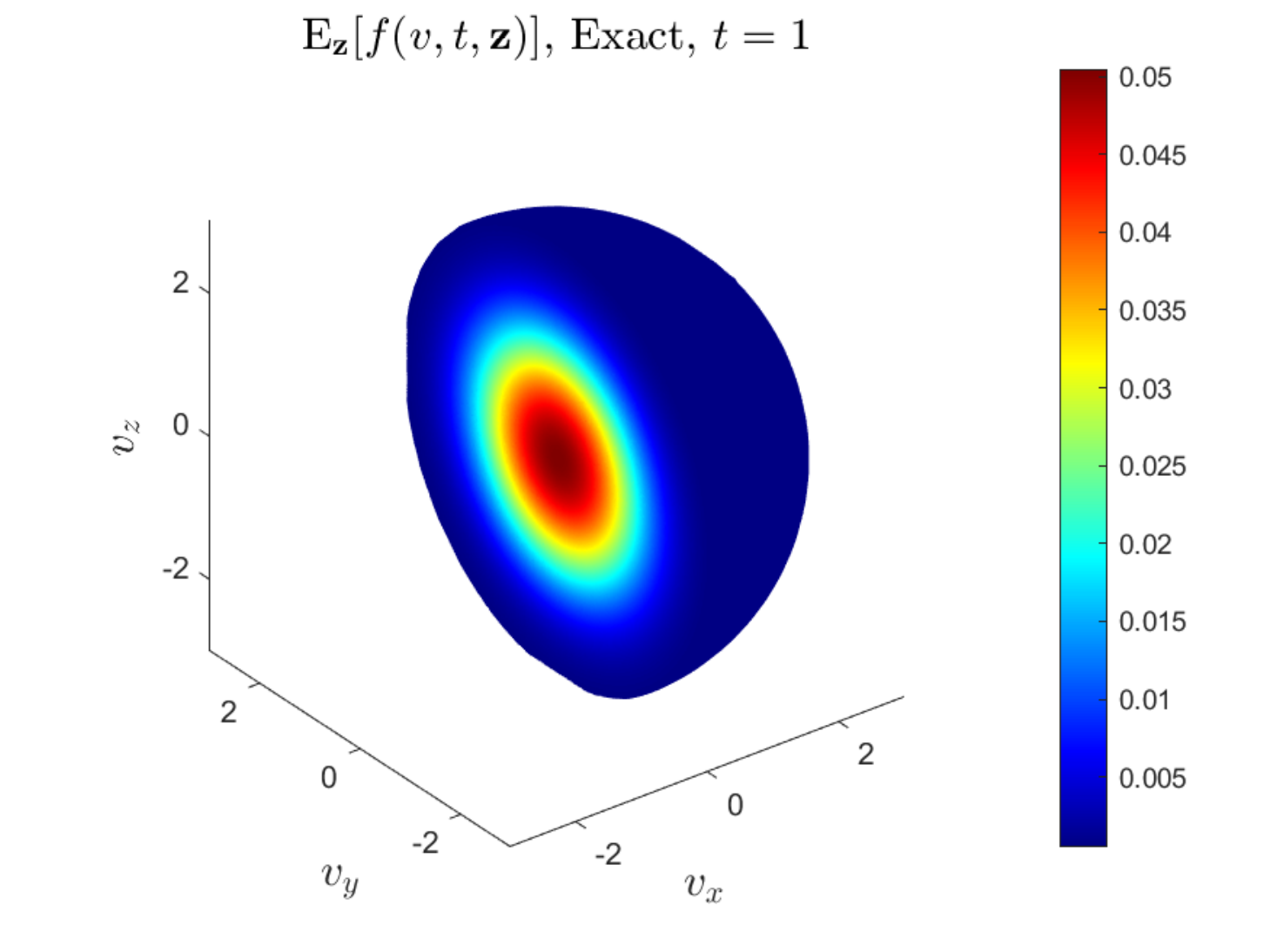}
	\includegraphics[width = 0.3\linewidth]{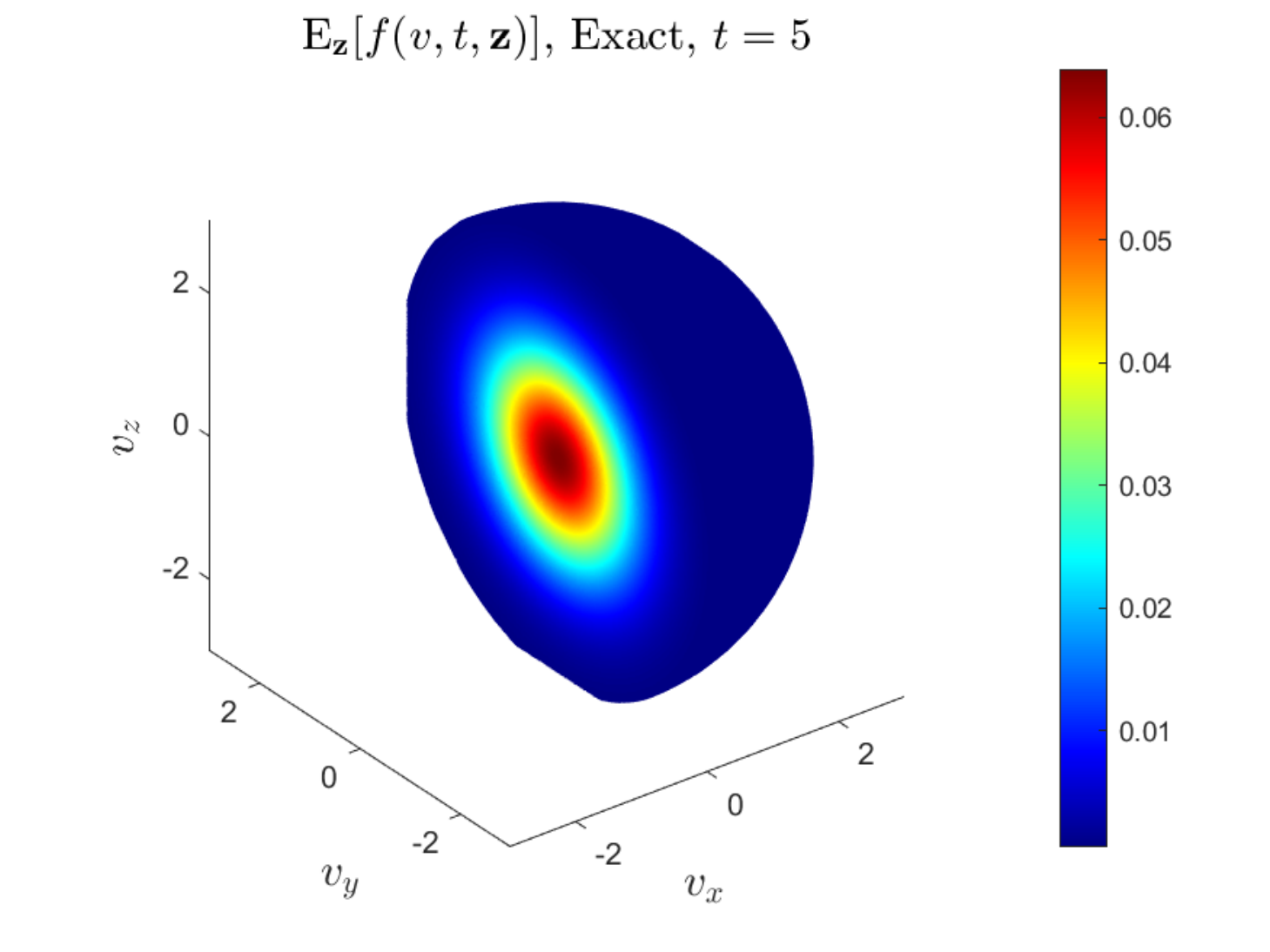}
	\caption{\small{\textbf{Test 3}. Slices of the expected distribution $\mathbb{E}_{\z}[f(v,t,\z)]$ at fixed times $t=0,1,5$, for the Maxwellian collision model with uncertainties. Upper row: DSMC-sG solution obtained with $N=5\times10^7$ particles, $\Delta t=\epsilon/\rho=0.1$, $M=5$, $D^{(3)}_*$, and the Nanbu-Babovsky scheme. Lower row: exact BKW solution. }}
	\label{fig:test_3_BKW_f3D}
\end{figure}
\begin{figure}
	\centering
	\includegraphics[width = 0.3\linewidth]{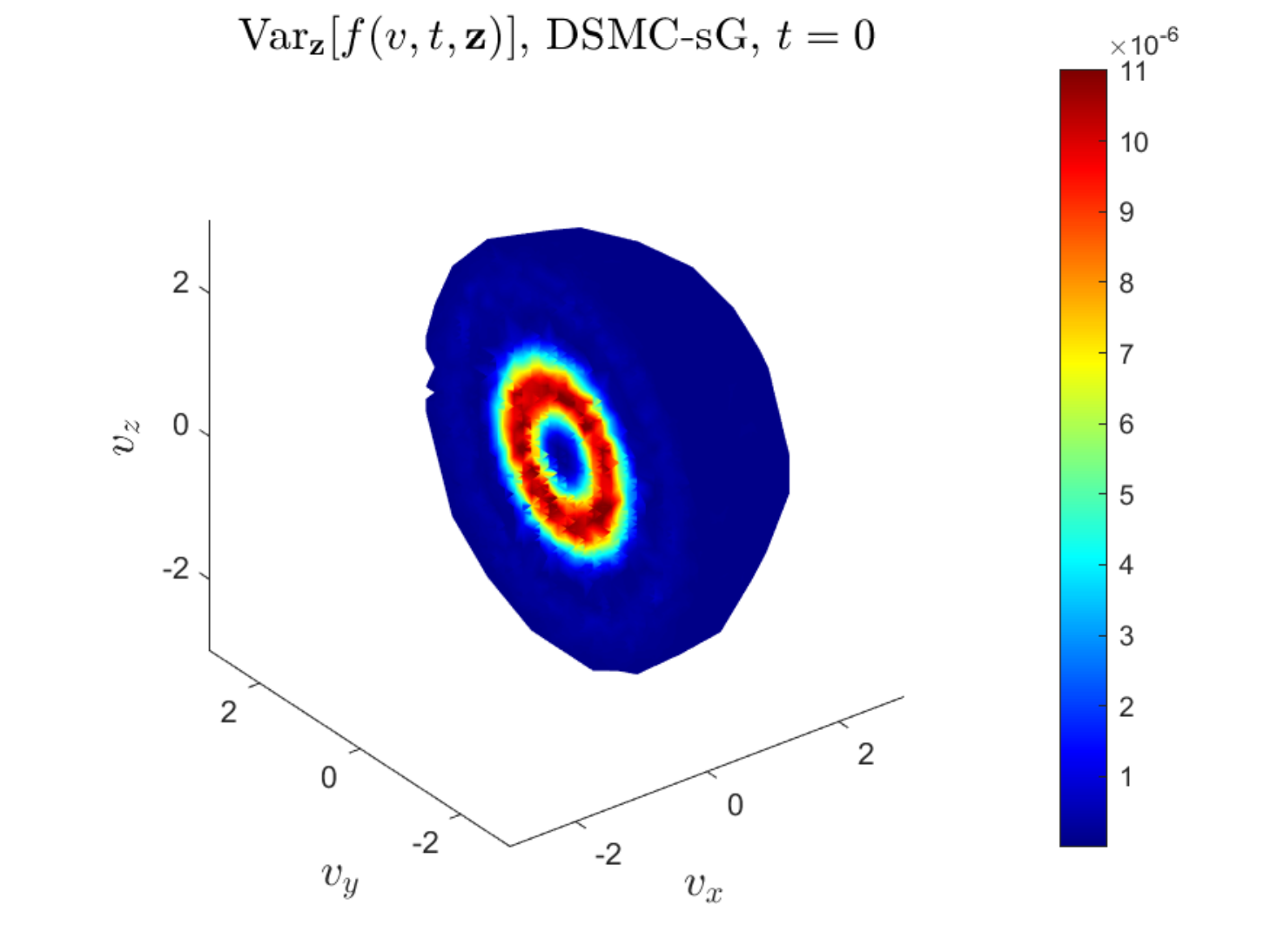}
	\includegraphics[width = 0.3\linewidth]{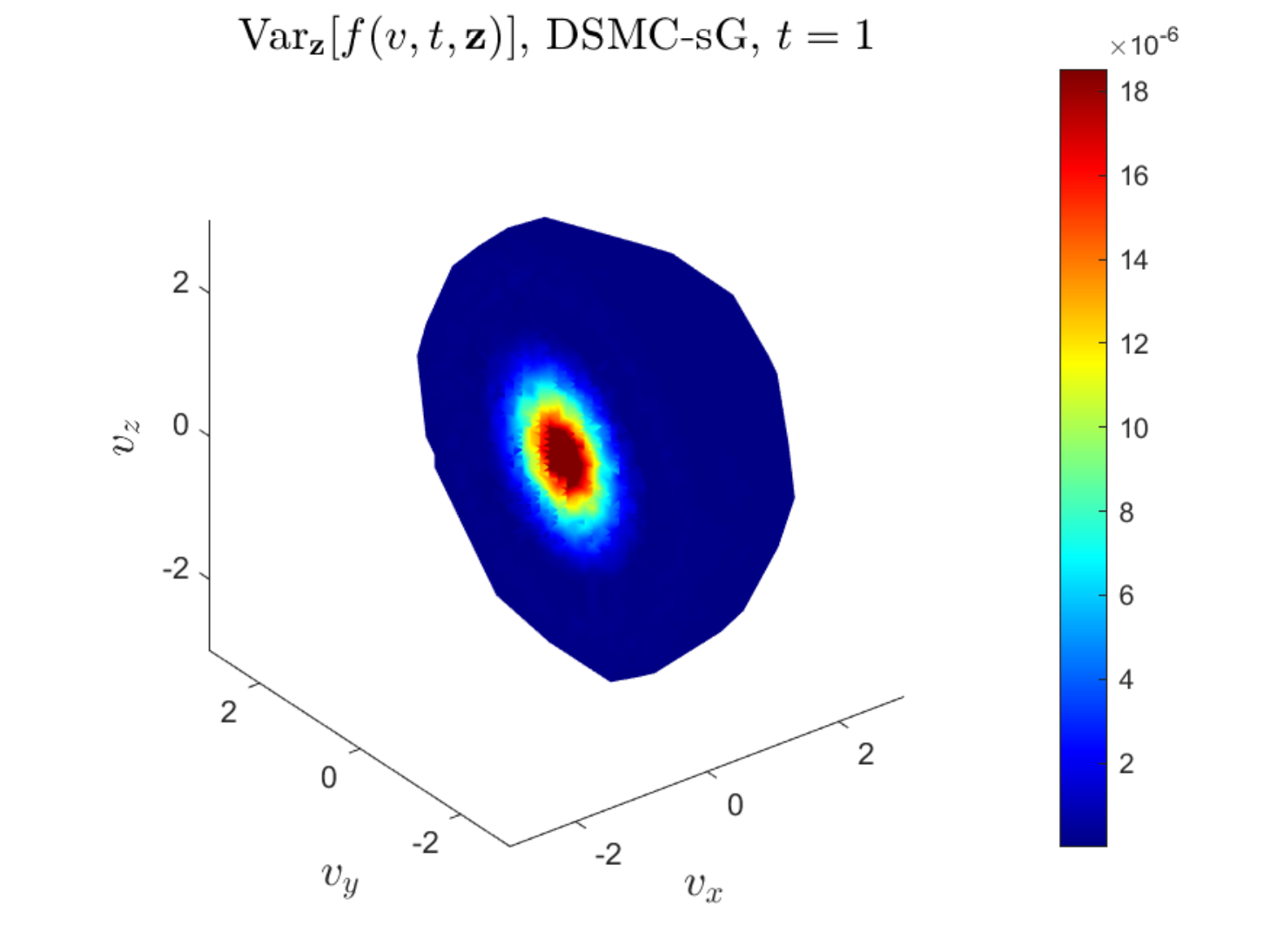}
	\includegraphics[width = 0.3\linewidth]{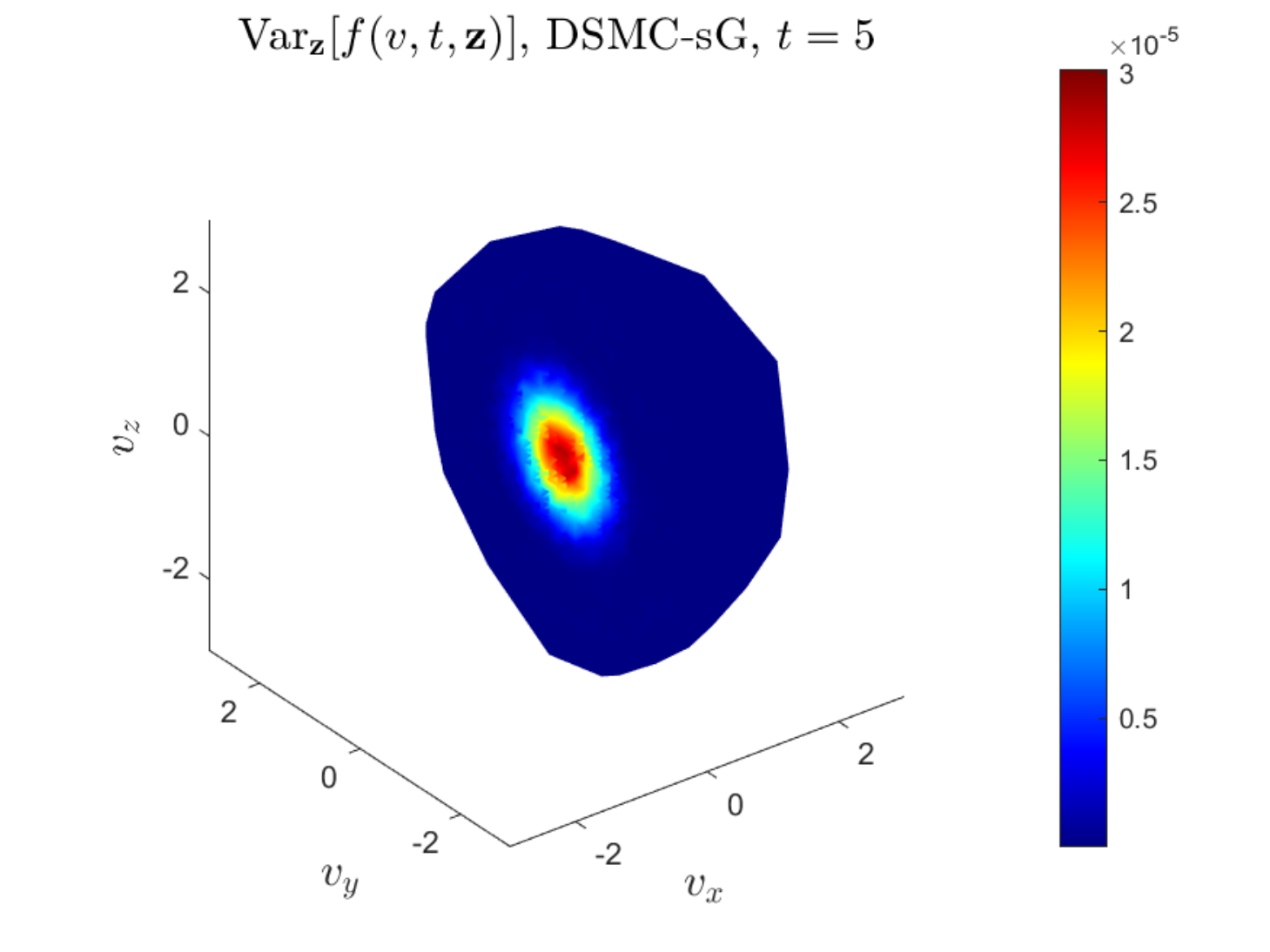}
	\includegraphics[width = 0.3\linewidth]{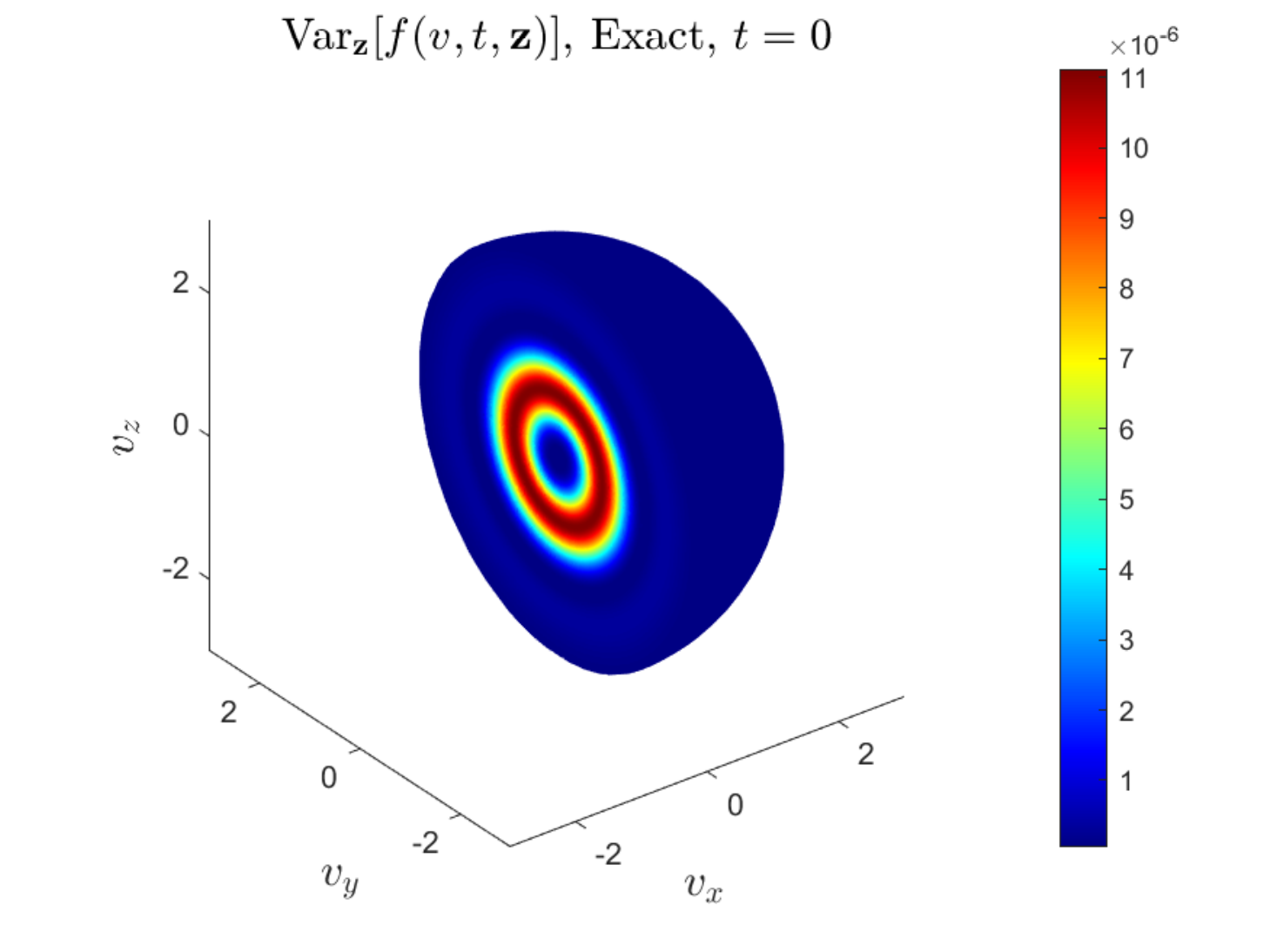}
	\includegraphics[width = 0.3\linewidth]{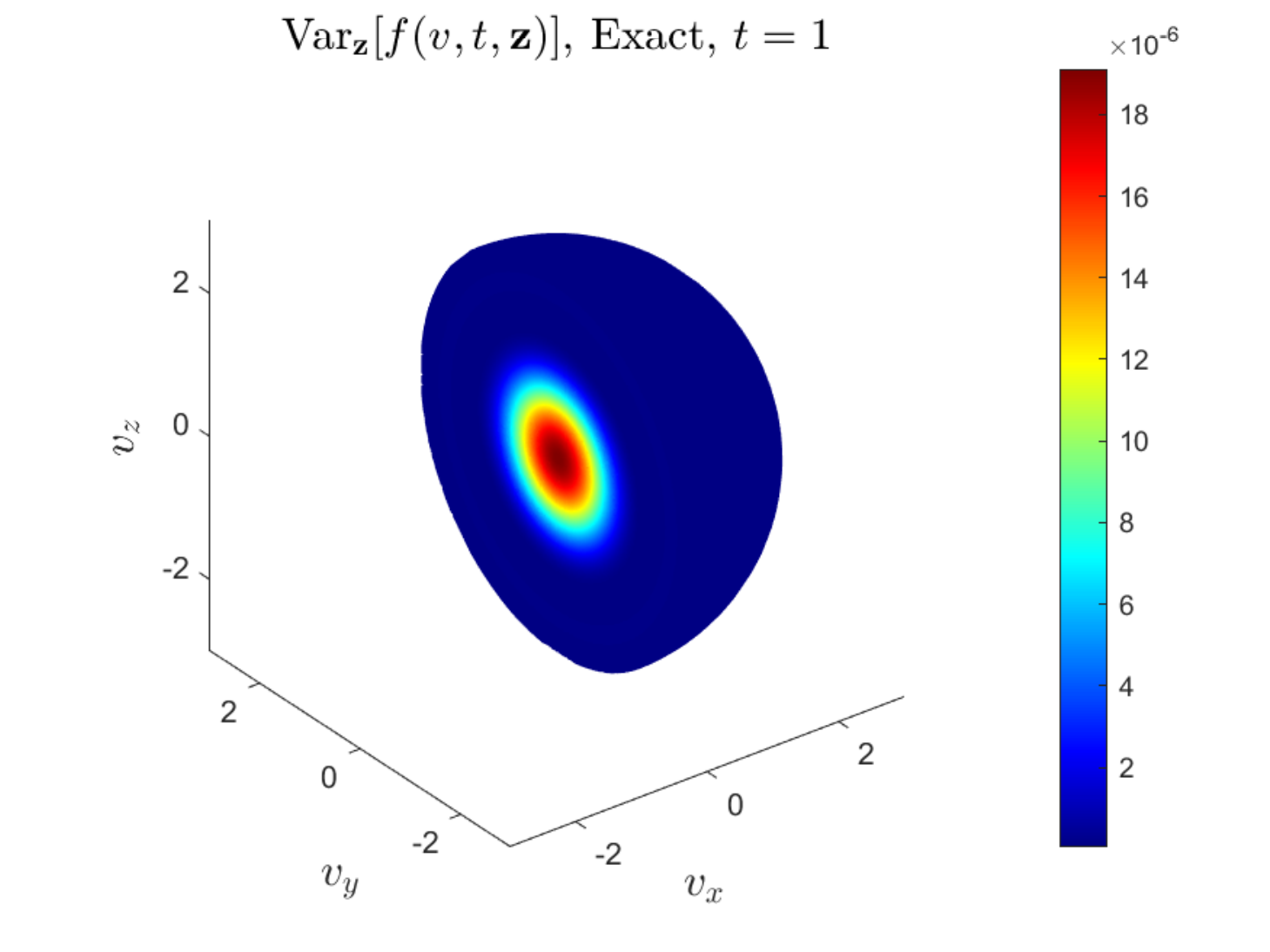}
	\includegraphics[width = 0.3\linewidth]{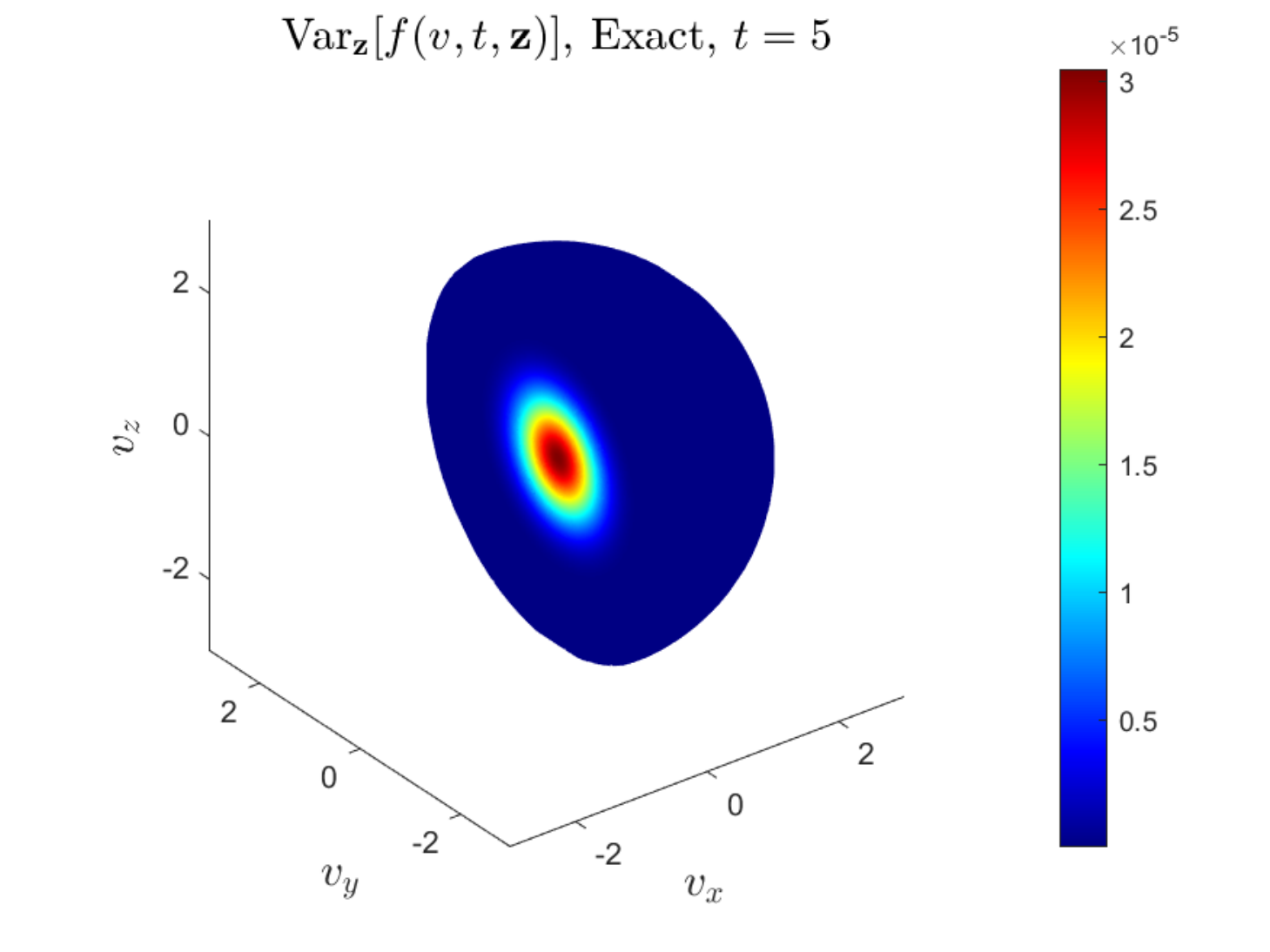}
	\caption{\small{\textbf{Test 3}. Slices of the variance $\mathrm{Var}_{\z}[f(v,t,\z)]$ at fixed times $t=0,1,5$, for the Maxwellian collision model with uncertainties. Upper row: DSMC-sG solution obtained with $N=5\times10^7$ particles, $\Delta t=\epsilon/\rho=0.1$, $M=5$, $D^{(3)}_*$, and the Nanbu-Babovsky scheme. Lower row: exact BKW solution. }}
	\label{fig:test_3_BKW-Var3D}
\end{figure}
We consider here the model with Maxwell molecules and uncertainties in the initial temperature $T(\z)$, i.e.,
\be\label{eq:BKW}
f(v,t,\z) = \frac{1}{(2\pi K(t,\z))^{3/2}} e^{-\frac{|v|^2}{2K(t,\z)}} \left(\frac{5K(t,\z) - 3 T(\z)}{2K(t,\z)} + \frac{T(\z)-K(t,\z)}{2K(t,\z)^2}|v|^2\right),
\ee
with
\[
\textcolor{black}{K(t,\z) = T(\z) \left( 1 - \frac{2}{5}e^{-t/2} \right),}
\]
and fourth order moment given by
\be
\textrm{M}4(t) = 9 K(t,\z) (2T(\z) -K(t,\z)).
\ee

We initialize the particles in the DSMC-sG scheme by sampling the initial conditions given by
\be\label{eq:initBKW}
f^0(v,\z)=f(v,t=0,\z) \quad \textrm{with} \quad T(\z) = \kappa + 0.1 \z, \quad \z\sim\mathcal{U}([0,1]).
\ee
In all the tests, we adopt the Nanbu-Babovsky scheme given by Algorithm \ref{NB_unc} and the kernel $D_*^{(3)}$, with $\Delta t = \epsilon/\rho=0.1$. 

In Figure \ref{fig:test_3_sG_error_Maxwell} we show the stochastic Galerkin error for increasing values of $M$ (left) and its time evolution at fixed $M$ (right). We compute a reference solution with $M=30$ and we store the collisional tree. Then, for different values of $M$, we compute the $L^2$-Error of the temperature $T(\z)$ with respect to the reference solution. We notice that the machine accuracy is reached with a finite number of modes, and the error is constant in time. 

In Figures \ref{fig:test_3_BKW}-\ref{fig:test_3_BKW_f3D}-\ref{fig:test_3_BKW-Var3D} we may see the comparison between the BKW exact solution given by \eqref{eq:BKW} and the DSMC-sG approximation. We choose $N=5\times 10^7$ particles and a stochastic Galerkin expansion up to order $M=5$. In particular, in Figure \ref{fig:test_3_BKW} we display at fixed times $t=0,1,5$ the marginals of $\mathbb{E}_{\z}[f(v,t,\z)]$ and $\mathrm{Var}_{\z}[f(v,t,\z)]$. We may notice the accordance between the DSMC-sG approximation (red circles) and the exact BKW solution (solid black lines). In Figures \ref{fig:test_3_BKW_f3D}-\ref{fig:test_3_BKW-Var3D} the show at the fixed times $t=0,1,5$ the slices of the three dimensional plots of $\mathbb{E}_{\z}[f(v,t,\z)]$ and $\mathrm{Var}_{\z}[f(v,t,\z)]$ respectively. In both the figures, the top row is the DSMC-sG approximation, while the bottom row is the exact BKW solution. We observe again a good agreement.

\subsection{Test 4: Coulomb case with uncertainties}
We consider now the Coulomb case with uncertainties, i.e., scattering cross section given by \eqref{eq:coulomb} and $\gamma=-3$ in $\tau_0$, that is, \eqref{eq:tau0_c}.
\paragraph{Stochastic Galerkin error}
\begin{figure}
	\centering
	\includegraphics[width = 0.5\linewidth]{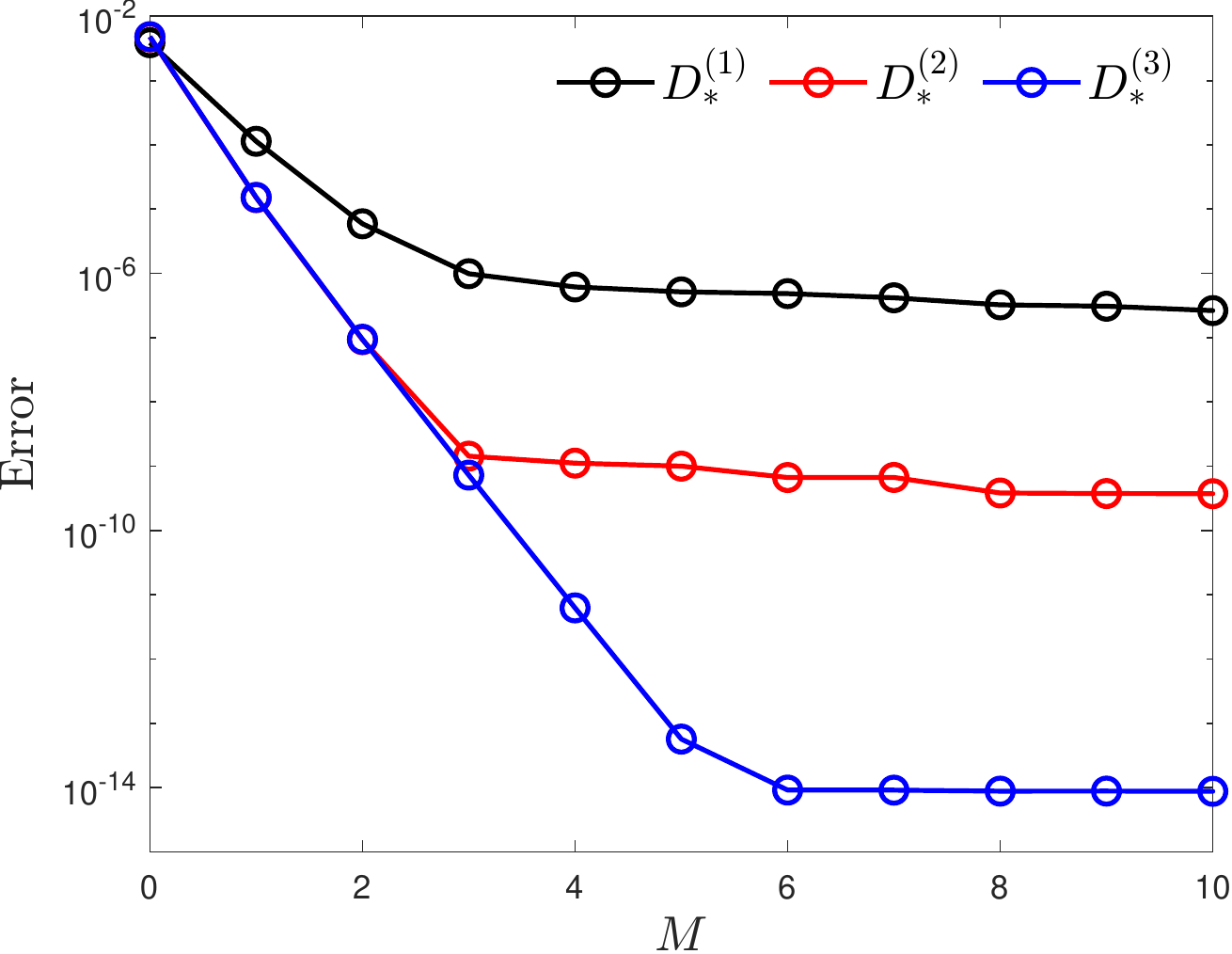}
	\caption{\small{\textbf{Test 4}. Comparison of the $L^2$ error in the evaluation of the temperature $T(\z)$ for different kernels, for increasing $M$, with the NB scheme. We choose $N=10^6$, $t=1$, $\Delta t=\epsilon/\rho=0.1$ and the reference solution is computed with an order $M=30$.}}
	\label{fig:test_4_sG_error}
\end{figure}
We are interested first in evaluating the convergence of the scheme in the space of the random parameters. To this aim, we consider the initial conditions $f_0(v,\z)=f(v,0,\z)$
\[
f_0(v,\z) = \frac{1}{\left(2\pi T(\z)\right)^{3/2}} \left(e^{-\frac{v^2_x}{2 T^0_x(\z)}}+e^{-\frac{v^2_y}{2 T^0_y}}+e^{-\frac{v^2_z}{2 T^0_z}}\right),
\]
where the total temperature, conserved in time, is defined as
\[
T(\z)=\frac{1}{3}\left(T_x(\z,t)+T_y(\z,t)+T_z(\z,t)\right)
\]
with uncertain initial conditions in the temperature along the $x$ axis
\[
\begin{split}
	& T^0_x(\z) = T_x(\z,0) = 1 + 0.05 \z \\
	& T^0_y = T_y(0) = 0.75 \\
	& \textcolor{black}{T^0_z = T_z(0) = 0.75} 
\end{split}
\]
and $\z\sim\mathcal{U}([0,1])$. We choose $\Delta t=0.1$, in a way that $\epsilon=\rho\Delta t=0.1$, and $N=10^6$ particles. We consider a reference solution obtained with $M=30$ up to time $t=1$ and we store both the initial data and the collisional tree. Then, we compute the $L^2$ error in the evaluation of the total temperature $T(\z)$ for increasing $M$. The results obtained with the Nanbu-Babovky (NB) and the Bird (B) scheme are similar, therefore we show only the NB results. 

In Figure \ref{fig:test_4_sG_error} we show the decay of the $L^2$ error in the evaluation of the temperature $T(\z)$, for the kernels $D^{(1)}_*$, $D^{(2)}_*$, and $D^{(3)}_*$. We observe that the error convergence deteriorates for $D^{(1)}_*$, since we need to introduce a cut-off for the numerical resolution of the nonlinear equation \eqref{eq:nonlineq}, and for $D^{(2)}_*$, because the function $\nu(\tau_0)$ is not differentiable for $\tau_0=1$. With the choice $D^{(3)}_*$, we reach the machine accuracy with a small order $M$. 

\paragraph{Trubnikov test}
\begin{figure}
	\centering
	\includegraphics[width = 0.3\linewidth]{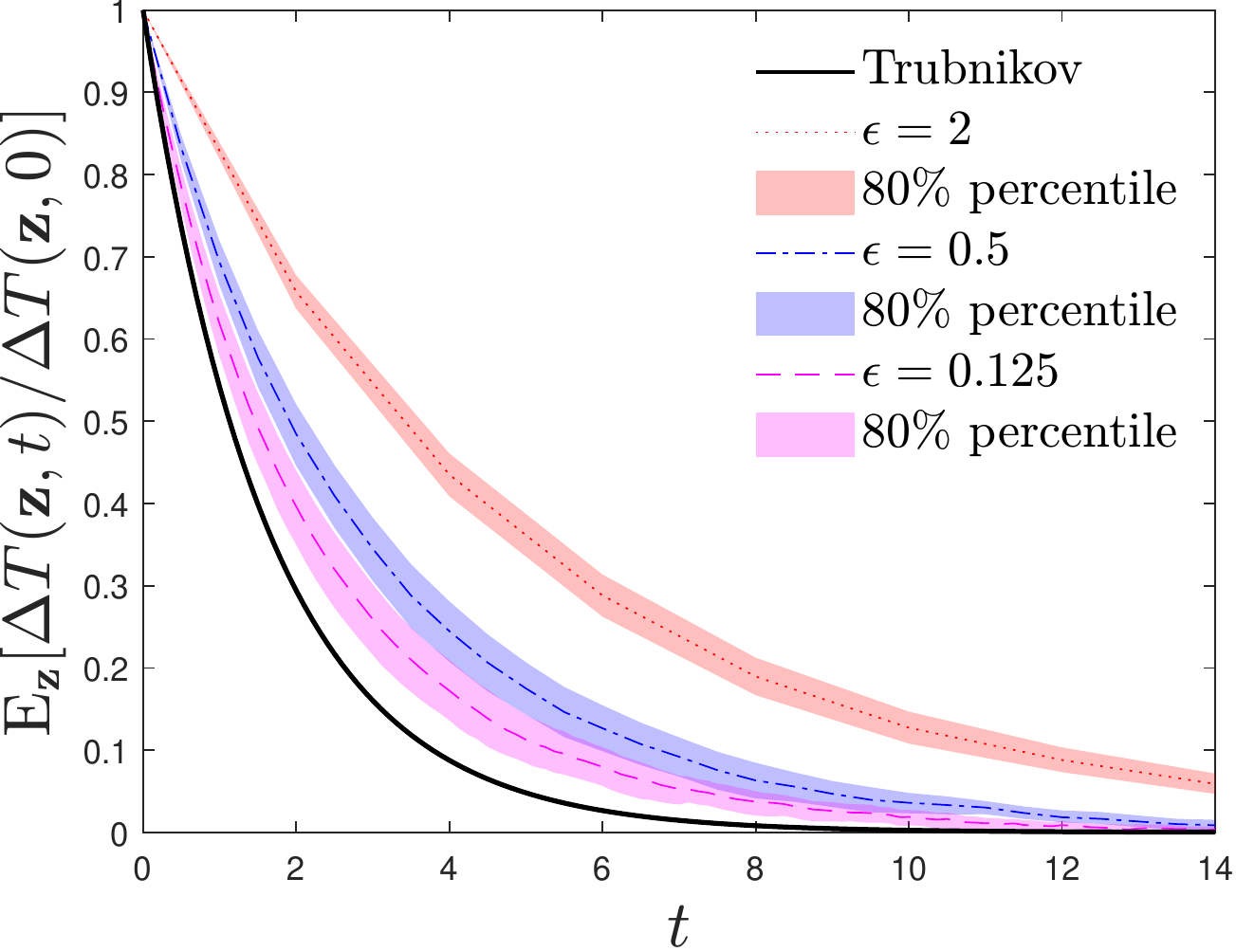}
	\includegraphics[width = 0.3\linewidth]{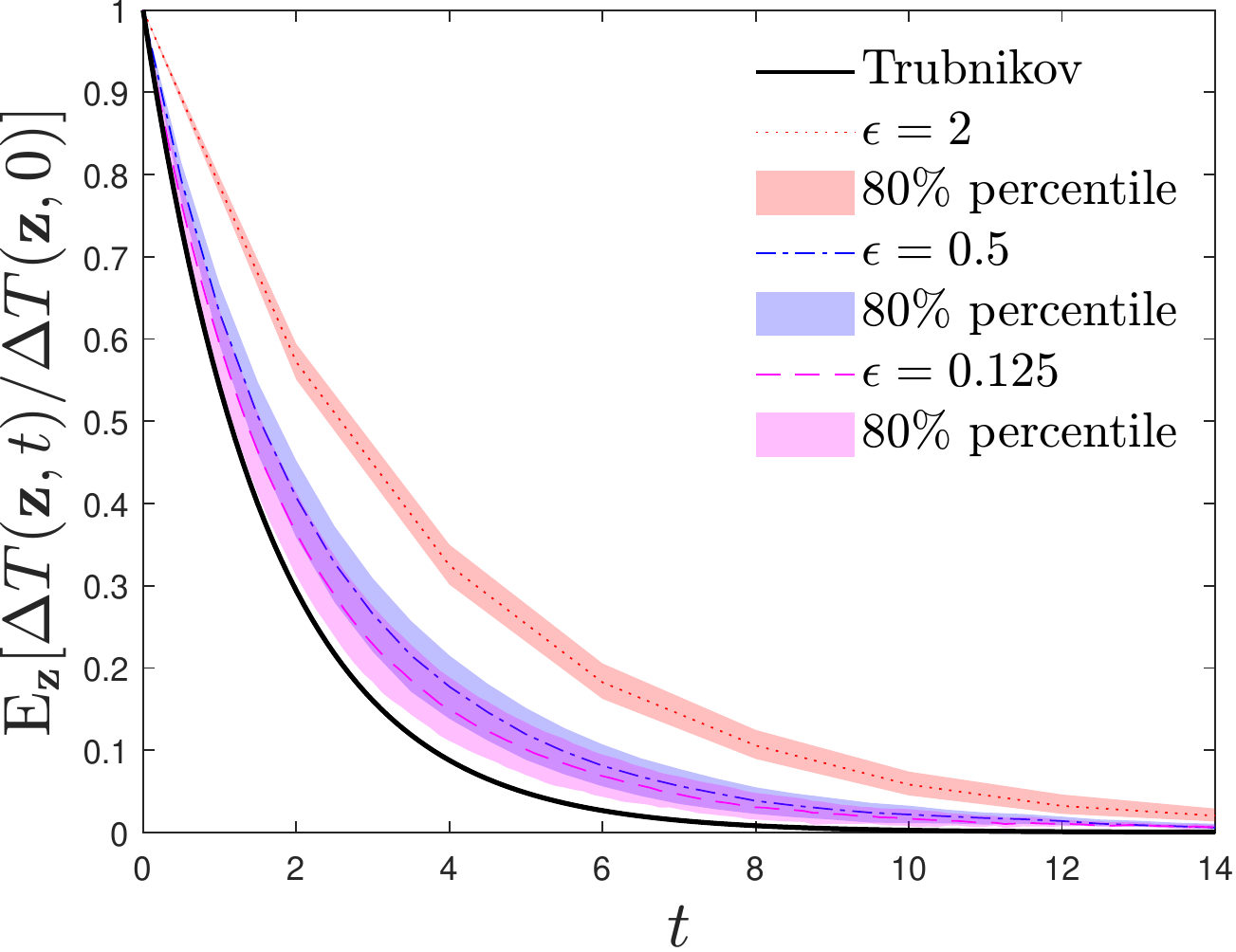}
	\includegraphics[width = 0.3\linewidth]{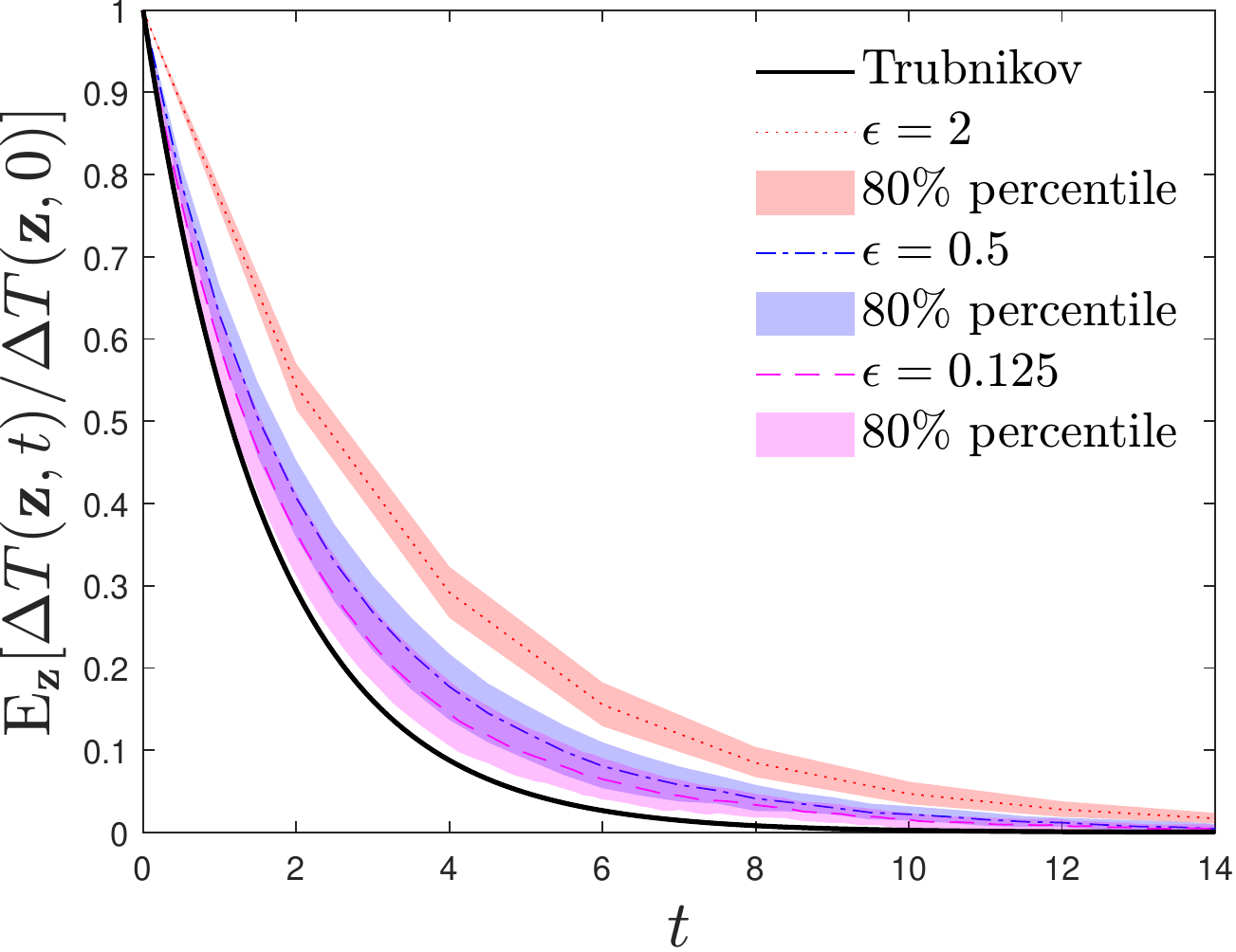}
	\includegraphics[width = 0.3\linewidth]{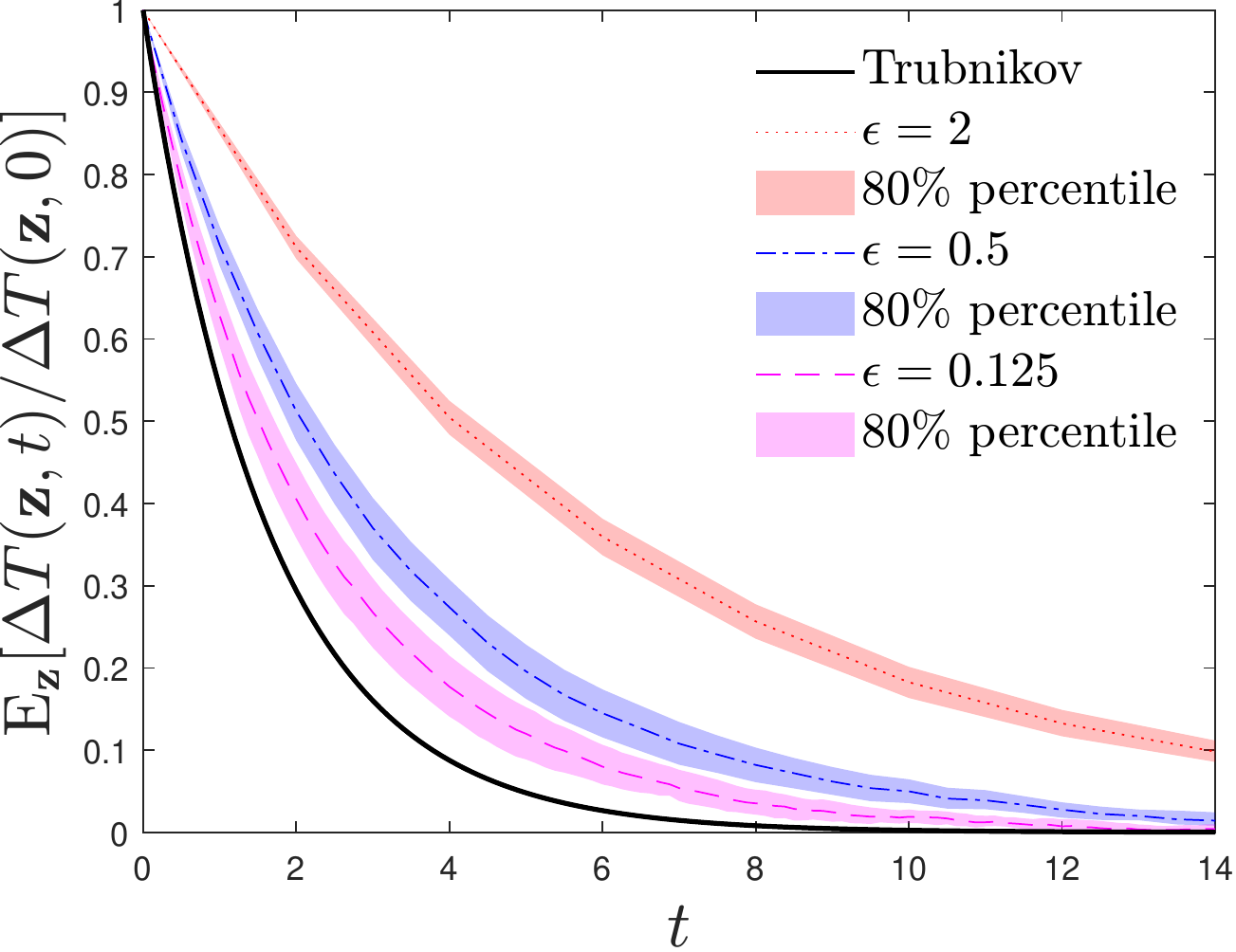}
	\includegraphics[width = 0.3\linewidth]{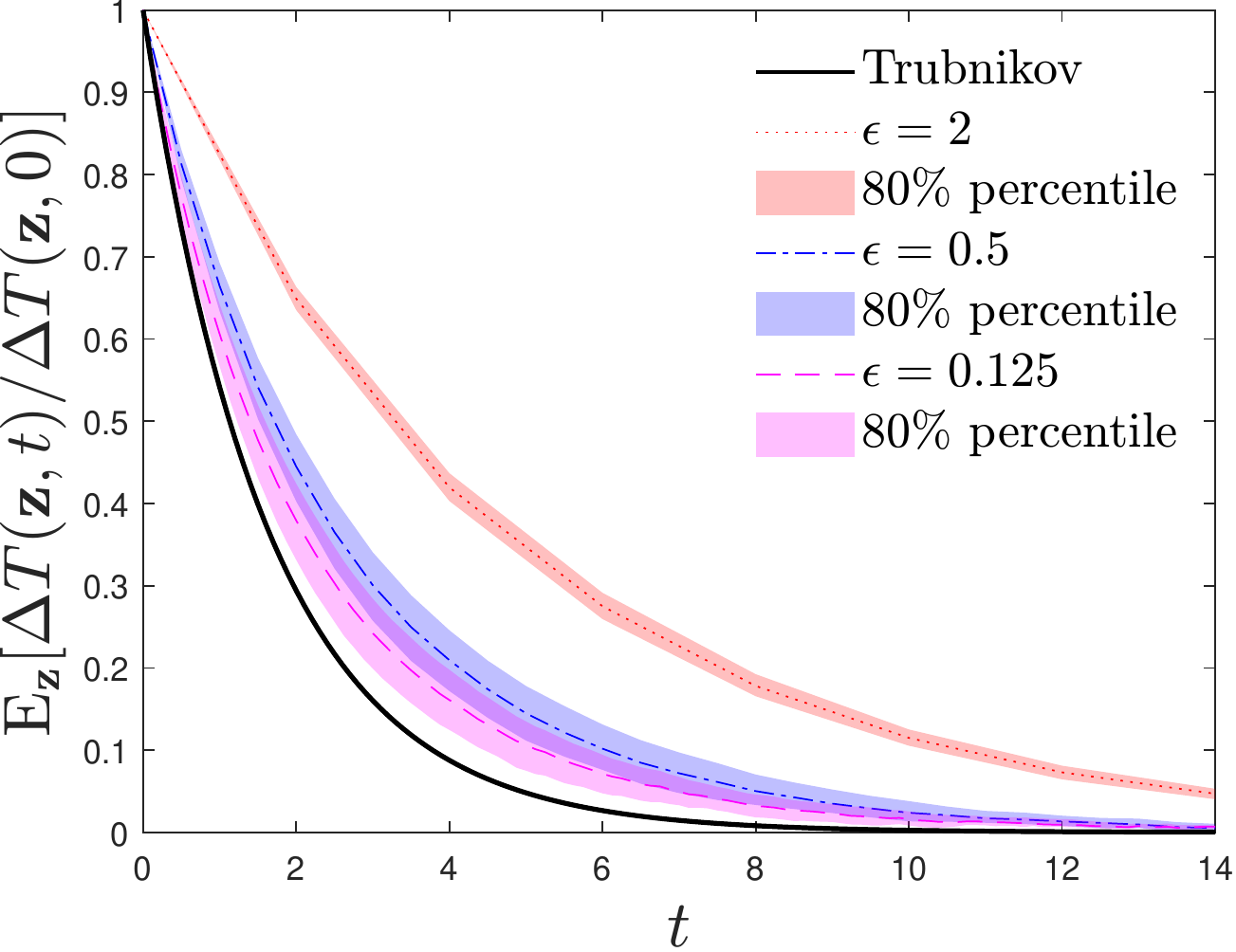}
	\includegraphics[width = 0.3\linewidth]{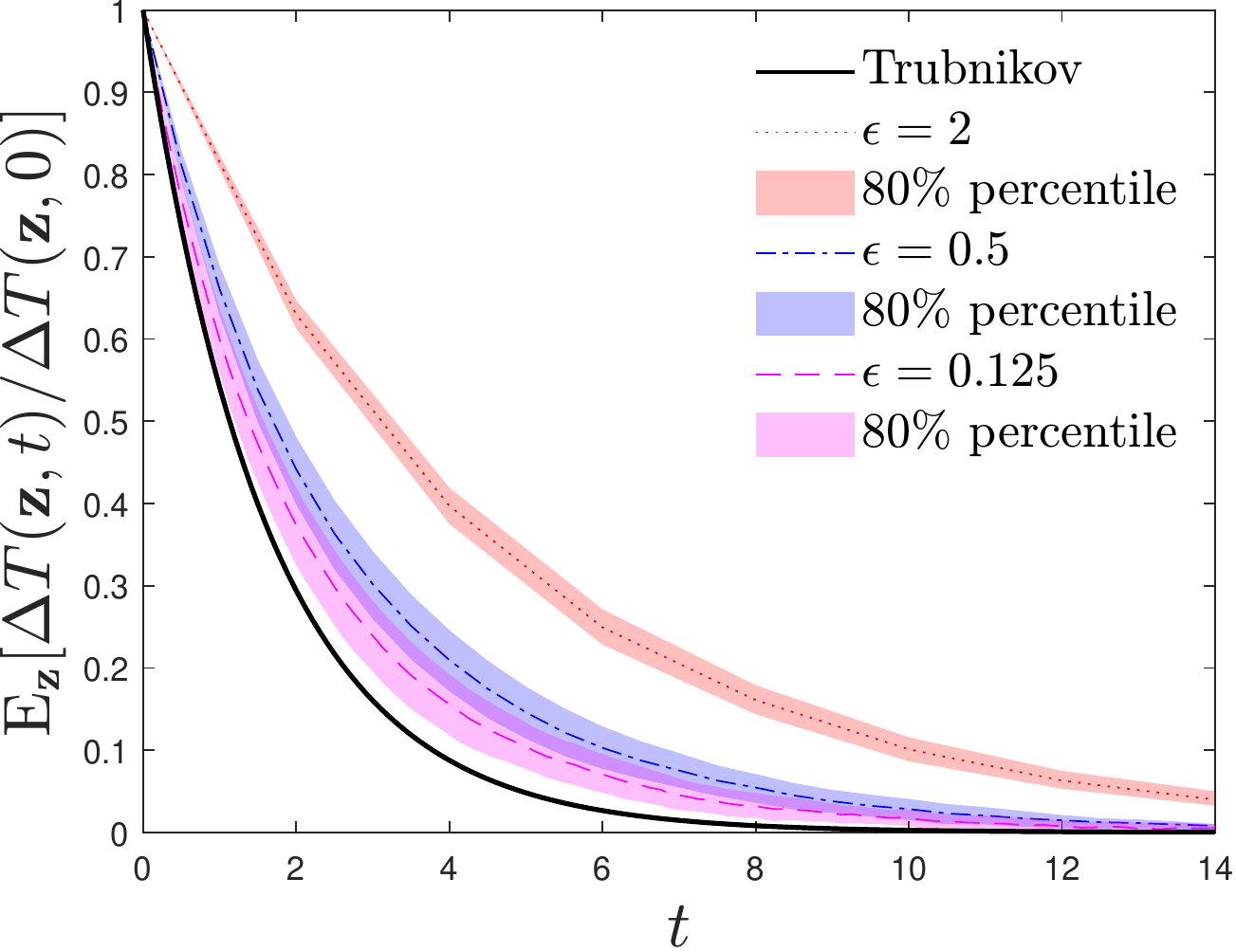}
	\caption{\small{\textbf{Test 4}. Time evolution of the expectation of $\Delta T(\z,t)/\Delta T(\z,0)$, for different values of the scale parameter $\epsilon$, for the Nanbu-Babovsky (top row) and Bird (bottom row) schemes, for $D^{(1)}_*$ (left panels), $D^{(2)}_*$ (centre panels), and $D^{(3)}_*$ (right panels). The black line is the benchmark solution \eqref{eq:trubnikov}. The number of particles is $N=5\times10^6$, $M=5$ and $\Delta t = \epsilon/\rho$.}}
	\label{fig:test_4_trubnikov}
\end{figure}
We initialize the distribution as an ellipsoid, i.e., as
\[
f_0(v,\z) = \frac{1}{\left(2\pi\right)^{3/2}} \frac{1}{\sqrt{T^0_x(\z) T^0_y(\z) T^0_z}} \left(e^{-\frac{v^2_x}{2 T^0_x(\z)}}e^{-\frac{v^2_y}{2 T^0_y(\z)}}e^{-\frac{v^2_z}{2 T^0_z}}\right),
\]
with anisotropic uncertain initial temperature
\[
T^0_x(\z) = T^0_y(\z) > T^0_z. 
\]
In the case of small uncertain temperature difference, the Trubnikov^^>\cite{Trubnikov1965} formula reads
\begin{equation}\label{eq:trubnikov}
\Delta T(\z,t)=\Delta T^0(\z) \exp \left( - t / \tau_T \right), \quad \text{with} \quad \tau_T = \frac{5}{8}\sqrt{2\pi} \left( \frac{8\sqrt{m}}{\pi \sqrt{2}} \frac{T(\z)^{3/2}}{e^4 \rho \log \Lambda}\right),
\end{equation}
which can be used as a benchmark solution to test the numerical schemes. We choose 
\[
\begin{split}
	& T^0_x(\z) = T^0_y(\z) = 0.08 + 0.04 \z \\
	& T^0_z = 0.04,
\end{split}
\]
with $\z\sim\mathcal{U}([0,1])$, $N=5\times10^6$ particles and a sG expansion up to order $M=5$. 

First, we analyse the behaviour of the schemes for different values of the scale parameter $\epsilon=\rho\Delta t$. We test the Nanbu-Babovsky and the Bird scheme, for the kernels $D^{(1)}_*$, $D^{(2)}_*$, and $D^{(3)}_*$. The results are summarised in Figure \ref{fig:test_4_trubnikov}. We observe that for smaller values of $\epsilon$, the numerical solutions get closer to the reference one, indicating that the schemes are consistent with the grazing limit. However, as pointed out also in^^>\cite{caflisch2010}, there is a discrepancy between the two solutions, since the Trubnikov relation \eqref{eq:trubnikov} is still an (analytical) approximation. 

Then, at fixed $\epsilon=0.5$, we compare the solutions obtained with different kernels, for both the Nanbu-Babovsky and the Nanbu scheme. In Figure \ref{fig:test_4_trubnikov_2} we may notice that with the choices $D^{(2)}_*$ and $D^{(3)}_*$ the schemes perform better with respect to the choice $D^{(1)}_*$, in the sense that they are closer to the Trubnikov solution.

In the end, we compare the results obtained with the Nanbu-Babovsky and the Nanbu schemes for all the kernels. We fix $\epsilon=0.5$, $N=10^6$ particles and $M=5$. In Figure \ref{fig:test_4_trubnikov_3} we show the time evolution of the temperature difference for $D^{(1)}_*$ (left panel), $D^{(2)}_*$ (centre panel), and $D^{(3)}_*$ (right panel).

\begin{figure}
	\centering
	\includegraphics[width = 0.4\linewidth]{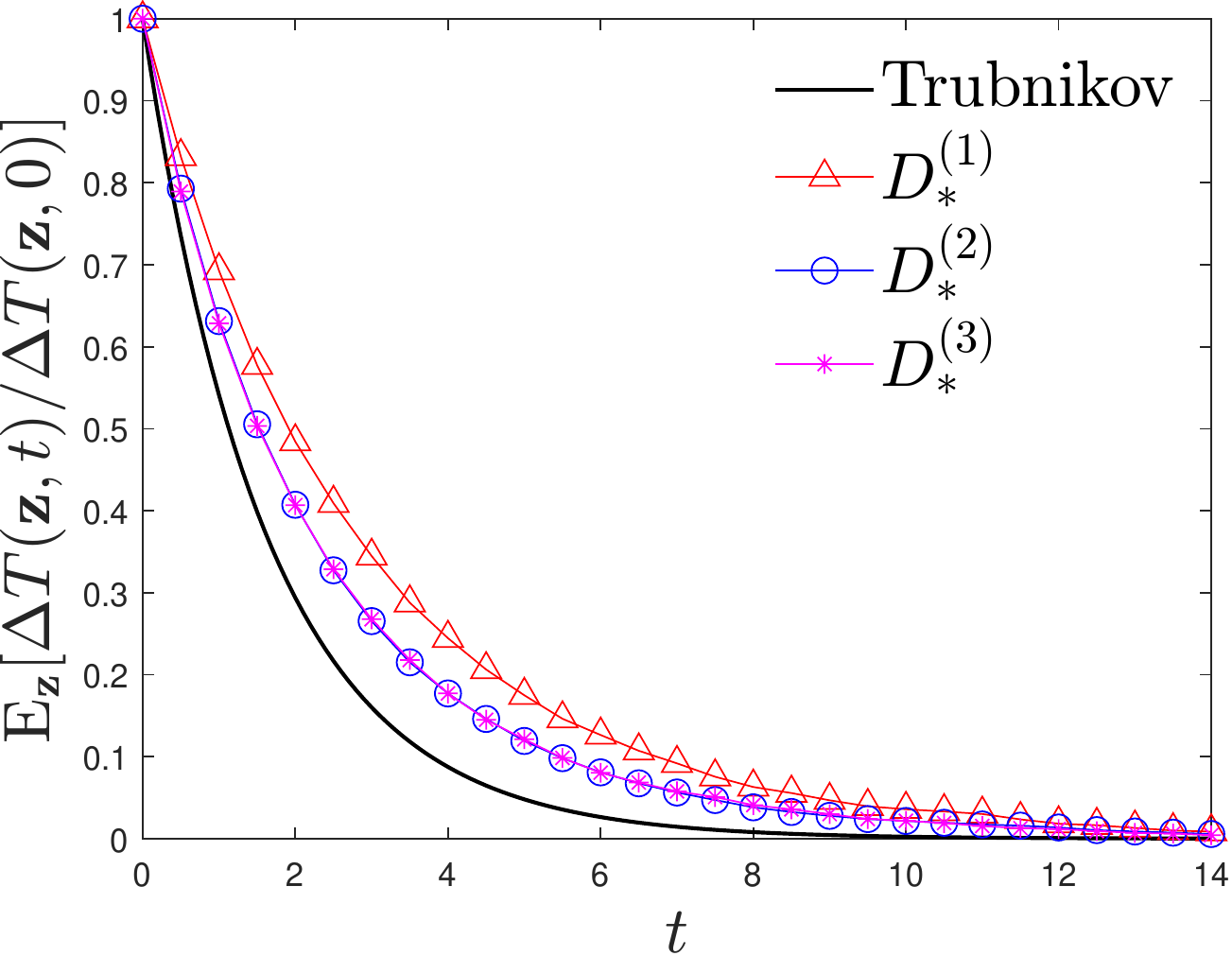}
	\includegraphics[width = 0.4\linewidth]{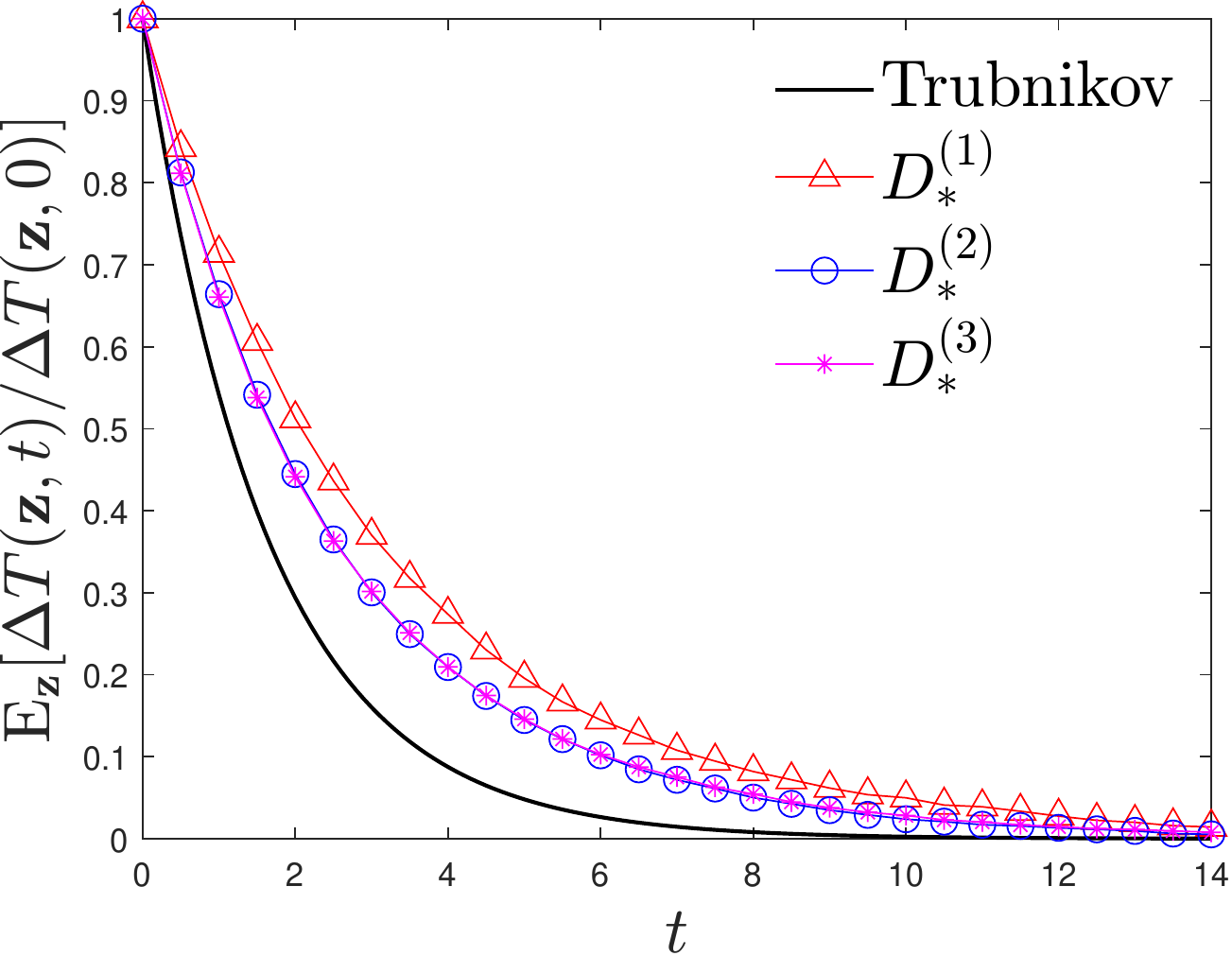}
	\caption{\small{\textbf{Test 4}. Comparison of the time evolution of the expectation of $\Delta T(\z,t)/\Delta T(\z,0)$ for the choices of the kernel $D^{(1)}_*$, $D^{(2)}_*$, and $D^{(3)}_*$, for fixed $\epsilon=0.5$, for both the Nanbu-Babovsky (left panel) and Bird (right panel) schemes. The black line is the benchmark solution \eqref{eq:trubnikov}. The number of particles is $N=5\times10^6$, $M=5$ and $\Delta t = \epsilon/\rho = 1 $.}}
	\label{fig:test_4_trubnikov_2}
\end{figure}
\begin{figure}
	\centering
	\includegraphics[width = 0.3\linewidth]{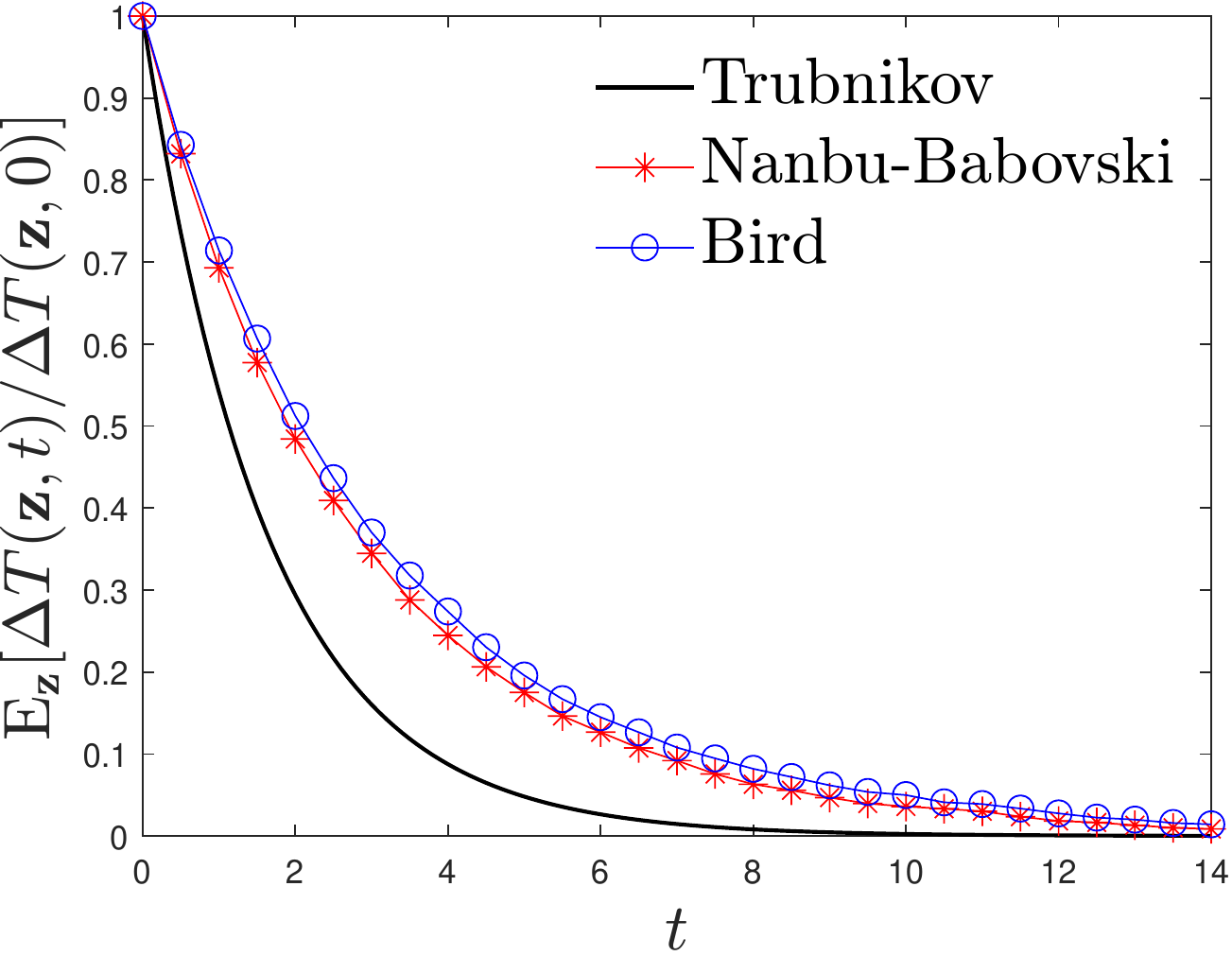}
	\includegraphics[width = 0.3\linewidth]{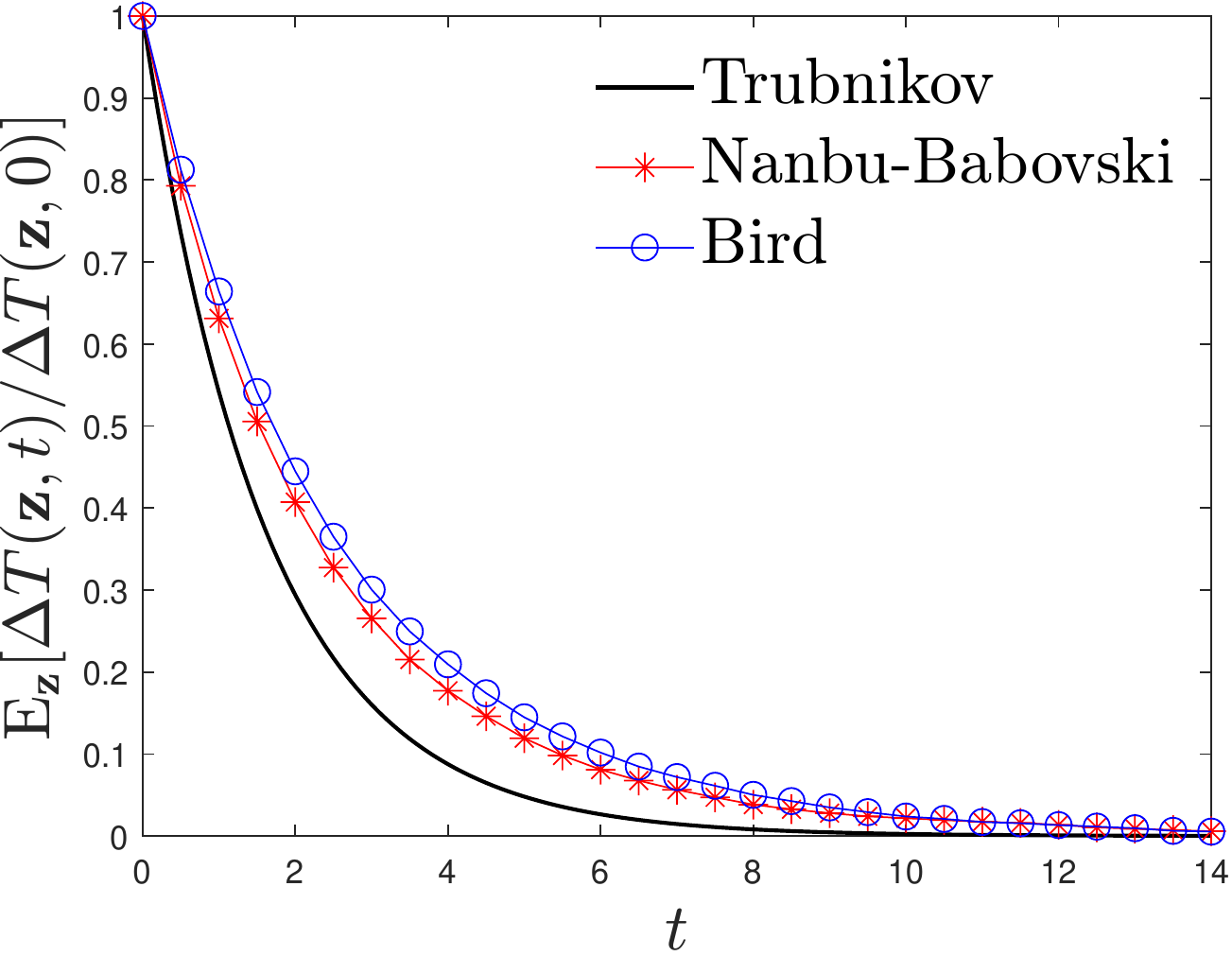}
	\includegraphics[width = 0.3\linewidth]{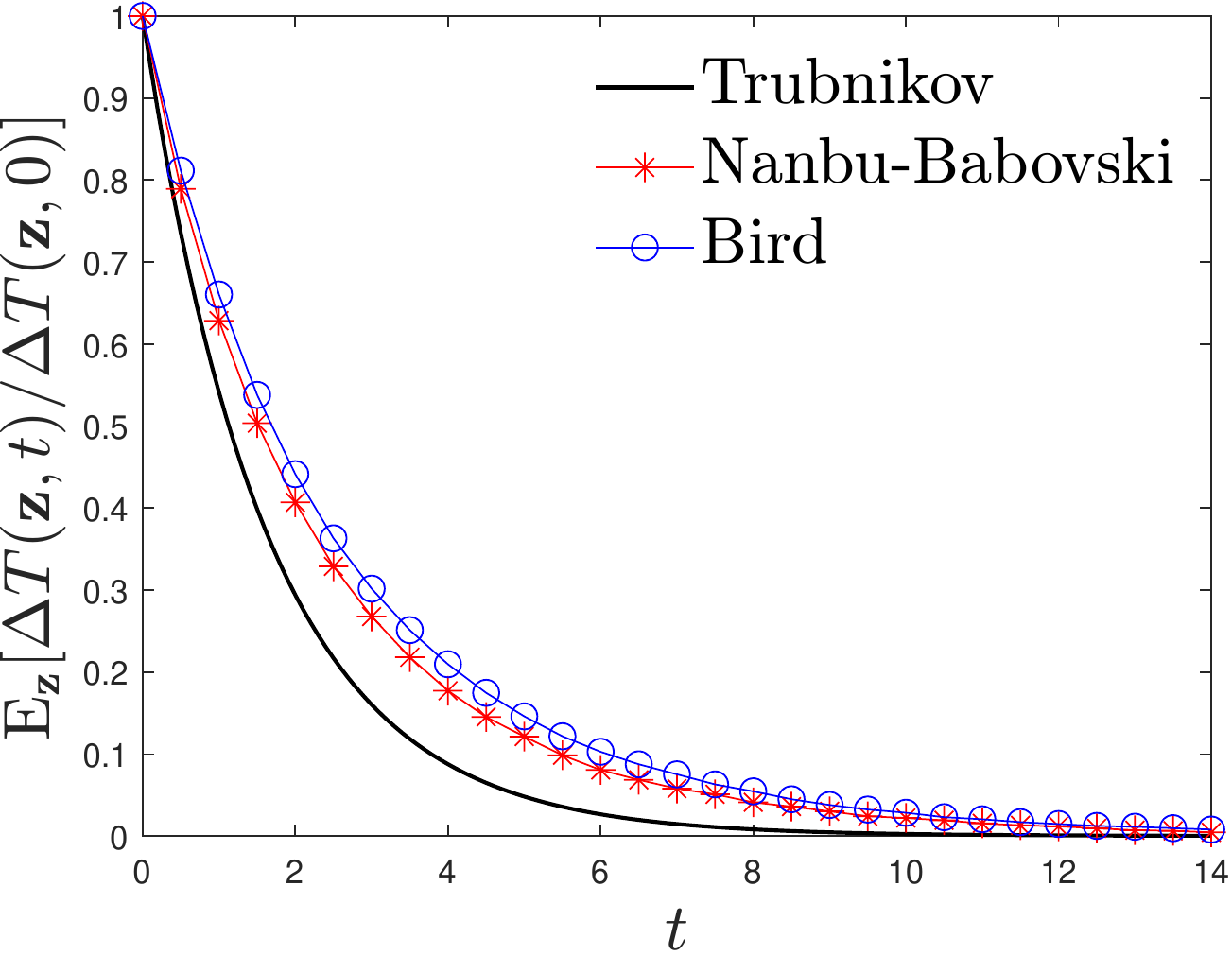}
	\caption{\small{\textbf{Test 4}. Comparison of the time evolution of the expectation of $\Delta T(\z,t)/\Delta T(\z,0)$ of the Nanbu-Babovsky and Bird schemes, for the kernels $D^{(1)}_*$ (left panel), $D^{(2)}_*$ (centre panel), and $D^{(3)}_*$ (right panel), for fixed $\epsilon=0.5$. The black line is the benchmark solution \eqref{eq:trubnikov}. The number of particles is $N=5\times10^6$, $M=5$ and $\Delta t = \epsilon/\rho $.} }
	\label{fig:test_4_trubnikov_3}
\end{figure}
\paragraph{Sum of two Gaussians}
\begin{figure}
	\centering
	\includegraphics[width = 0.3\linewidth]{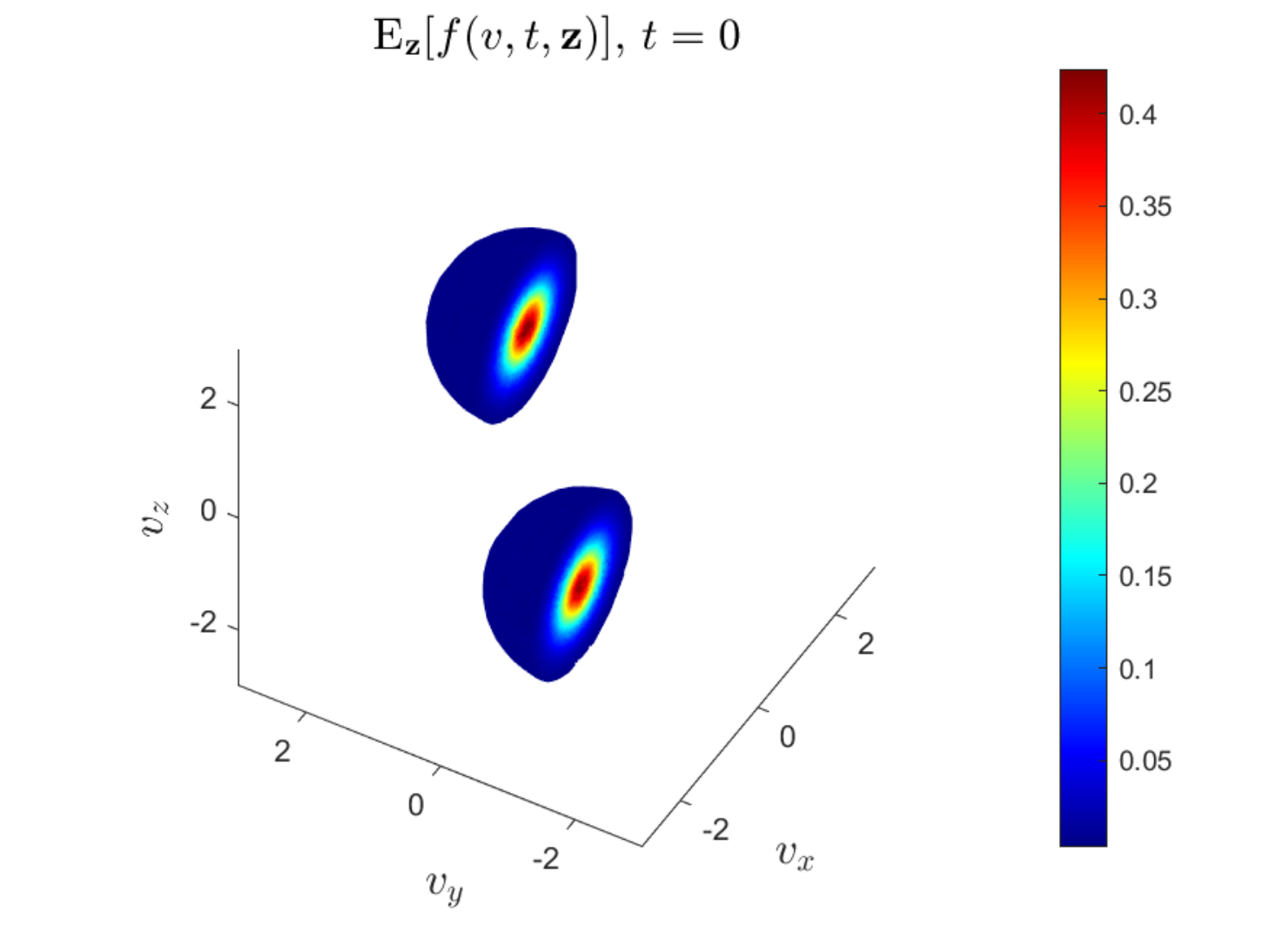}
	\includegraphics[width = 0.3\linewidth]{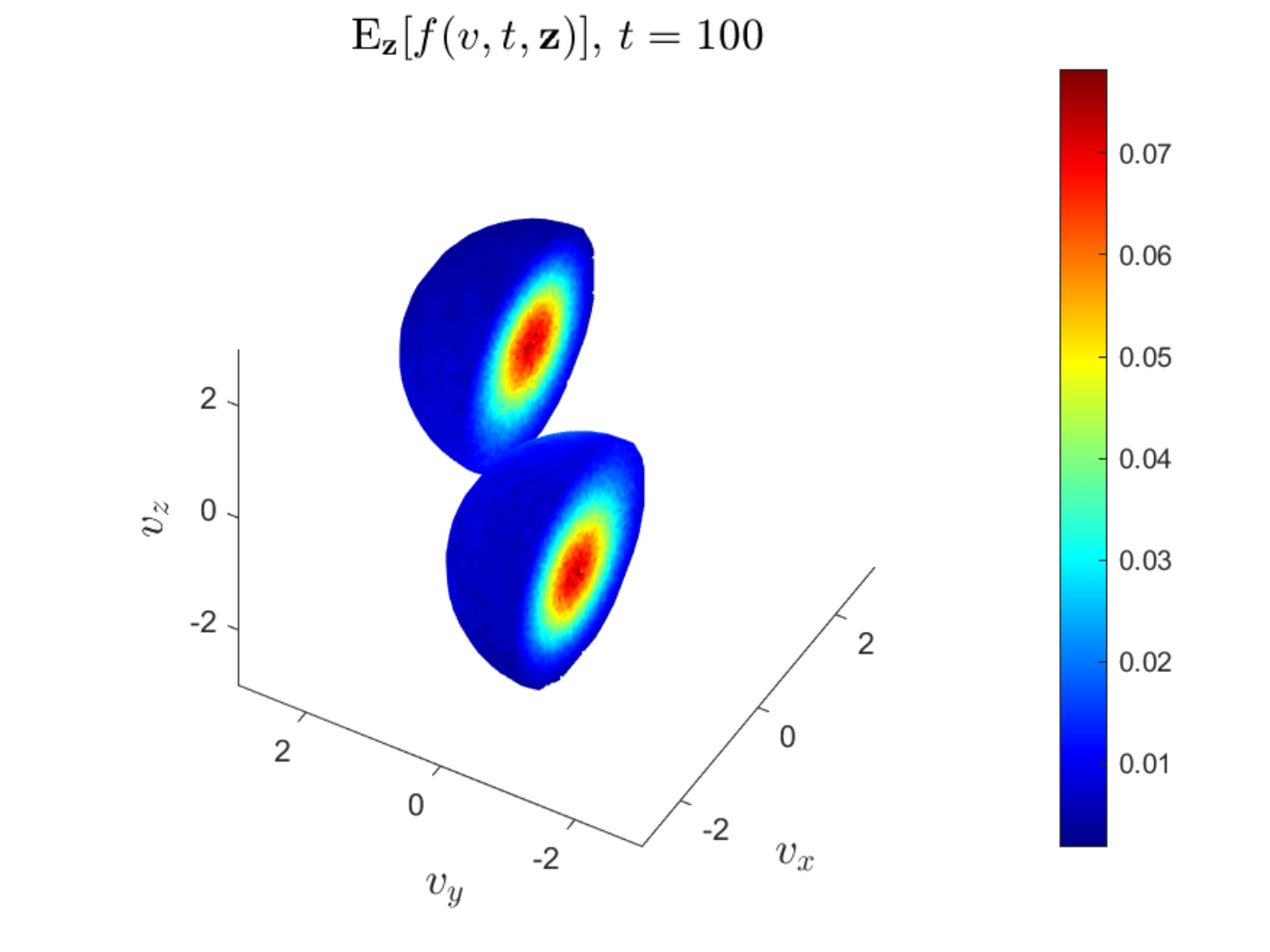}
	\includegraphics[width = 0.3\linewidth]{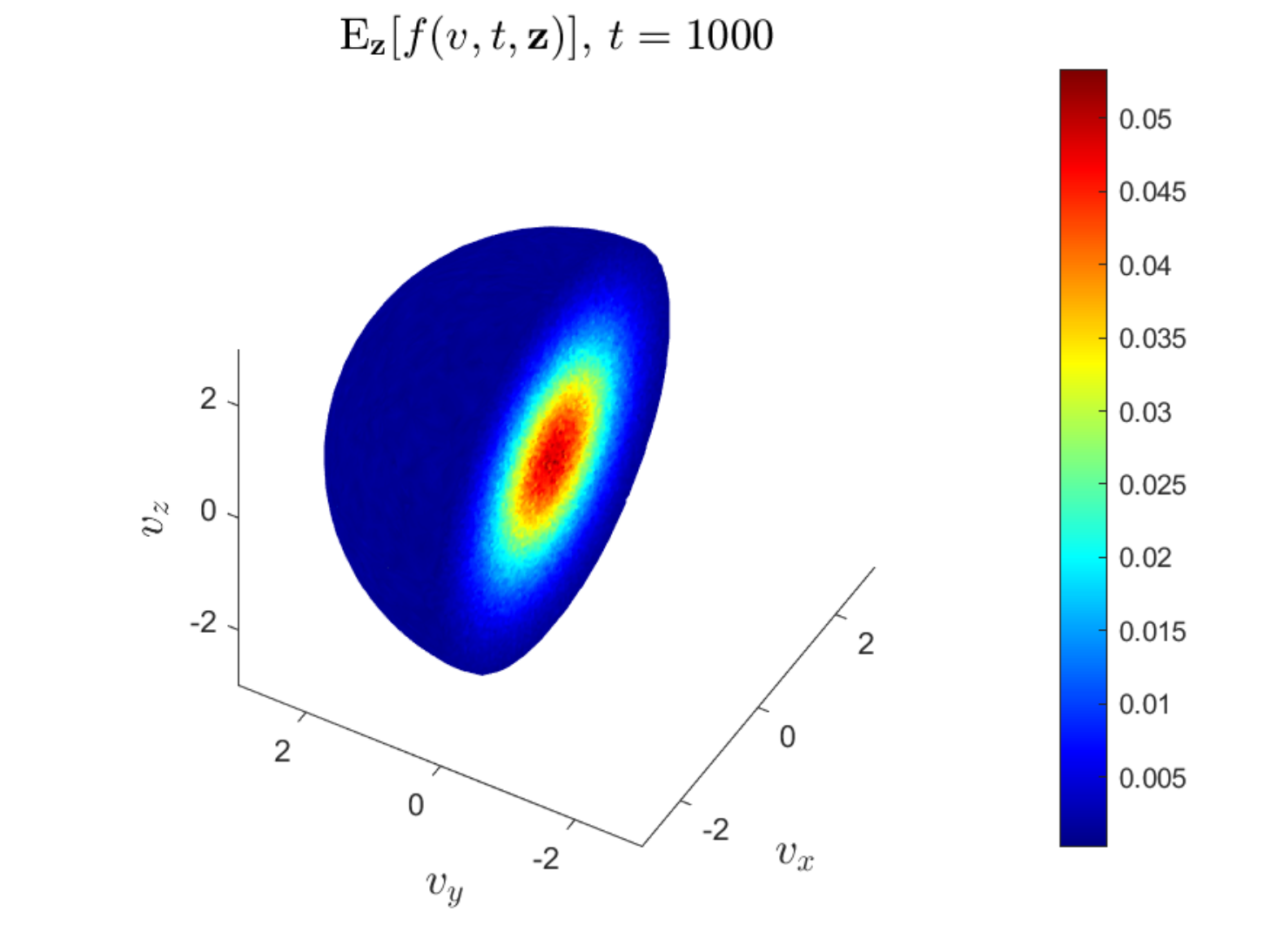}
	\includegraphics[width = 0.3\linewidth]{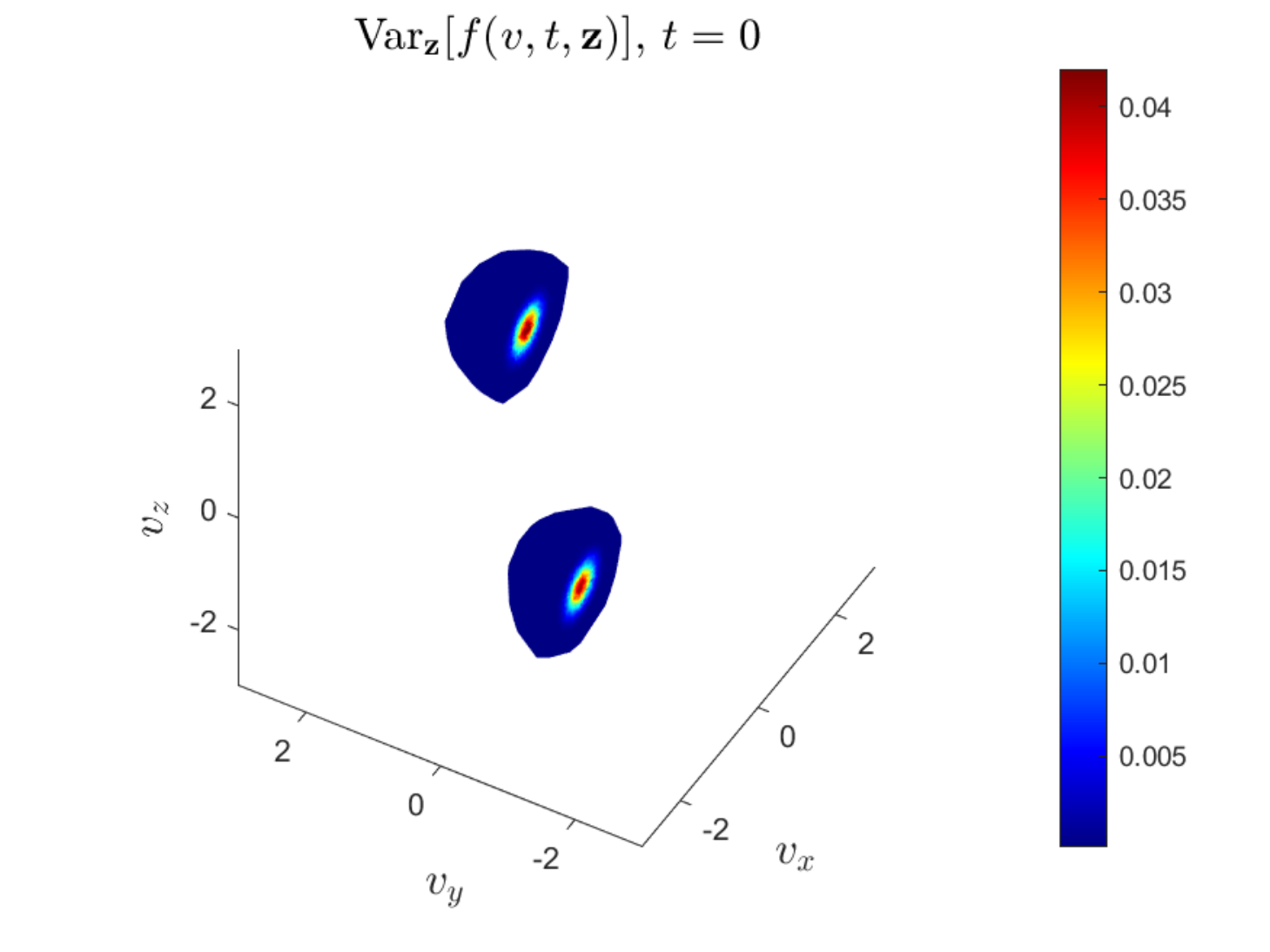}
	\includegraphics[width = 0.3\linewidth]{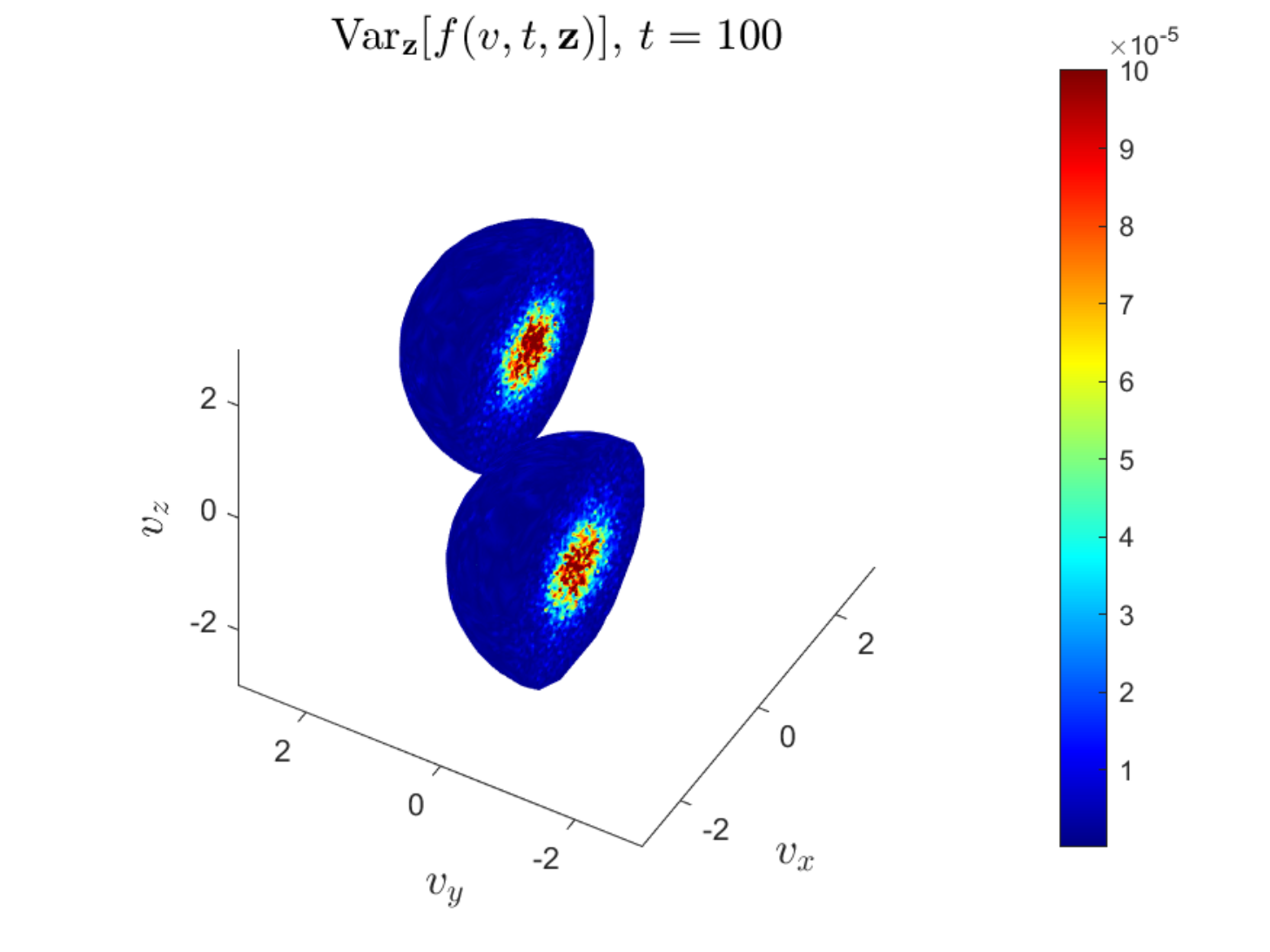}
	\includegraphics[width = 0.3\linewidth]{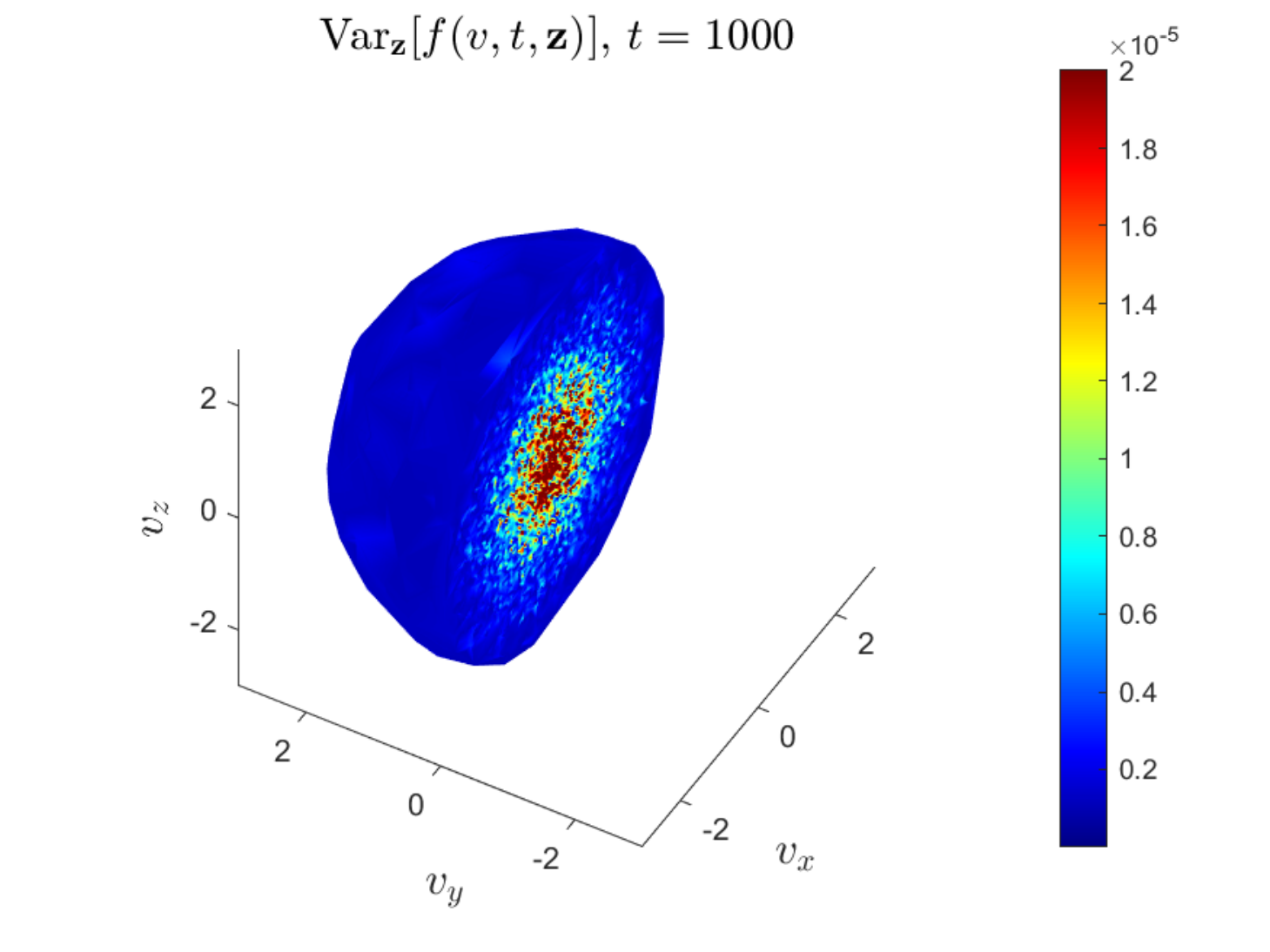}
	\caption{\small{\textbf{Test 4}. Slices of the expected distribution $\mathbb{E}_{\z}[f(v,t,\z)]$ (upper row) the variance $\mathrm{Var}_{\z}[f(v,t,\z)]$ (lower row) at fixed times $t=0,100,1000$, for the Coulombian model with uncertainties, with sum of two Gaussians initial conditions \eqref{eq:bimodal_init}. We choose $N=5\times10^7$ particles, $M=5$, $\Delta t=\epsilon/\rho=1$.} }
	\label{fig:test_4_bimodal_3D}
\end{figure}
\begin{figure}
	\centering
	\includegraphics[width = 0.3\linewidth]{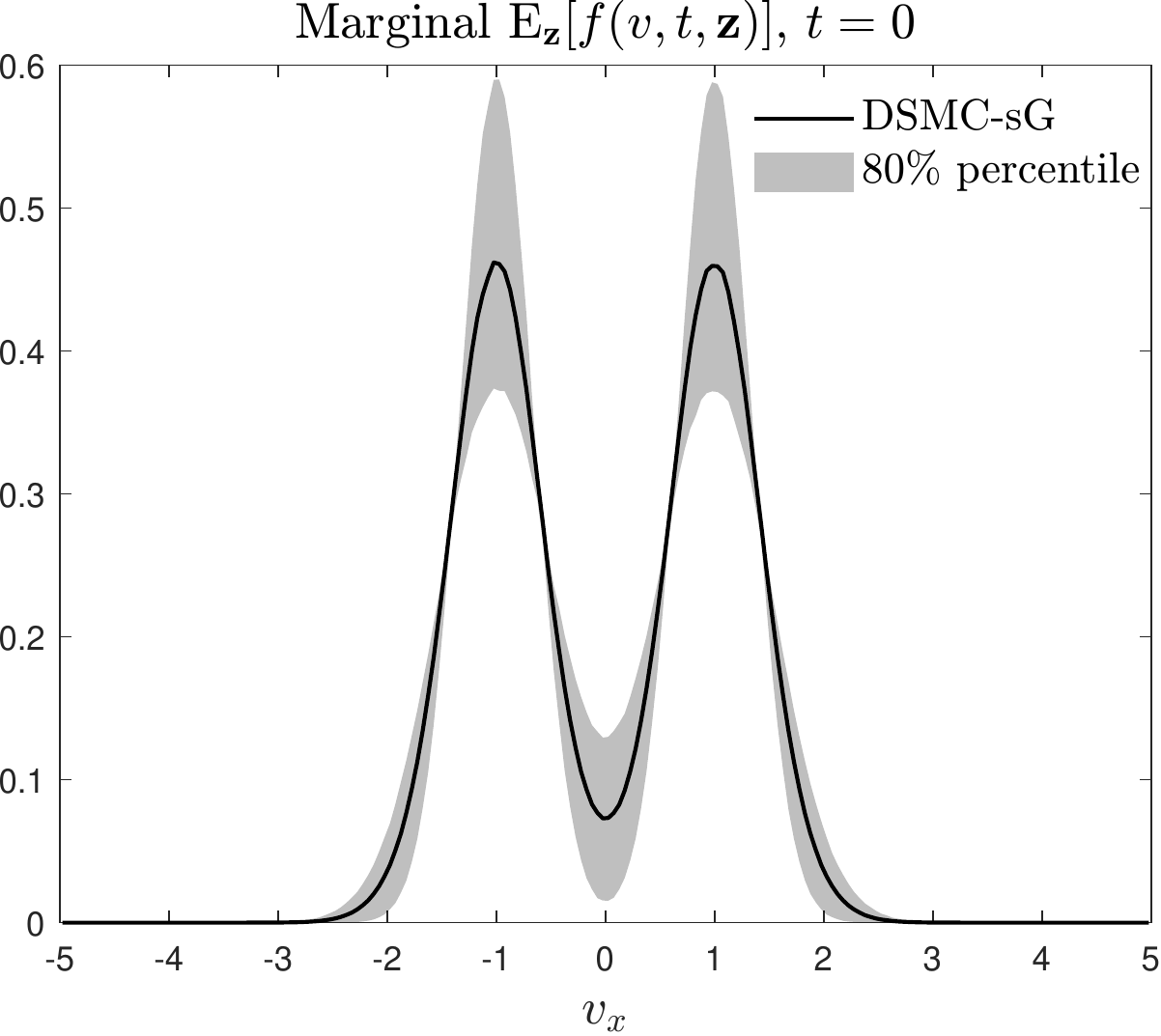}
	\includegraphics[width = 0.3\linewidth]{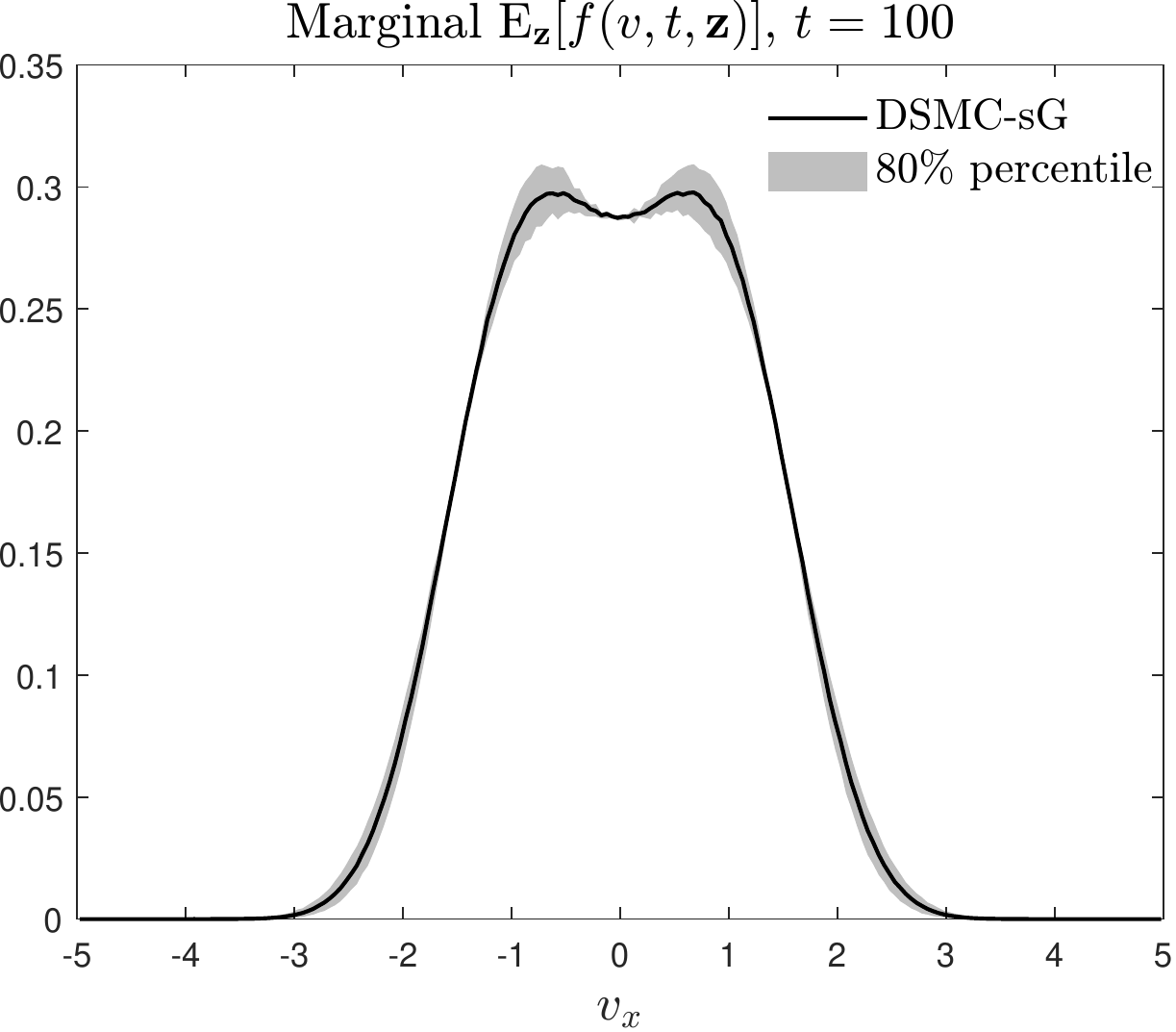}
	\includegraphics[width = 0.3\linewidth]{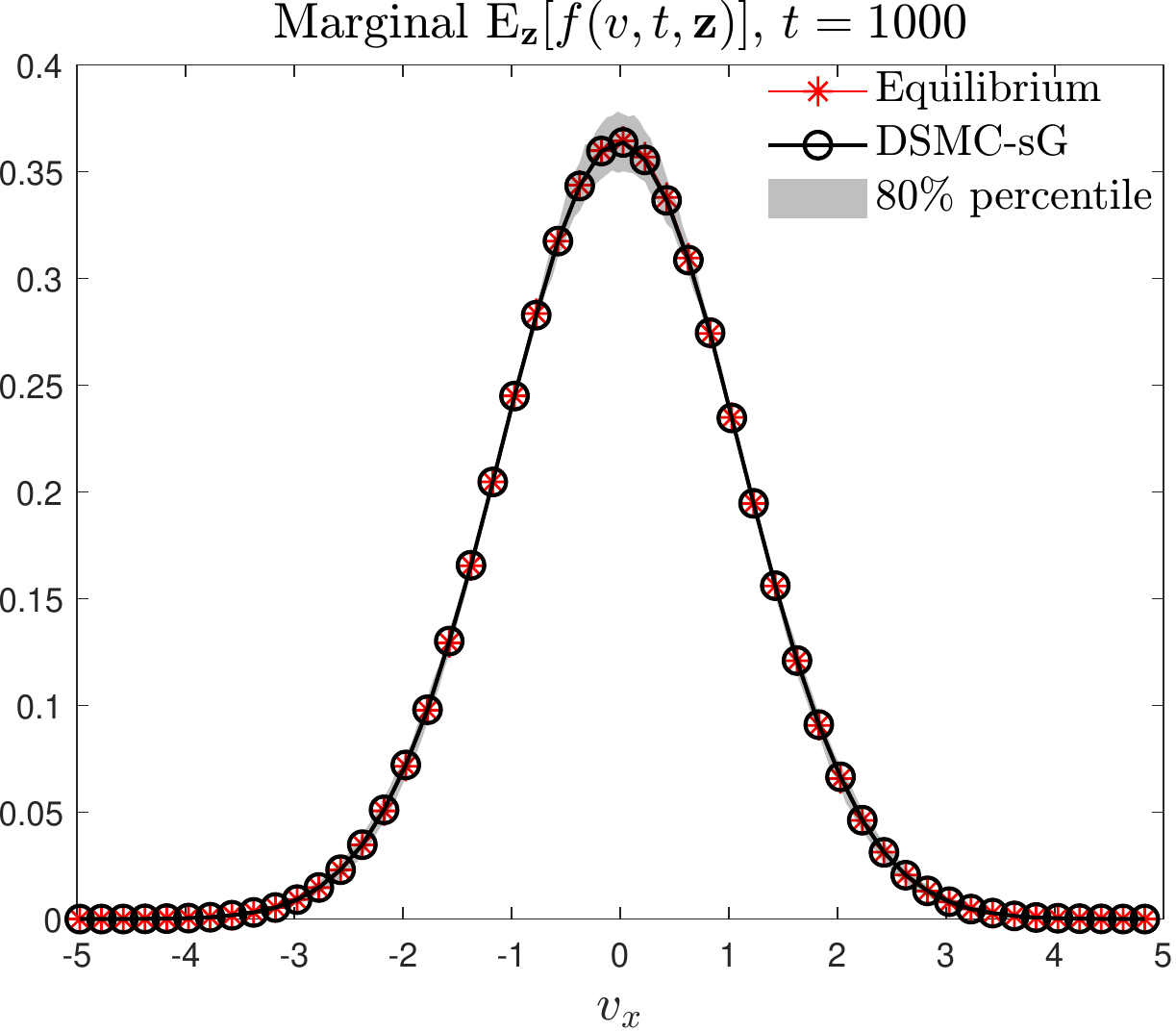}
	\caption{\small{\textbf{Test 4}. Evolution at fixed times $t=0,100,1000$ of the marginal $\mathbb{E}_{\z}[f(v,t,\z)]$ of the Coulombian model with uncertainties, with sum of two Gaussians initial conditions \eqref{eq:bimodal_init}. We choose $N=5\times10^7$ particles, $M=5$, $\Delta t=\epsilon/\rho=1$. The equilibrium marginal distribution (red starred) of the right panel is the Maxwellian distribution computed with the theoretical (conserved) mean and temperature.}}
	\label{fig:test_4_bimodal}
\end{figure}
We initialize the distribution as the sum of two Gaussians centred in $v=\pm 1$ and with the same uncertain temperature $T(\z)=0.1+0.2\z$, $\z\sim\mathcal{U}([0,1])$, i.e.,
\be \label{eq:bimodal_init}
f_0(v,\z) = \dfrac{1}{2(2\pi T(\z))^{3/2}} \left( e^{-\dfrac{|v+1|^2}{2 T(\z)}} + e^{-\dfrac{|v-1|^2}{2 T(\z)}}\right).
\ee
We use the Nanbu-Babovsky algorithm and the kernel $D^{(3)}_*$. We choose $N=5\times 10^7$ particles, $\Delta t=\epsilon/\rho=1$ and a stochastic Galerkin expansion with $M=5$. As we expected, the system evolves toward the equilibrium distribution, that is the centred Gaussian with the temperature $T(\z)$. In Figure \ref{fig:test_4_bimodal_3D} we show the slices at fixed times $t=0,100,1000$ of $\mathbb{E}_{\z}[f(v,t,\z)]$ (upper row), and $\mathrm{Var}_{\z}[f(v,t,\z)]$ (lower row). In Figure \ref{fig:test_4_bimodal} we display the marginals of $\mathbb{E}_{\z}[f(v,t,\z)]$ at the same times. From the right panel, we may notice the accordance between the numerical solution (black circled line) and the expected equilibrium distribution (red starred line).

\paragraph{Bump on tail}
\begin{figure}
	\centering
	\includegraphics[width = 0.45\linewidth]{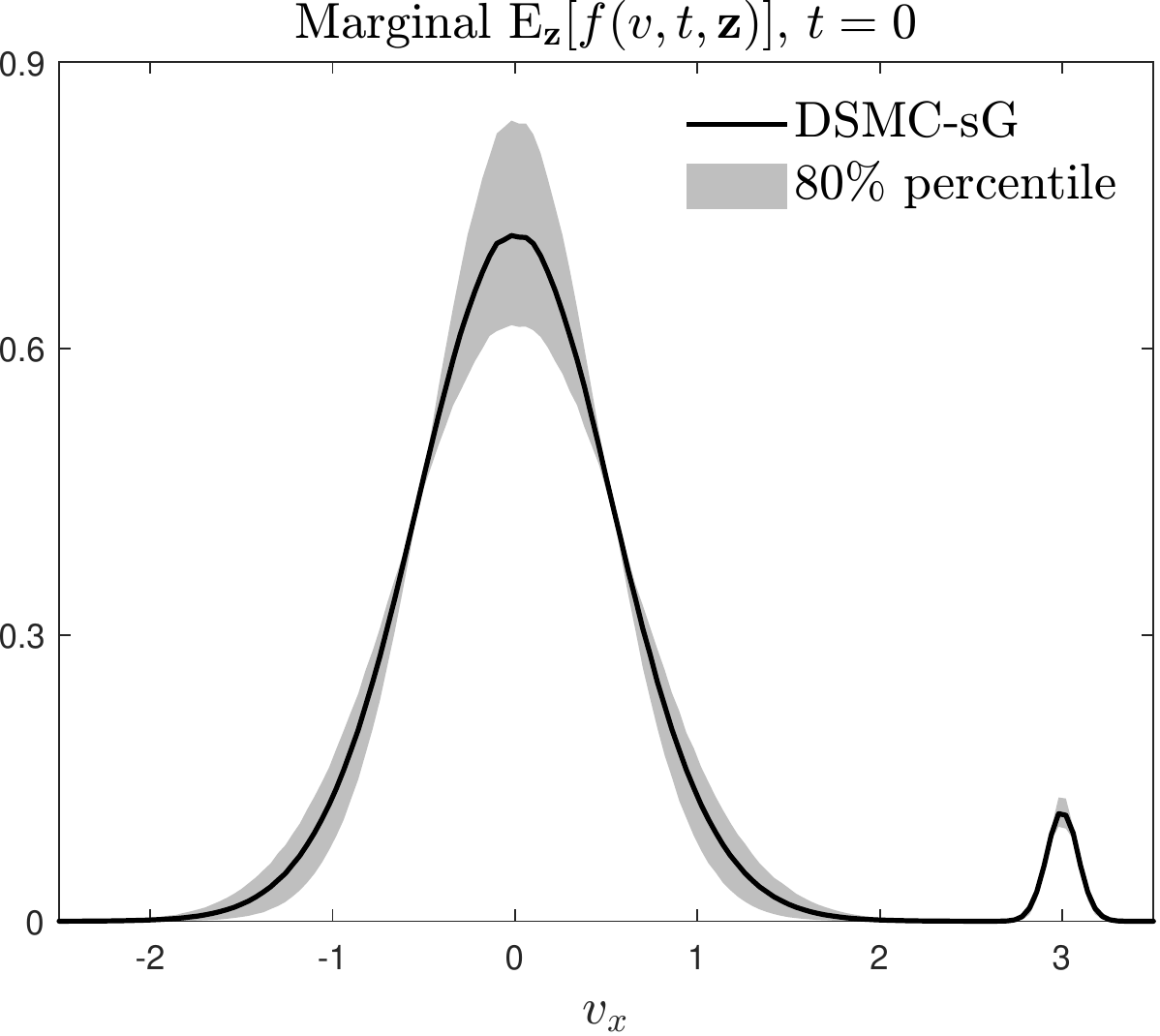}
	\includegraphics[width = 0.45\linewidth]{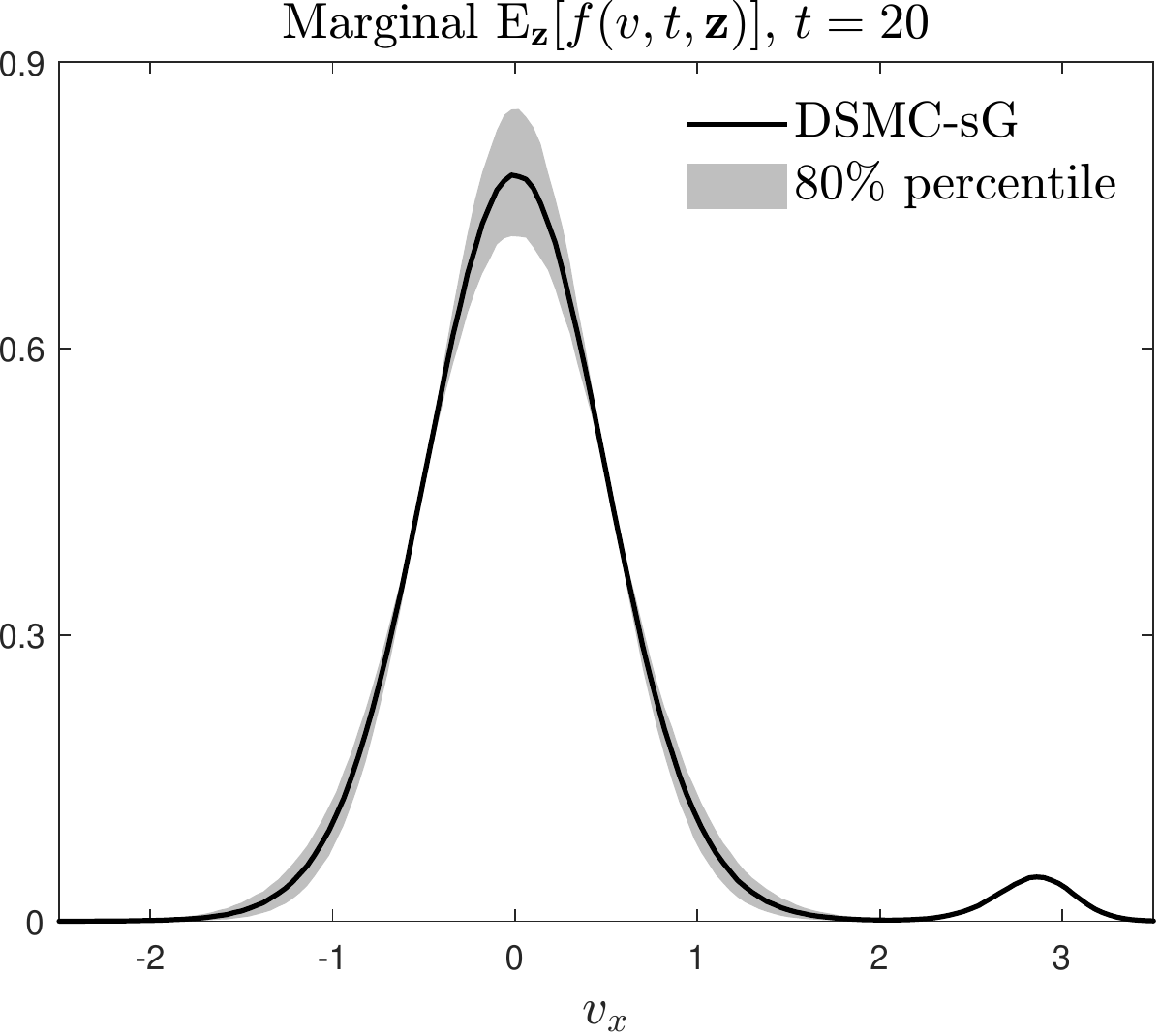}
	\includegraphics[width = 0.45\linewidth]{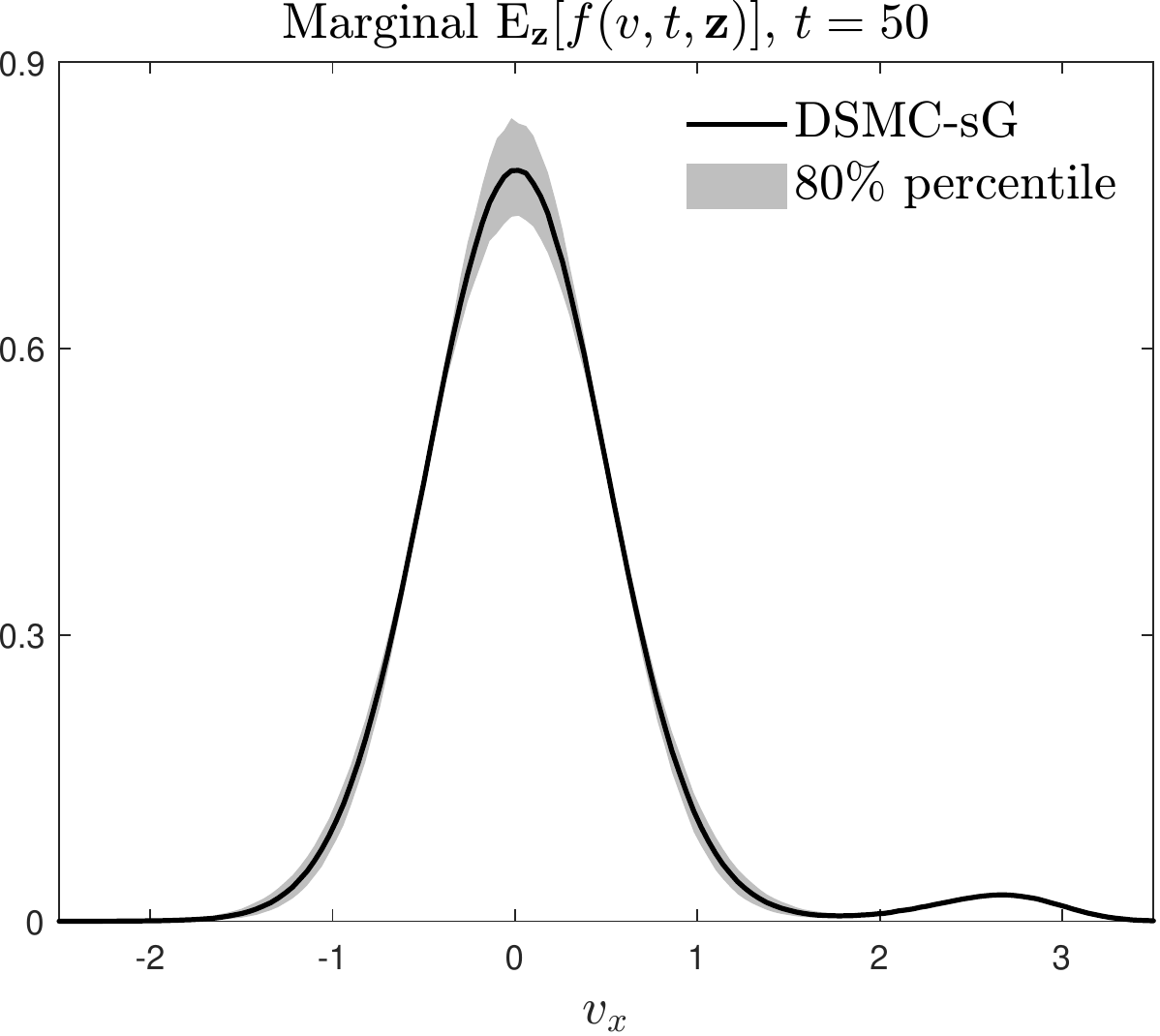}
	\includegraphics[width = 0.45\linewidth]{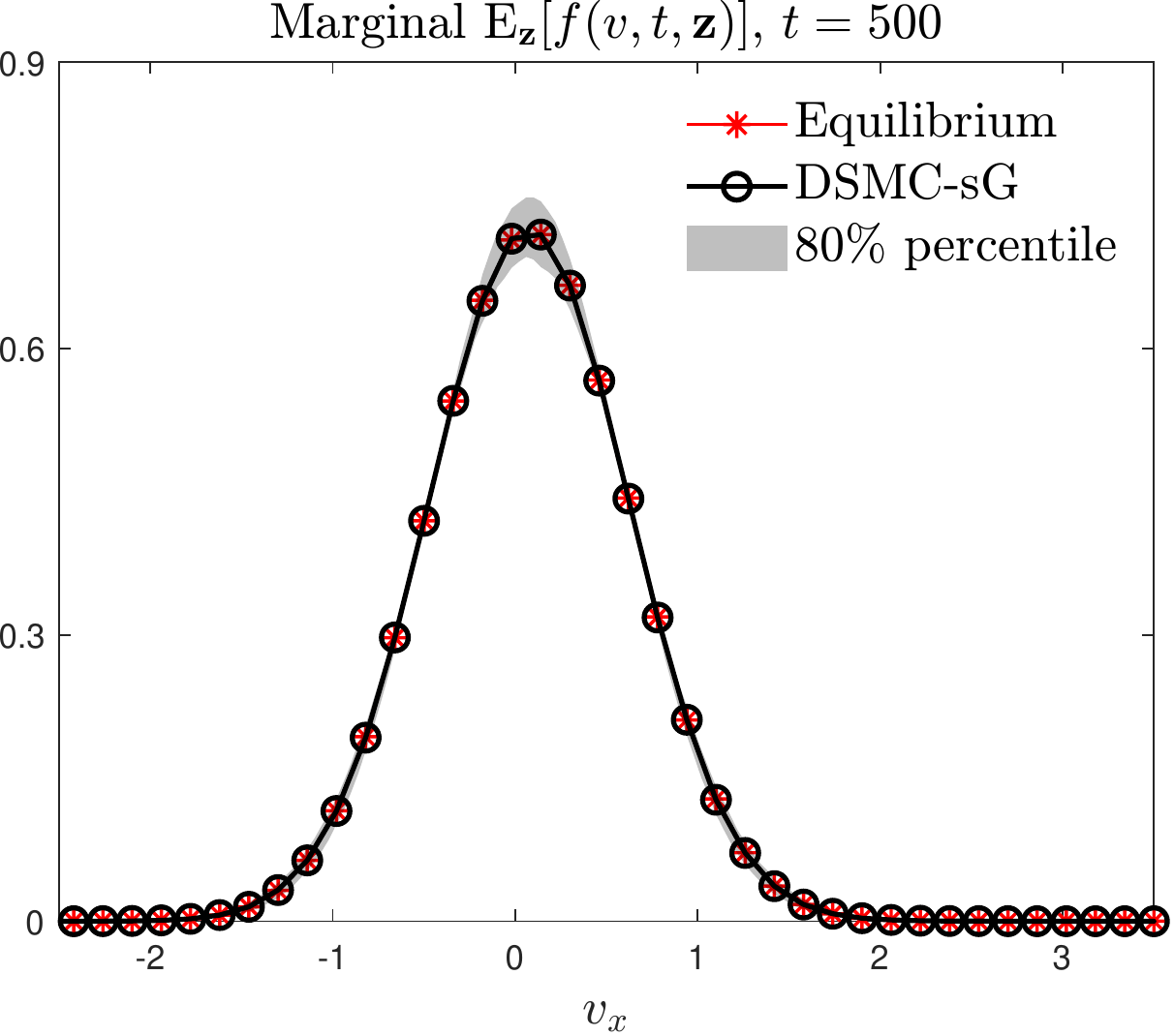}
	\caption{\small{\textbf{Test 4}. Evolution at fixed times $t=0,20,50,500$ of the marginal $\mathbb{E}_{\z}[f(v,t,\z)]$ of the Coulombian model with uncertainties, with bump on tail initial conditions. We choose $N=5\times10^7$ particles, $M=5$, $\Delta t=\epsilon/\rho=1$. The equilibrium marginal distribution (red starred) of the lower row, the right panel is Maxwellian distribution computed with the theoretical (conserved) mean and temperature.} }
	\label{fig:test_4_bump}
\end{figure}
We initialize the distribution as a centred Gaussian with a bump on tail, namely a small portion of mass concentrated in $v=3$
\be \label{eq:bot_init}
f_0(v,\z) = \dfrac{1}{(2\pi T_1(\z))^{3/2}} \dfrac{39}{40}e^{-\dfrac{|v|^2}{2 T_1(\z)}} + \dfrac{1}{(2\pi T_2(\z))^{3/2}} \dfrac{1}{40}e^{-\dfrac{|v-3|^2}{2 T_2(\z)}}
\ee
with uncertain temperatures $T_1(\z)=0.2+0.2\z$, $\z\sim\mathcal{U}([0,1])$, and $T_2(\z)=T_1(\z)/40$. Given this initial condition, we observe that the conserved quantities, besides the mass, are $\textrm{M}1=\frac{3}{40}$ and $T(\z)=\frac{39}{40}T_1(\z)+\frac{1}{40}T_2(\z)=\frac{1561}{1600}T_1(\z)$. We choose again the Nanbu-Babovsky algorithm and the kernel $D^{(3)}_*$, with $N=5\times 10^7$ particles, $\Delta t=\epsilon/\rho=1$ and a stochastic Galerkin expansion $M=5$.

In Figure \ref{fig:test_4_bump} we display the marginals $\mathbb{E}_{\z}[f(v,t,\z)]$ at fixed times $t=0,20,50,500$. We observe that the bump is absorbed as the time increase and, at the time $t=500$, the system is close to the equilibrium, i.e. the Gaussian with mean $\textrm{M}1=\frac{3}{40}$ and temperature $T(\z)$. In particular, in the bottom row, right panel, the accordance is good between the DSMC-sG approximation and the expected equilibrium distribution.

\subsection{Test 5: DSMC-sG versus DSMC-MC}
We are interested here in comparing the DSMC-sG algorithm with the DSMC-MC, i.e., the standard DSMC scheme with a Monte Carlo sampling of the random parameter. We take the same computational setting of Section \ref{sec:test3}, that is, Nanbu-Babovsky scheme, kernel $D^{(3)}_*$, and initial conditions given by \eqref{eq:initBKW}. We fix the number of particles $N=10^6$, the time step $\Delta t=\epsilon/\rho=0.1$, and we compute the error in the evaluation of the fourth order moment at fixed time $t=1$
\[
\mathrm{Error}=|\mathbb{E}_{\z}[\mathrm{M}4(t,\z)]-\mathbb{E}_{\z}[\widetilde{\mathrm{M}4}(t,\z)]|,
\]
where $\widetilde{\mathrm{M}4}(t,\z)$ is a reference solution computed with a collocation algorithm with $N=10^8$ particles and $M=20$ collocation nodes. We want to compare the error with the computational cost, that is $N\cdot M^2$ for the DSMC-sG, where $M$ is the order of expansion, and $N\cdot M$ for the DSMC-MC, where $M$ is the size of the sample of the random parameter. In Figure \ref{fig:test_5_error_cost} we show the results for the cost divided by the fixed number of particles $N$. As we can see, the DSMC-sG performs better than the DSMC-MC for small cost. Moreover, we observe that the error of the DSMC-sG is saturated by the Monte Carlo part of the algorithm since it is constant for increasing order of expansion.

\begin{figure}
	\centering
	\includegraphics[width = 0.45\linewidth]{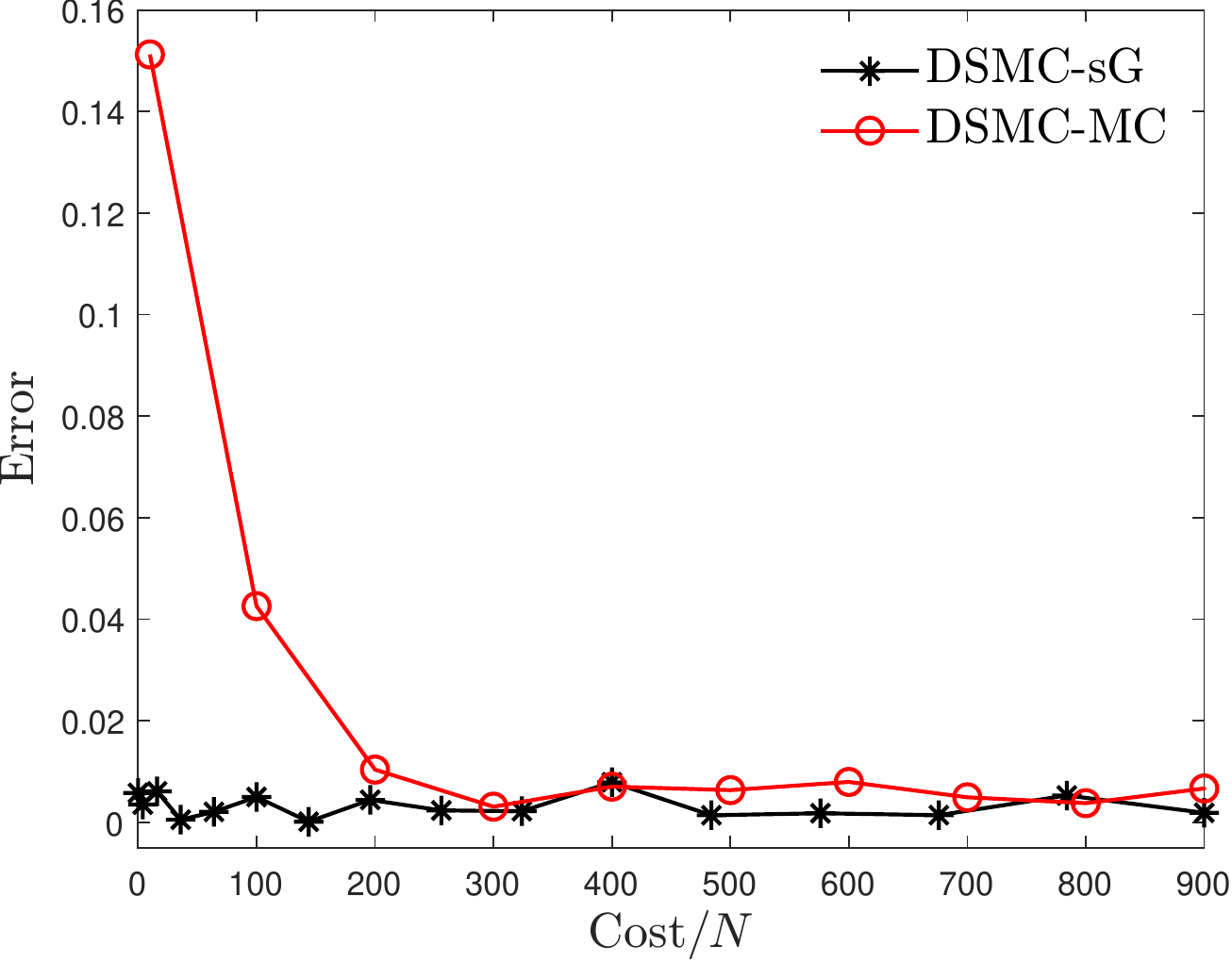}
	\caption{\small{\textbf{Test 5}. Comparison cost-Error between the DSMC-sG scheme and the standard DSMC with a Monte Carlo sampling of the random parameter (MC-MC). }}
	\label{fig:test_5_error_cost}
\end{figure}

\section*{Conclusions} 
In this work, we have investigated the design of efficient particle methods for the Landau-Fokker-Planck equation in the presence of uncertainties. The equation has significant implications for the development of fusion reactors and has been extensively utilized by researchers and engineers to gain insights into the behaviour of charged particles in plasma. To accurately predict the behaviour of these particles, it is crucial to account for uncertainties in the constitutive parameters that characterize the system behaviour.

Our approach combines collision algorithms inspired by the grazing limit of the Boltzmann equation with a stochastic Galerkin particle projection. This method leverages the general structure of kernels in collision algorithms and introduces a regularized kernel to ensure spectral accuracy in the random space. In addition, the method benefits from a more efficient collision strategy which avoids iterative algorithms and, by employing particle reconstruction, it preserves the nonnegativity of the solution and other essential physical properties.

To validate the effectiveness and accuracy of our approach, we have conducted various numerical tests, including classical benchmarks such as BKW, Trubnikov's solution for Coulombian particles, and the bump-on-tail problem. The results demonstrate the efficiency and accuracy of our method in capturing the complex dynamics described by the Landau equation with uncertain data.

In conclusion, we have developed a robust and accurate approach that can effectively capture the behaviour of charged particles in collisional plasmas under uncertain data. Future developments may involve extending the methodology to more complex scenarios and investigating additional applications in plasma physics in space non-homogeneous situations.

\section*{Acknowledgments}
This work has been written within the activities of GNCS and GNFM groups of INdAM (National Institute of High Mathematics) and is partially supported by ICSC – Centro Nazionale di Ricerca in High Performance Computing, Big Data and Quantum Computing, funded by European Union – NextGenerationEU. L.P. acknowledges the partial support of MIUR-PRIN Project No. 2017KKJP4X “Innovative numerical methods for evolutionary partial differential equations and applications”.
M.Z. acknowledges partial support of MUR-PRIN Project No. 2020JLWP23 “Integrated mathematical approaches to socio-epidemiological dynamics”. 

\appendix
\section{Trubnikov formula for Maxwell molecules} \label{sec:appendixA}
We follow^^>\cite{Trubnikov1965} (Section 20, pages 200--203) to derive the Trubnikov formula for the relaxation of anisotropic initial temperatures in a system characterized by Maxwell molecules. We consider equation \eqref{eq:LFP} with $\gamma=0$ written in flux form
\[
\frac{\partial f(v,t)}{\partial t} = \nabla_v \cdot J(f)(v,t),
\]
with initial conditions given by
\[
f_0(v)=\rho\left(\frac{1}{2\pi}\right)^{3/2}\frac{1}{T_\perp\sqrt{T_z}} e^{-\dfrac{v^2_\perp}{2T_\perp}} e^{-\dfrac{v^2_z}{2T_z}},
\]
where $T_\perp=T_x=T_y > T_z$ and $v^2_\perp=v^2_x+v^2_y$. Since the temperature $T=(2T_\perp+T_z)/3$ is constant in time, we have
\[
\frac{d}{dt} T_\perp = - \frac{d}{dt} \frac{T_z}{2} = - \frac{1}{\rho} \int_{\R^3} \frac{v^2_z}{2} \frac{\partial f}{\partial t} dv = \frac{1}{\rho} \int_{\R^3} v_z J_z dv,
\]
where $J_z$ is the $z$-component of the flux
\[
J_z = - \frac{1}{8} \int_{\R^3} \sum_{j=\{x,y,z\}} (|q|^2 \delta_{zj} - q_z q_j)\left(f(v)\partial_{v_{*j}} f(v_*) - f(v_*) \partial_{v_{j}} f(v) \right) dv_*.
\] 
Using the expression of the initial distribution, we can compute explicitly the sum inside the integral, which gives
\be \label{eq:dtperp}
\frac{d}{dt} T_\perp = - \frac{d}{dt} \frac{T_z}{2} = - \frac{1}{8 \rho} \frac{T_\perp - T_z}{T_\perp T_z} \int_{\R^3} \int_{\R^3} f(v) f(v_*) v_zq_z q^2_\perp dv_* dv.
\ee
In the Coulombian case, we have a multiplicative $1/q^3$ term inside the double integral, which forces us to make the assumption that $|T_\perp - T_z| \ll 1$, (see^^>\cite{Trubnikov1965}, Section 20, page 202, for further details). In the Maxwell case, we can compute the double integral explicitly, using again the expression of the initial distribution. Through the change of variables $(v,v_*)\to(V,q)$, where $V=(v+v_*)/2$ and $q=v-v_*$, we have
\[
\int_{\R^3} \int_{\R^3} f(v) f(v_*) v_zq_z q^2_\perp dv_* dv = 4 \rho^2 T_\perp T_z. 
\] 
Plugging the previous expression into \eqref{eq:dtperp} we get
\[
\frac{d}{dt} T_\perp = - \frac{d}{dt} \frac{T_z}{2} = - \frac{\rho}{2} \Delta T,
\]
where $\Delta T= T_\perp - T_z$. Hence, if we observe that
\[
\frac{d}{d t} \Delta T = 3 \frac{d}{d t} T_\perp
\]
we finally get
\[
\Delta T (t) = \Delta T(0) e^{-t/\tau}, \quad \textrm{with} \quad \tau = \frac{2}{3 \rho}.
\]
We remark that this expression holds independently of the magnitude of $|T_\perp - T_z|$. Furthermore, the rate inside the exponential function does not depend on the temperature $T$ as in the Coulombian case.

\bibliographystyle{abbrv}
\bibliography{UQ_Plasmi.bib}

\end{document}